\documentclass[10pt,a4paper,reqno]{amsart}
\usepackage[inner=2cm,outer=2.5cm,
			top=3.5cm,bottom=4cm]{geometry} 
\usepackage{graphicx} % Required for inserting images
\usepackage[T1]{fontenc}
\usepackage[utf8]{inputenc}
\usepackage{amsfonts,amsmath,amsthm,amssymb}
\usepackage{hyperref}
\usepackage[capitalise]{cleveref}
\usepackage{bm}
\usepackage{enumitem}
\usepackage{varwidth}
\usepackage{booktabs,subfig} %package for figures
\usepackage[dvipsnames]{xcolor}
\usepackage{verbatim}
\usepackage[toc,page]{appendix}
\graphicspath{{Fig/}}
\usepackage{enumitem}

\numberwithin{equation}{section}

\newtheorem{theorem}{Theorem}[section]
\newtheorem{lemma}[theorem]{Lemma}
\newtheorem{remark}[theorem]{Remark}%
\newtheorem{corollary}{Corollary}[section]
\newtheorem{definition}[theorem]{Definition}

\numberwithin{equation}{section}

\newcommand{\ie}{\emph{i.e.}\ }

\def\half{\frac{1}{2}}
\newcommand{\bigo}{{\mathcal O}}

\newcommand\bfe{{\mathbf e}}
\newcommand\bfeta{{\mathbf \eta}}
\newcommand\bff{{\bm f}}
\newcommand\bfh{{\bm h}}

\newcommand\bfp{{\mathbf p}}
\newcommand\bfn{{\bm n}}
\newcommand\bfng{{\bm n_{\Gamma}}}
\newcommand\bfm{{\mathbf m}}
\newcommand\bfr{{\mathbf r}}

\newcommand\bfs{{\mathbf s}}
\newcommand\bfu{{\bm u}}
\newcommand\bfv{{\bm v}}
\newcommand\bfw{{\mathbf w}}

\newcommand\bfz{{\mathbf z}}

\newcommand\bfB{{\mathbf B}}

\newcommand\bfE{{\mathbf E}}
\newcommand\bfH{\bm{H}}

\newcommand\bfI{{\bm I}}
\newcommand\bfF{{\mathbf F}}
\newcommand\bfG{{\mathbf G}}

\newcommand\bfV{{\bm V}}
\newcommand\bfP{{\bm P}}

\newcommand{\eu}{\bfe_\bfu}
\newcommand{\eut}{(\bfe_\bfu)_{T}}

\newcommand{\ep}{\bfe_p}
\newcommand{\el}{\bfe_{\lambda}}
\newcommand\bfkappa{{\boldsymbol \kappa}}

\newcommand{\Ga}{\Gamma}
\newcommand{\Gah}{\Gamma_h}
\newcommand{\Galin}{\Gamma^{(1)}_h}
\newcommand{\ds}{d\sigma}
\newcommand{\dslin}{d\sigma_h^{(1)}}
\newcommand{\dsh}{d\sigma_h}
\newcommand{\nbg}{\nabla_{\Gamma}}
\newcommand{\nbgcov}{\nabla_{\Gamma}^{cov}}
\newcommand{\nbgcovh}{\nabla_{\Gamma_h}^{cov}}

\newcommand{\nbgh}{\nabla_{\Gamma_h}}
\newcommand{\nbglin}{\nabla_{\Gamma_h^{(1)}}}

\newcommand{\divg}{\textnormal{div}_{\Gamma}}
\newcommand{\divgh}{\textnormal{div}_{\Gah}}

\newcommand{\nb}{\nabla}

\newcommand{\R}{\mathbb{R}}

\def \to {\rightarrow}
\newcommand{\Th}{\mathcal{T}_h}
\newcommand{\Thlin}{\mathcal{T}_h^{(1)}}
\newcommand{\Ttilde}{\Tilde{T}} 
\newcommand{\Ihz}{\widetilde{I}_h^z}
\newcommand{\Ihzlin}{\widetilde{I}_{h,1}^z}

\newcommand{\Ihzl}{{I}_h^z}

\newcommand{\vhl}{\bm{v}_h^{\ell}}
\newcommand{\vh}{\bm{v}_h} % the velocity of the above
\newcommand{\vhtilde}{\Tilde{\bm{v}}_h}
\newcommand{\uh}{\bm{u}_h}

\newcommand{\uhl}{\bm{u}_h^{\ell}}
\newcommand{\wh}{\bfw_h}

\newcommand{\qh}{q_h}
\newcommand{\qhtilde}{\Tilde{q}_h}
\newcommand{\qhl}{q_h^\ell}
\newcommand{\ph}{p_h}
\newcommand{\lh}{\lambda_h}

\newcommand{\lhl}{\lambda_h^{\ell}}
\newcommand{\thbfu}{\theta_{\bfu}}
\newcommand{\thp}{\theta_{p}}
\newcommand{\thl}{\theta_{\lambda}}
\newcommand{\nh}{\bfn_h}
\newcommand{\nhl}{\bfn_h^{\ell}}
\newcommand{\nhtil}{\Tilde{\bfn}_h}

\newcommand{\muhlin}{\mu_h^{(1)}}
\newcommand{\muh}{\mu_h}
\newcommand{\mh}{\bfm_h}
\newcommand{\shlin}{\bfs_h^{(E)}}

\newcommand{\phl}{p_h^\ell}
\newcommand\bfPg{{\bm{P}_{\Ga}}}
\newcommand\bfPh{{\bm P_h}}

\newcommand\Bhg{{\bm B_h}}
\newcommand\Bhkg{{\bm B_h^{k_g}}}

\newcommand{\ah}{a_h}

\newcommand{\bhlin}{b_h^{(1)}}

\newcommand{\bhtil}{b_h^{L}}
\newcommand{\btil}{\Tilde{b}}
\newcommand{\bh}{b_h}
\newcommand\norm[1]{||#1||}

\begin{document}

\title{SFEM for the Lagrangian formulation of the surface Stokes problem}
\author{Charles M. Elliott, Achilleas Mavrakis}
\address{Mathematics Institute, Zeeman Building, University of Warwick, Coventry CV4 7AL, UK}
\email{\href{mailto:c.m.elliott@warwick.ac.uk}{c.m.elliott@warwick.ac.uk}, \href{mailto:Achilleas.Mavrakis@warwick.ac.uk}{Achilleas.Mavrakis@warwick.ac.uk}}
\date{}

\begin{abstract}
We consider the surface Stokes equation with Lagrange multiplier and approach it numerically. Using a Taylor-Hood surface finite element method, along with an appropriate estimate for the additional Lagrange multiplier, we derive a new inf-sup condition to help with the stability and convergence results. We establish optimal velocity convergence both in energy and tangential $L^2$ norms, along with optimal $L^2$ norm convergence for the two pressures, in the case of super-parametric finite elements. Furthermore, if the approximation order of the velocities matches that of the extra Lagrange multiplier, we achieve optimal order convergence even in the standard iso-parametric case. In this case, we also establish some new estimates for the normal $L^2$ velocity norm. In addition, we provide numerical simulations that confirm the established error bounds and also perform a comparative analysis against the penalty approach.
\end{abstract}

\maketitle

\section{Introduction}
For a stationary two dimensional surface $\Ga$ the  \emph{(steady) generalized  surface Stokes} problem is to: Find the velocity field $\bfu  : \Ga \to \R^3$, surface pressure $p  : \Ga \to \R$ with $\int_\Ga p ds =0$ and Lagrange multiplier $\lambda : \Ga \to \R$ that solves on $\Gamma$
\begin{align}
		\begin{cases}
			\label{eq: generalized Lagrange surface stokes}
			- 2\mu\divg(E(\bfu)) + \bfu + \lambda \bfng+ \nbg p = \bff, \\
			\divg \bfu =g, \\
            \bfu \cdot \bfng =0,
		\end{cases}
	\end{align}
where $\mu$ is the viscosity coefficient, $\bff$ is a given   force field $
	\bff \in (L^2(\Gamma))^3$,
$\lambda$ is a normal force  field to be determined and $g$ is a given   source term $g \in L^2(\Gamma)$, with $\int_\Ga g ds =0$. Essentially we view $\lambda$ as a Lagrange multiplier for the constraint that the normal component of $\bfu$ is zero.  Here   the   rate-of-strain (deformation) tensor $E(\cdot)$, \cite{Miura2020}, is defined for an arbitrary vector field $\bfv$, by 	\begin{equation}
		\begin{aligned}\label{eq: Deformation tensor}
			E(\bfv)  = \half (\nbgcov \bfv + \nbg^{cov,t}\bfv).
		\end{aligned}
	\end{equation}
where  $ \nbgcov \bfv := \bfPg \nbg \bfv $ (with adjoint $\nbg^{cov,t}\bfv=(\nbg \bfv)^t\bfPg$) and  $\bfPg$ denotes the projection to the tangent space of $\Gamma$.
    
 Writing $\bfu\equiv \bfu_T+\bfu\cdot \bfng \bfng$  and $\bfu_T:=\bfPg \bfu$,
  this may also be written as: Find a tangential velocity $\bfu_T : \Ga \to \R^3$, with $\bfu_T \cdot \bfn_\Gamma =0$ and surface pressure $p : \Ga \to \R$ with $\int_\Ga p ds =0$ that solves on $\Gamma$
	\begin{align}
		\begin{cases}
			\label{eq: generalized tangential surface stokes}
			- 2\mu\bfPg\divg(E(\bfu_T)) + \nbg p + \bfu_T= \bfPg \bff, \\
			\divg \bfu_T =g.
		\end{cases}
	\end{align}

The well-posedness of both surface Stokes systems has been studied in \cite{jankuhn2018incompressible, lengeler2015stokestype, Simonett_2022, Miura2020}, where the authors do not consider the zero-order term, but instead filter out the \emph{Killing fields}, that is, the non-zero tangential vectors $\bfv_K$ satisfying $E(\bfv_K)=0$. In our case, we follow \cite{jankuhn2021Higherror,olshanskii2021inf,Olshanskii2018,reusken2024analysis,demlow2024tangential,hardering2023parametric} and add a zero-order term $\bfu$ to ensure well-posedness (uniqueness) and avoid the resulting technical difficulties in the numerical analysis. \cite{bonito2020divergence} have studied the effect this zero-order term has on the solution and also the error bounds. % have studied the effects of Killing fields in solutions, where they prove a mechanism of filtering them out.

In this work, we consider the discretization of a variational form of the problem \eqref{eq: generalized tangential surface stokes} using $H^1$-conforming \emph{Taylor-Hood} surface finite elements. Surface finite elements methods (SFEM) have been studied extensively in \cite{Dziuk88,Dziuk90,DziukElliott_acta,DziukElliott_SFEM} for elliptic and parabolic problems on a surface $\Ga$, including their evolving surface counterpart \cite{EllRan21,DziukElliott_ESFEM,DziukElliott_L2,highorderESFEM}. There is also work for the vector- or tensor-valued corresponding problem using SFEM \cite{hansbo2020analysis,Hardering2022}.
% Another finite element method used to study such problems is TraceFem method  \cite{jankuhn2021trace}.

The tangential Stokes problem has been studied using TraceFem in \cite{jankuhn2021Higherror,olshanskii2021inf,Olshanskii2018}, and more recently \cite{bonito2020divergence,reusken2024analysis,demlow2024tangential,hardering2023parametric,brandner2022finite} with the help of SFEM. There have also been discretizations of the stream function formulation using SFEM in \cite{BranReusSteam2020,reusken2018stream,Nitschke_Voigt_Wensch_2012}. The issue with analyzing a discretization of \eqref{eq: generalized tangential surface stokes} is that trying to enforce the tangential constraint in the discrete space would not lead to $H^1$-conforming finite elements, due to the discontinuity across the edges of the mesh. Some papers enforce the tangentiality condition exactly, see \cite{bonito2020divergence} where an $H(\divg)$ conforming finite element method with an interior penalty approach to enforce $H^1-$continuity quickly is analyzed, and \cite{demlow2024tangential} again using a $H(\divg)$ conforming finite element method where no penalization is needed due to the use of the Piola transform. There are also such methods, using the Piola transform to approximate the tangential condition exactly for the Darcy problem \cite{Ferroni2016} as well. These methods, though, more often than not are more difficult to implement numerically. 

An easier approach is to use $H^1-$conforming  $\emph{Taylor-Hood}$ surface finite element methods, where one enforces the tangential constrain weakly via penalization or Lagrange multiplier.  Several papers have studied the penalization approach using TraceFEM  Stokes problem  \cite{jankuhn2021Higherror,olshanskii2021inf,Olshanskii2018} and SFEM \cite{hardering2023parametric,reusken2024analysis}. However, there appears to be none employing the Lagrange multiplier method, except for \cite{jankuhn2021trace}, where the authors use Tracefem to approximate the vector-Laplace equation.

 In \cite{hardering2023parametric} 
 using a penalty approach for the tangentiality condition, with penalty scaling $\eta = h^{-1}$, they are, with the help of the macro-element technique, able to prove optimal tangential $L^2$-norm velocity error bounds. With this technique, it is possible to prove well-posedness for different known pairs of stable Stokes elements from the flat case, for example, the \emph{Taylor-Hood} MINI element. In comparison, in \cite{reusken2024analysis} the author enforces the tangential condition again via a penalty with a penalty parameter $\eta = h^{-2}$, and achieves optimal $H^1$ and $L^2$ error bounds for the velocity field and pressure respectively, in the case where a higher-order approximation of normal $\nhtil$ is considered. Both authors have to approximate the strain tensor $E(\bfPg \bfv)$, for which the mean curvature tensor was evaluated as the gradient derivative of the geometry normal.

In our work, as mentioned above, we use $H^1-$conforming  $\bm{\mathcal{P}}_{k_u}/\mathcal{P}_{k_{pr}}/\mathcal{P}_{k_{\lambda}}$ extended \emph{Taylor-Hood} finite elements, where we enforce the tangential constraint weakly via an extra Lagrange multiplier $\lh$, to approximate \eqref{eq: generalized tangential surface stokes}. We consider and treat this new Lagrange multiplier as an extra pressure and, therefore, show a combined discrete \emph{inf-sup condition} involving both $\{\ph,\lh\}$. In particular, we establish two different discrete \emph{inf-sup conditions} depending on the choice of norm for $\lh$. One concerning the $L^2 \times L^2$-norm of the pressures with respect to the energy norm of the velocity, and another involving the $L^2\times H_h^{-1}$-norm with respect to the $H^1$-norm of the velocity instead, where $H_h^{-1}$ is the dual of the finite element space for $\lh$.
For the error bounds, we consider two different scenarios depending on the choice of approximation of the extra Lagrange multiplier, i.e. $k_\lambda$. First, in the case where $k_\lambda=k_u-1$ we show optimal order discrete error estimates in energy and tangential $L^2$-norms for the velocity and $L^2\times L^2$-norm for the pressures, only under sufficiently accurate geometry approximation; e.g. we have $\bigo(h^{k_g-1})$ convergence for the geometric part of the error in the energy norm. Therefore, we use \emph{super-parametric finite elements}, where the order of the geometric approximation of the discrete surface is one higher than the order of the velocity approximation, $k_u = k_g +1$ and $k_g \geq 2$. On the other hand, for the $k_\lambda=k_u$ case, 
we establish optimal error bounds with respect to the geometric part of the error even with \emph{iso-parametric finite elements}; e.g. we show $\bigo(h^{k_g})$ convergence with respect to the geometric approximation in the energy norm. The reason for the improvement in the geometric error is a new improved $H^1$ coercivity result for weakly discretely tangential divergence-free functions \cref{lemma: new H1 estimate thbfu kl=ku} (this result is obtained by better controlling the normal part of such functions due to the richer F.E. space for $\lh$) and the $L^2\times H_h^{-1}$ \emph{discrete inf-sup condition}. In this case, we also establish new error estimates for the $L^2$-norm of the normal part of the velocity, which are half-order sub-optimal compared to the tangential $L^2$-norm, unless an improved normal $\nhtil$ is used; see \cref{remark: About the L^2 Error of the normal velocity}. Moreover, these new improved estimates hold even for planar finite elements, i.e. for $k_g=1$, which was not the case for $k_\lambda=k_u-1$. We also want to notice that these results are important for the study of the \emph{time-dependent Navier-Stokes} on stationary or evolving surfaces. Regarding the numerical implementation of our method, we notice that it is simple and straightforward. Experiments are provided in \cref{Section: Numerical results}, where we test our theoretical results. Lastly, we notice that the choice of $k_\lambda$ plays an important role in the numerical implementation, since, despite the good results we observe for $k_\lambda=k_u$, this choice leads to a worse conditioned system compared to the other case.

% The main difference between the two choices of $k_\lambda$, as we shall see, boils down to the approximation of the normal part of a weakly discretely tangential divergence-free function i.e. our velocity, which is the reason for the improvement in the geometric error. 

% However, the choice of $k_\lambda$ plays an important role in the numerical implementation, since 

\subsection{Outline}
In \cref{sec: Differential geometry on Surfaces} we introduce our notation along with our function spaces. The continuous variational formulation is introduced in \cref{sec: Weak formulation}, which helps us to formulate our finite element scheme in \cref{section: Finite element approximation}. There, the $H^1-$ conforming $\bm{\mathcal{P}}_{k_u}/\mathcal{P}_{k_{pr}}/\mathcal{P}_{k_{\lambda}}$ extended \emph{Taylor-Hood} finite elements are also presented. In \cref{sec: well-posedness of discrete formulation} we prove two different discrete \emph{inf-sup conditions}, one involving $L^2\times L^2$ and the other the $L^2 \times H_h^{-1}$ norms of the two pressures, from which we establish well-posedness of our discrete formulation. In \cref{Section: error estimates} we analyze error bounds for two different choices of $k_\lambda$, i.e. the finite element approximation space $\Lambda_h$ of the second pressure $\lambda$. For $k_\lambda=k_u-1$, we show energy and tangential $L^2$ error bounds for the velocity along with $L^2\times L^2$ pressure error bounds, which are optimal only if \emph{super-parametric finite elements} are used.  In contrast, for $k_\lambda = k_u$, we derive a new $H^1$ coercivity result for weakly discrete tangential divergence-free functions, and using the $L^2 \times H_h^{-1}$ discrete \emph{inf-sup condition}, we establish optimal error bounds in the \emph{iso-parametric} finite element setting, including in the $H^1-$norm. In this case, we also prove $L^2$ norm error bounds for the normal part of the velocity. In \cref{Section: Numerical results} we demonstrate numerical results and perform a comparative analysis against the penalty approach.
% For $k_\lambda = k_u$, we prove a new $H^1$ coercivity result for weakly discrete tangential divergence-free functions and using the $L^2 \times H^{-1}$ discrete \emph{inf-sup condition}, we prove optimal error bounds in the \emph{iso-parametric} finite element case, by improving the error component that involves the approximation of geometry. In this case, we also prove $L^2$ norm error bounds for the normal part of the velocity,
% which are only suboptimal with respect to the geometric approximation.
% In \cref{Section: Numerical results} the numerical results are demonstrated and perform a comparative analysis against the penalty approac
\section{Differential geometry on surfaces}\label{sec: Differential geometry on Surfaces}
In this section, we recall some fundamental notions and tools concerning surface vector-PDEs necessary for the study of the surface Stokes problem. We mainly follow definitions in \cite{hansbo2020analysis} and \cite{jankuhn2018incompressible}. 

\subsection{The closed smooth surface}
Let $\Ga \subset \mathbb{R}^3$ be a closed, oriented, compact $C^m$  two-dimensional hypersurface embedded in $\mathbb{R}^3$. We will see later that $m\geq 4$ is required in the error analysis. Note that  since  $\Gamma$ is the boundary of an open set we choose the orientation by setting  $\bfng$ to be the unit outward pointing normal to $\Ga$.  Let $d(\cdot) : \mathbb{R}^3 \to \mathbb{R}$ be the signed distance function, defined in \cite{DziukElliott_acta}, and
for $\delta>0 $, let $U_\delta \subset \mathbb{R}^3$ be the tubular neighborhood $U_{\delta} = \{x \in \R^3\ : |d(x)| <\delta\}$. Then we may define the closest point projection mapping to $\Ga$ as $\pi(x) = x - d(x)\bfng(\pi(x)) \in \Ga$ for each $x \in U_\delta$  for $\delta>0$ small enough, see \cite{GilTrud98,DziukElliott_acta}. Moreover, $\nb d(x) = \bfng(\pi(x))$ for $x \in U_\delta$. We can see that $d(\cdot)$ and $\pi(\cdot)$ are of class $C^m$ and $C^{m-1}$ on $\overline{U_{\delta}}$. 

The extended normal allows us to define the extended Weingarten map on the surface and in $U_{\delta}$, by $\bfH := \nb \bfn^e = \nb^2 d$, which is of class $C^{m-2}$. The mean curvature of $\Ga$ is $\bfkappa:= tr(\bfH)$. We also define the orthogonal projection operator onto the tangent plane $\bfPg (x) = \bfI - \bfng(x) \otimes \bfng(x)$ for $x \in \Ga$ that satisfies $\bfPg^T=\bfPg^2=\bfPg$, $|\bfPg|_{Fr}=2$, and $\bfPg \cdot \bfng=0$.

Using the projection $\pi$ we may extend functions $\bfv : \Ga \to \mathbb{R}^3$  on $\Ga$ to the tubular neighbourhood $U$, by 
\begin{equation}
\label{eq: smooth extension}
    \bfv^e(x) = \bfv(\pi(x)), \ \ \ x \in U_{\delta}.
\end{equation}
This extension is constant in the normal direction of $\Ga$ and thus contains certain useful additional properties in comparison to other regular extensions. The extension of $\bfPg$ to the $\delta-$strip $U_\delta$, is defined as 
\begin{equation}
    \bfP = \bfI - \bfng^e(x) \otimes \bfng^e(x) = \bfI - \bfn(x) \otimes \bfn(x), \ \text{ for } x \in U_\delta,
\end{equation}
where we have used $\bfng^e = \bfng\circ \pi = \bfn$. We can then easily see that $\bfP|_{\Ga} = \bfPg$.

\subsection{Scalar functions and vector fields}

We start by defining useful quantities for \emph{scalar} functions following \cite{DziukElliott_acta}. For a function $f \in C^1(\Ga)$ we define the \emph{tangential (surface) derivative} of $f$ as 
\begin{equation}
    \nbg f(x) = \bfPg(x) \nb f^e(x), \ \ \text{with } \ \underline{D}_if(x) = \sum_{j=1}^3 P_{ij}\partial_jf^e(x), \ \ \ x\in\Ga, 
\end{equation}
where $(\nb f^e)_i = \partial_i(f^e)$, $1 \leq i \leq 3$ is the Euclidean derivative and $\nbg f(x) = (\underline{D}_1f(x),$ $ \underline{D}_2f(x),$ $\underline{D}_3f(x))^t$ is a column vector. Here $f^e$ is the extension to $U_{\delta}$ where $f^e|_{\Ga} = f$ as in \eqref{eq: smooth extension}. Now, since $\bfPg$ is the orthogonal projection, which is symmetric, we have that 
\begin{equation}
    \nbg f(x) = \bfPg(x) \nb f^e(x) =  \nb^t f^e(x)\bfPg(x), \ \ \ \ \bfPg \nbg f = \nbg f  \ \ \text{on } \Ga.
\end{equation}
As noted in \cite[Lemma 2.4]{DziukElliott_acta} the definition of the tangential gradient is independent of the extension. We now introduce derivatives of surface \emph{vector fields}. We begin by defining tangential derivatives, analogous to the scalar case, as Euclidean derivatives of extended quantities (this time-vector fields) and proceed to establish the covariant derivatives. 

\begin{definition}[Tangential derivative]
Let $\bfv : \Ga \to \mathbb{R}^3$, $\bfv = (v_1, v_2, v_3)^T$ be a smooth vector field, and $\bfv^e$ the smooth extension in \eqref{eq: smooth extension}, the tangential gradient of this vector field  is defined by 
\begin{equation}\label{eq: tangential derivative}
    (\nbg \bfv)_{ij} = (\nb \bfv^e \bfPg)_{ij} = \sum_{l=1}^3 \partial_l v^e_iP_{jl}  ,
    %\sum_{j=1}^3 \partial_l v_iP_{lj}
\end{equation}
where $(\nbg \bfv)_{ij} = \underline{D}_j v_i$ and $\underline{D}_j v = (\bfp_j \cdot \nb)v^e = (\nb v^e)\bfp_j$. This derivative is a $3 \times 3$ matrix, which may be presented as 
\begin{equation}\label{eq: tangential derivative matrix}
\nbg \bfv = \nbg \begin{bmatrix}
v_1 \\
v_2 \\
v_3
\end{bmatrix} = \begin{bmatrix}
\underline{D}_1 v_1 & \underline{D}_2 v_1 & \underline{D}_3 v_1\\
\underline{D}_1 v_2 & \underline{D}_2 v_2 & \underline{D}_3 v_2\\
\underline{D}_1 v_3 & \underline{D}_2 v_3 & \underline{D}_3 v_3
\end{bmatrix}.
\end{equation}

\end{definition}
We note again that due to the extension $\bfv^e = \bfv \circ \pi$ we find that the tangential derivative turns out to be the Euclidean derivative corresponding to this extended vector field \ie $\nbg \bfv = \nb \bfv^e|_{\Ga}$. With the help of the tangential derivatives for vector fields we may define the Weingarten map $\bfH$ on the $\delta-$strip $U_{\delta}$ again
\begin{equation}\label{eq: Weingarten map}
    \bfH := \nbg \bfn, \ \text{ with } \ \bfH \bfn=0, \ \ \ \bfH \bfPg = \bfPg \bfH =\bfH,
\end{equation}
where we observe that it is tangent to $\Ga$. Note that $\bfH$ is of class $C^{m-2}$ and thus bounded on $\Ga$, see \cite{DziukElliott_acta, GilTrud98}. 

We continue by providing an extended definition of covariant derivatives for general vector fields in $\mathbb{R}^3$, as seen in \cite{brandner2022derivations} and \cite{fries2018higher}.
\begin{definition}[Covariant derivatives]\label{def: Covariant derivatives}
Let $\bfv : \Ga \to \mathbb{R}^3$, $\bfv = (v_1, v_2, v_3)^T$ be a smooth vector field, and $\bfv^e$ the smooth extension in \eqref{eq: smooth extension}, the covariant derivative $\nbgcov \bfv$ is then defined by
\begin{equation}
    \nbgcov \bfv = \bfPg \nbg \bfv = \bfPg \nb \bfv^e|_{\Ga} \bfPg = \sum_{j=1}^3 P_{jk}\partial_k v^e_lP_{li},
    %\sum_{j=1}^3 P_{jk}\partial_k v_lP_{li}
\end{equation}
where the derivative is a  $3 \times 3$ matrix.
\end{definition}

\noindent Notice that the covariant derivative, in comparison to the tangential derivative of a smooth vector field defined in \eqref{eq: tangential derivative}, is a tangential tensor field, i.e. matrix.
Now, for a smooth vector-valued function $\bfv : \Ga \to \mathbb{R}^3$, the surface divergence is given by 
\begin{equation}\label{eq: divg vector definition}
\divg(\bfv) = tr(\nbgcov \bfv) = tr(\bfPg \nabla \bfv^e \bfPg) = tr(\bfPg \nabla \bfv^e) = tr(\nabla \bfv^e \bfPg) =  tr(\nbg \bfv ) = tr(\nabla\bfv^e),    
\end{equation}
and thus we can easily see that 
\begin{equation}\label{eq: Bound of divg}
\norm{\divg\bfv}_{L^2(\Ga)} \leq 3\norm{\nbg\bfv}_{L^2(\Ga)}. 
\end{equation}
We may split the arbitrary smooth vector field $\bfv$ into $\bfv = \bfv_T + v_n \bfng$ with $\bfv_T = \bfPg \bfv$ and $v_n = \bfv \cdot \bfng$ such that $\nbgcov \bfv = \nbgcov \bfv_T +\nbgcov (v_n \bfng)$, and get the following useful formulae 
\begin{equation}
    \begin{aligned}\label{eq: split cov}
       \nbgcov \bfv = \nbgcov \bfv_T + \bfH v_n.
       \end{aligned}
\end{equation}

\noindent We also define with the help of \eqref{eq: divg vector definition}  the surface divergence of a tensor function $\bfF : \Ga \to \mathbb{R}^{3 \times 3}$:
\begin{equation}
    \begin{aligned}
       \divg(\bfF) :=  \begin{bmatrix}
\divg(\bfF_{1,j}) \\
\divg(\bfF_{2,j}) \\
\divg(\bfF_{3,j})
\end{bmatrix}, \quad j=1,2,3.
       \end{aligned}
\end{equation}

\subsection{Function spaces}\label{sec: Function spaces}

We now define function spaces on the surface $\Ga$. By $(\cdot,\cdot)$ we denote the usual $L^2(\Ga)$ inner product with norm $\norm{\cdot}_{L^2(\Ga)} = (\cdot,\cdot)^{1/2}$ and by $\norm{\cdot}_{L^\infty(\Ga)}$ we denote the $L^{\infty}(\Ga)$ norm. By $<\cdot,\cdot>_{H^{-1},H^1}$ we denote the duality pairing. Also $(H^k(\Ga))^n$ denotes the standard Sobolev spaces as presented in \cite{DziukElliott_acta} of scalar $(n=1)$ or vector valued functions $(n=3)$ for  $\bfeta : \Ga \to \mathbb{R}^n$ with the corresponding norm being given by
\begin{equation}
    \begin{aligned}
    \norm{\bfeta}_{H^k(\Ga)}^2 = \sum_{j=0}^k \norm{\nbg^j \bfeta}^2_{L^2(\Ga)},
    \end{aligned}
\end{equation}
where $\nbg^j$ denote all the weak tangential derivative of order $j$, i.e. $\nbg^j \bfeta = \underbrace{\nbg \cdot \cdot \cdot \nbg}_{\text{ j times}} \bfeta$. For $k=0$ we get the usual $L^2-$ norm $\norm{\cdot}_{L^2(\Ga)}$ for scalar or vector-valued functions. Now more specifically for a vector fields $\bfv : \Ga \to \mathbb{R}^3$ the space $\bfH^1(\Ga) = (H^1(\Ga))^3$ is equipped with the following norm
\begin{equation}
    \begin{aligned} \label{eq: H1 norm definition}
    \norm{\bfv}_{H^1(\Ga)}^2 = \norm{\bfv}_{L^2(\Ga)}^2 + \norm{\nbg \bfv}_{L^2(\Ga)}^2.
    \end{aligned}
\end{equation}

\noindent We set 
$$L^2_0(\Ga) :=\{\eta \in L^2(\Ga) | \ \int_\Ga \eta \, \ds =0\},$$
equipped with the standard $L^2$-norm and the subspace of tangential vector fields: 
$$\bfH^1_T := \{\bfu \in H^1(\Ga)^3 \: | \: \bfu \cdot \bfng =0\},$$
endowed with the $H^1$-norm \eqref{eq: H1 norm definition}.
For any vector field $\bfv$ we use the decomposition $\bfv:=\bfv_T+v_n\bfn_\Gamma$, where $\bfPg \bfv = \bfv_T$
and $v_n:=\bfv\cdot \bfn_\Gamma$. We may also define the space \begin{equation}\label{eq: V* space}
    \bfV := \{\bfv \in L^2(\Ga)^n : \bfv_T \in \bfH^1_T,\, v_n \in L^2(\Ga)\}, \ \ \text{ with } \ \ \norm{\bfv}_{ \bfV } :=\norm{\bfv_T}_{H^1(\Ga)} + \norm{v_n}_{L^2({\Ga})},
\end{equation}
	where we can interpret $ \bfV  \sim \bfH^1_T \oplus L^2(\Ga)$.

\noindent The following inequality known as \emph{\bf surface Korn's inequality} was established in \cite{jankuhn2018incompressible}. There  exists a constant $c_K>0$  such that
	\begin{equation}
		\label{eq: Korn inequality}
		 c_K \norm{\bfv_T}_{H^1(\Ga)}\le \norm{E(\bfv_T)}_{L^2(\Ga)} + \norm{\bfv_T}_{L^2(\Ga)}~~\text{ for all } \bfv\in \bfH^1_T.
	\end{equation}

\section{Variational formulation}\label{sec: Weak formulation}

\subsection{Bilinear forms}
In order to set up the variational formulations we define  bilinear forms.  

 \begin{definition}
      For $\bfu, \bfv \in  \bfH^1(\Ga)$  and $\{q,\xi\} \in L^2(\Ga)\times L^2(\Ga)$ we define 
      \begin{align}
          \label{eq: regular bilinear a}
         a(\bfu,\bfv) &:= \int_\Ga E(\bfu)  : E(\bfv) + \bfu\cdot\bfv \, \ds.\\
         \label{eq: lagrange bilinear forms continuous}
         b^L(\bfu,\{q,\xi\}) &:= -\int_\Ga q \ \divg \bfu_T  \, \ds + \int_\Ga \xi u_n\, \ds := b(\bfu,q) + \int_\Ga \xi u_n\, \ds.
      \end{align}
    In the case where $\bfu \in  \bfH^1(\Ga)$  and $\{q,\xi\} \in  L^2(\Ga) \times H^{-1}(\Ga)$ we define
     \begin{equation}
     \begin{aligned}
         \btil(\bfu,\{q,\xi\}) := b(\bfu,q) + <\xi,\, u_n>_{H^{-1},H^1},
     \end{aligned}
 \end{equation}
 where by $<\cdot,\cdot>_{H^{-1},H^1}$ we mean the duality pairing.
 \end{definition}

\subsection{Lagrange multiplier variational formulations}\label{sec: Lagrange formulation}

Consider \eqref{eq: generalized Lagrange surface stokes} with $g=0$. By  applying integration by parts, we want to look for a solution for the following problem,\\

\noindent {\bf (LP)}: Given $\bff \in L^2(\Ga)^3$
 determine $(\bfu,\{p,\lambda\}) \in \bfH^1(\Ga) \times(L^2_0(\Ga)\times L^2(\Ga))$ such that,
\begin{align}
\begin{cases}
\tag{LP}
    \label{weak lagrange hom}
        a(\bfu,\bfv) \ + \!\!\!\!&b^L(\bfv,\{p,\lambda\}) = (\bff,\bfv)_{L^2(\Ga)} \ \ \ \ \ \text{for all } \bfv\in \bfH^1(\Ga),\\
        &b^L(\bfu,\{q,\xi\})=0 \ \ \  \text{ for all } \{q,\xi\}\in (L^2_0(\Ga)\times L^2(\Ga)),
    \end{cases}
\end{align}

\begin{remark}
   The surface Stokes problem \eqref{eq: generalized Lagrange surface stokes} with $g\neq0$ can be  easily transformed  to \eqref{weak lagrange hom} with a different right-hand side  $L^2-$ functional via simple use of a Leray-Helmholtz decomposition for surfaces, \cite[Theorem 4.2]{reusken2018stream}. \end{remark}

In order to prove  well-posedness  of the problem above, we consider the following auxilary problem,

\noindent {\bf (LP-aux)}: Given $\bff \in L^2(\Ga)^3$
 determine $(\bfu,\{p,\lambda\}) \in \bfH^1(\Ga) \times(L^2_0(\Ga)\times H^{-1}(\Ga))$ such that,
\begin{align}
\begin{cases}
\tag{LP-aux}
    \label{weak lagrange hom 2}
        a(\bfu,\bfv) \ + \!\!\!\!&\btil(\bfv,\{p,\lambda\}) = (\bff,\bfv)_{L^2(\Ga)} \ \ \ \ \ \text{for all } \bfv\in \bfH^1(\Ga),\\
        &\btil(\bfu,\{q,\xi\})=0 \ \ \  \text{ for all } \{q,\xi\}\in (L^2_0(\Ga)\times H^{-1}(\Ga)).
    \end{cases}
\end{align}

\begin{lemma}\label{lemma: inf-sup cont lagrange}
   Let $\Ga$ be  $C^2$ and compact. Then there exists a constant $c_b$ such that the inf-sup property holds,
\begin{equation}
    \begin{aligned}\label{infsup lagrange}
         \inf_{\{p,\lambda\} \in (L^2_0(\Ga)\times H^{-1}(\Ga))} \ \sup_{\bfv\in \bfH^1(\Ga)} \frac{\btil(\bfv,\{p,\lambda\})}{\norm{\bfv}_{H^1(\Ga)} \norm{\{p,\lambda \}}_{L^2(\Ga)\times H^{-1}(\Ga)}} \geq c_b > 0,
    \end{aligned}
\end{equation}
\end{lemma}
\begin{proof}
\ For any $p\in L^2_0(\Ga)$ and  $\lambda \in H^{-1}(\Ga)$  we may set $\bfv = \bfv_T + v_n\bfng$, where,  $\bfv_T \in \bfH_T^1$ may be chosen  such that
     $$b(\bfv_T,p) =\norm{p}^2_{L^2(\Ga)}, \ \ \ \norm{\bfv_T}_{H^1(\Ga)} \leq c \norm{p}_{L^2(\Ga)},$$
     due to \cite[Lemma 4.2]{jankuhn2018incompressible}, and  we may set  $v_n = \mathcal{R}_{H^1} \lambda$, where $\mathcal{R}_{H^1} \lambda$ is the Riesz representation of $\lambda$.
    Thus the following holds
    \begin{equation}
        \begin{aligned}
            \norm{\bfv}_{H^1(\Ga)} = \norm{\bfv_T}_{H^1(\Ga)} + \norm{v_n \bfng}_{H^1(\Ga)} &\leq  c (\norm{p}_{L^2(\Ga)} + \norm{\mathcal{R}_{H^1} \lambda}_{H^1(\Ga)}) \\
            & \leq c(\norm{p}_{L^2(\Ga)} + \norm{\lambda}_{H^{-1}(\Ga)}),
        \end{aligned}
    \end{equation}
and thus
\begin{equation*}
        \begin{aligned}
            \btil(\bfv,\{p,\lambda\}) = b(\bfv_T,p) + \norm{\lambda}_{H^{-1}(\Ga)}^2 &= \norm{p}_{L^2(\Ga)}^2+\norm{\lambda}_{H^{-1}(\Ga)}^2\\
            &\geq c (\norm{p}_{L^2(\Ga)}^2 + \norm{\lambda}_{H^{-1}(\Ga)}^2)^{1/2}\norm{\bfv}_{H^1(\Ga)},
            \end{aligned}
    \end{equation*}
which concludes the proof.
\end{proof}

\begin{lemma}\label{lemma: well-posed lagrange hom}
Assuming (\ref{infsup lagrange}) holds, then 
the problem \eqref{weak lagrange hom} is well-posed with a unique solution $(\bfu, \{p,\lambda\})$ satisfying $\bfu\cdot \bfn =0$, 
and the a-priori stability estimates
\begin{equation}\label{eq: Stability estimate lagrange hom 2}
    \norm{\bfu}_{H^1(\Ga)} + \norm{\{p,\lambda\}}_{L^2(\Ga)\times L^2(\Ga)}\leq c\norm{\bff}_{L^2(\Ga)}.
\end{equation}
\end{lemma}
\begin{proof}
   Let us start with the problem \eqref{weak lagrange hom 2}.
    It is then clear that both bilinear forms $a(\cdot,\cdot)$ and $\btil(\cdot,\{\cdot,\cdot\})$ are bounded in $\bfH^1(\Ga)\times\bfH^1(\Ga)$ and $\bfH^1(\Ga) \times L^2(\Ga) \times H^{-1}(\Ga)$.
Since
\begin{equation}\label{eq: kernel b lagrange}
\begin{aligned}
    \mathcal{K} = \{\bfv \in \bfH^1(\Ga) \, : \, \btil(\bfv,\{p,\lambda\}) = 0  \text{ for all } \{p,\lambda\}&\in  L^2_0(\Ga)\times H^{-1}(\Ga)\} \\
    &\subset \{\bfv \in \bfH^1(\Ga) \, : \, \bfv \cdot \bfng =0\},
    \end{aligned}
\end{equation}
    it follows from  \eqref{eq: Korn inequality}  that
\begin{equation}
a(\bfv,\bfv) = a(\bfv_T,\bfv_T) \geq c_K\norm{\bfv_T}_{H^1(\Ga)}^2 \text{ for all } \bfv \in \mathcal{K}.
\end{equation}
 Coupled with the inf-sup condition \Cref{lemma: inf-sup cont lagrange},  well-posedness  and boundedness of $(\bfu,\{p,\lambda\})$  in $\bfH^1(\Ga) \times L^2(\Ga) \times H^{-1}(\Ga)$ readily follows by the abstract theory in  \cite{brezzi2012mixed}. Testing now with $q=0$ and any $\xi \in H^{-1}(\Ga)$  in the second equation of \eqref{weak lagrange hom 2} we get that our solution is tangent, i.e. $\bfu \cdot \bfng =0$.
From \eqref{weak lagrange hom 2} it follows that 
\begin{equation}
    <\lambda,\bfv\cdot \bfng>_{H^{-1},H^{1}} = -a(\bfu,\bfv) - b(\bfv_T,p) + f(\bfv) \text{ for all } \bfv \in \bfH^1(\Ga).
\end{equation}
Considering $\bfv_T =0$, then we have that 
\begin{equation}
\begin{aligned}
    <\lambda,\bfv\cdot \bfng>_{H^{-1}(\Ga),H^{1}(\Ga)} &= -a(\bfu,\bfv_n) + (\bff,\bfv_n)_{L^2(\Ga)}\\
    &= -\int_{\Ga} E(\bfu):v_n\bfH \, \ds + (f,\bfv_n)_{L^2(\Ga)}\\
    &\leq c(\norm{\bfu}_{H^1(\Ga)} + \norm{\bff}_{L^2(\Ga)})\norm{\bfv\cdot \bfng}_{L^2(\Ga)},
    \end{aligned}
\end{equation}
and since $\bfH^1(\Ga)$ is dense in $L^2(\Ga)$ we have that the solution of \eqref{weak lagrange hom 2} satisfies \eqref{eq: Stability estimate lagrange hom 2} and also solves \eqref{weak lagrange hom},
which completes the proof. \end{proof}

\section{Surface finite elements}\label{section: Finite element approximation}

\subsection{Triangulated surfaces}\label{sec: Triangulated surfaces}
\subsubsection{Piecewise linear triangulation $\Galin$}
To approximate the continuous surface $\Ga$ we start by constructing a polyhedral approximation $\Galin$ contained in the tubular neighborhood $U_\delta$ i.e. $\Galin \subset U_\delta$, whose vertices lie on $\Ga$ and which  is equipped with an admissible and conforming subdivision of triangulation $\Thlin$, see \cite[Section 6.2]{EllRan21}, so that
\begin{equation*}
    \Galin = \cup_{\Ttilde \in \Thlin}\Ttilde.
\end{equation*}
We denote by  $h_{\Ttilde}$ the diameter of a simplex $\Ttilde \in \Thlin$ and set $h := \text{max} \{ \text{diam}(\Ttilde) \, : \, \Ttilde \in \Thlin\}$. We assume that the triangulation is \emph{shape-regular}, hence the following inequality holds  
\begin{equation}\label{eq: shape regularity}
    c_1 h_{\Ttilde_i} \leq h_{\Ttilde} \leq  c_2 h_{\Ttilde_i}, \quad \text{for all } \Ttilde_i \in \omega_{\Ttilde}.
\end{equation}
for fixed constants $c_1, \,c_2$ \cite{camacho20152}, where $\omega_{\Ttilde}$ denotes the neighbouring elements to $\Ttilde$ assumed to  bounded by a fixed constant.
We further assume that the family of triangulations is also \emph{quasi-uniform} \cite[Definition 6.22]{EllRan21}, hence associated with a quasi-uniform constant $\sigma$ such that
\begin{equation}\label{eq: quasi-uniform}
 {\max_{\Ttilde\in \Thlin} h_{\tilde T}} / {\min_{\Ttilde\in \Thlin} h_{\tilde T}} \leq \sigma.
\end{equation}

The set of all the edges of the triangulation $\Thlin$, is denoted by  $\mathcal{E}_h^{(1)}$. For each edge we may also define the outward pointing unit co-normals $\bfm_h^{(1),\pm}$ w.r.t. the two adjacent triangles $\Tilde{T}^{\pm}$ with a common edge. Notice that on $\Galin$,
\begin{itemize}
    \item \ \ $\bfm_h^{(1),+} \neq -\bfm_h^{(1),-}$  (unless the adjacent triangles are in the same plane),  
    \item \ \ $\bfm_h^{(1),\pm}$ is tangential to the planar triangles $\Tilde{T}^{\pm}$, respectively.
\end{itemize}

\subsubsection{Higher order triangulation $\Gah$}
Let us consider a reference element $\hat{T}\in\mathbb{R}^2$, then for every $\Ttilde \in \Thlin$ there exists an affine invertible map 
\begin{equation}
    \begin{aligned}\label{eq: reference map}
       \hat{F}_{\Ttilde}(\hat{x}) : \hat{T} \to \Ttilde,
    \end{aligned}
\end{equation}
that maps the reference triangle $\hat{T}$ to every $\Ttilde$. Now let us describe the higher-order surface approximation $\Gamma_h^{k_g} = \Gah$ with again \emph{quasi uniform} triangulation $T\in \Th$.  Using the previous polyhedral approximation $\Galin$ with triangulation $\Ttilde \in \Thlin$ we consider the map 
\begin{equation}
    \begin{aligned}\label{eq: tildeT to T map}
        F_T(\tilde{x}) : \Ttilde \to T ,
    \end{aligned}
\end{equation}
which may be constructed in the following way: Let $\{\phi_1,...,\phi_{n_{k_g}}\}$ be the Lagrange basis functions of degree $k_g$ of a planar element $\Ttilde\in\Thlin$ corresponding to a set of nodal points on that triangle $\{\tilde{a}_j\}_{j=1}^{n_{k_g}}$. Then with the help of the $k_g-$th order Lagrange interpolant \cite{DziukElliott_acta} on the planar simplex we define 
\begin{equation*}
    \begin{aligned}
        \pi_{k_g}(\tilde{x})  = \mathcal{I}_h^{k_g}\pi(\tilde{x}) = \sum_{i=1}^{n_{k_g}} \pi(\tilde{a}_j)\phi_i(\tilde{x}), \qquad \text{ for } \tilde{x} \in \Ttilde.
    \end{aligned}
\end{equation*}
And by setting $F_T(\tilde{x}) = \pi_{k_g}(\tilde{x})$ for $\tilde{x} \in \Galin$ we have constructed element-wise a continuous polynomial map on $\Galin$. The high-order triangulation can be written as $\Th =  \cup\{F_T(\Ttilde) \ | \ \Ttilde \in \Thlin\}$ and thus the high-order surface approximation can then be defined as $\Gah = \cup_{\Ttilde\in \Th}\{F_T(\Tilde{x}) \ | \ \tilde{x} \in \Ttilde\}$. We now denote element-wise outward unit normal to $\Gah$ by $\nh$ and define the discrete projection $\bfP_h$ onto the tangent space of $\Gah$ by
\begin{equation*}
    \bfPh(x) = \bfI - \nh(x) \otimes \nh(x), \quad x \in T, \text{ where } T \in \Th.
\end{equation*}
Also, the corresponding discrete Weingarten map is $\bfH_h := \nbgcovh\nh.$ For the curved triangulation, we may also consider the co-dimension one set of all the edges of the triangulation $\Th$, denoted by  $\mathcal{E}_h$.  For each edge we may also define the outward pointing unit co-normals $\bfm_h^{\pm}$ with respect to the two adjacent triangles $T^{\pm}$. Notice that on discrete surfaces in general

$$[\mh]|_E = \mh^+ + \mh^- ,$$ 
is not necessarily zero. Now returning back to the reference map \eqref{eq: reference map} we may introduce a new map 
\begin{equation}
    \begin{aligned}\label{eq: hatT T map}
        &\hat{F}_T : \hat{T} \to T, \\
        &\hat{F}_T(\hat{x}) = F_T(\hat{F}_{\tilde{T}}(\hat{x})) \quad \text{ for } \hat{x} \in \hat{T},
    \end{aligned}
\end{equation}
which maps the reference triangle $\hat{T}$ to every triangle $T \in \Th$. Note that this map agrees with the map constructed in \cite[Definition 6.1, Section 9.5]{EllRan21}.

\subsubsection{Geometric approximation errors}
The approximation of the smooth surface $\Ga$ by a discrete triangulated surface $\Gah$ of order $k_g$ produces geometric errors. In what follows we present some known key geometric estimates, see \cite{highorderESFEM, DziukElliott_SFEM, DziukElliott_acta, Heine2004}.
\begin{lemma}[Geometric Errors]\label{lemma: Geometric errors}
For $\Gah$ and $\Ga$ as above, we have the following estimates
\begin{equation}
    \begin{aligned}\label{eq: geometric errors 1}
        \norm{d}_{L^\infty(\Gah)} \leq ch^{k_g + 1}, \quad \norm{\bfn- \nh}_{L^\infty(\Gah)} \leq ch^{k_g}, \quad \norm{\bfH- \bfH_h}_{L^\infty(\Gah)} \leq ch^{k_g -1},
    \end{aligned}
\end{equation}
and thus we may derive the approximations,
\begin{equation}
    \begin{aligned}\label{eq: geometric errors 2}
       &\norm{\bfP- \bfPh}_{L^\infty(\Gah)} \leq ch^{k_g}, \quad \norm{\bfP\cdot\nh}_{L^\infty(\Gah)} \leq ch^{k_g}, \quad \norm{1-\bfn \cdot \nh}_{L^\infty(\Gah)} \leq ch^{k_g+1} \\
       &  \norm{\bfPh\cdot \bfn}_{L^\infty(\Gah)} \leq ch^{k_g}, \quad \norm{\bfPh \bfP - \bfP}_{L^\infty(\Gah)} \leq c h^{k_g},\quad  \norm{\bfP \bfPh \bfP - \bfP}_{L^\infty(\Gah)} \leq c h^{k_g+1}.
    \end{aligned}
\end{equation}
\end{lemma}
\begin{proof}
 The estimates in \eqref{eq: geometric errors 1} are already known, see \cite{DziukElliott_acta, Demlow2009}. The rest can be easily computed as consequence of these first results.
 \end{proof}

Next, we consider the local area deformation transforming $\Galin$ to $\Ga$, i.e. $\muhlin \dslin  = \ds$, where we have that (see  \cite{DziukElliott_acta,DziukElliott_SFEM}),
\begin{equation}
    \begin{aligned}\label{eq: muh1 estimate}
        \norm{1-\muhlin}_{L^{\infty}(\Galin)} \leq ch^{2},
    \end{aligned}
\end{equation}
from which we get the bounds $\norm{\muhlin}_{L^{\infty}(\Galin)} \leq c$ and $\norm{(\muhlin)^{-1}}_{L^{\infty}(\Galin)} \leq c$.

\noindent Similarly for the deformation area $\muh$ between the two surface $\Ga$ and $\Gah$ the following estimate also holds (see \cite{Demlow2009}), 
\begin{equation}
    \begin{aligned}\label{eq: muhkg estimate}
        \norm{1-\muh}_{L^{\infty}(\Gah)} \leq ch^{k_g+1},
    \end{aligned}
\end{equation}
from which we get yet again the bounds $\norm{\muh}_{L^{\infty}(\Gah)} \leq c$ and $\norm{\muh^{-1}}_{L^{\infty}(\Gah)} \leq c$.

\noindent We may now, also define $\mu_{h,1}^k$ to be the area deformation between the two approximated surfaces $\Galin$ and $\Gah$
\begin{equation}
    \begin{aligned}\label{eq: muh1k}
        \mu_{h,1}^k \dslin = \dsh(F_T(\tilde{x})), \quad \text{ for } \tilde{x}\in\Gah.
    \end{aligned}
\end{equation}
where the following estimate holds  (see \cite{EllRan21})
\begin{equation}
    \begin{aligned}\label{eq: muhk1 estimate}
        \norm{1-\mu_{h,1}^k}_{L^{\infty}(\Gah)} \leq ch^{2}.
    \end{aligned}
\end{equation}

\subsection{Finite element function spaces}\label{FEMspaces}

Since certain quantities on the discrete surface $\Gah$, e.g. projections and their derivatives, are only defined elementwise on each surface finite element $T$, it is convenient to introduce broken surface Sobolev spaces together  with their respective norms, \cite{EllRan21}. The norm $\norm{\cdot}_{H^m(\Th)}$ related to the \emph{broken Sobolev space} $\bfH^m(\mathcal{T}_h)=(H^m(\mathcal{T}_h))^3$ is defined by
\begin{equation}
    \begin{aligned}
        \norm{\bfv}_{H^m(\Th)}^2 = \sum_{T\in\Th}\norm{\bfv}^2_{H^m(T)}.
    \end{aligned}
\end{equation}
Given a triangulation $\Th$ of $\Gah$, as described above, we define the $H^1-$conforming Lagrange finite element space on $\Gah$, with finite elements of degree $k\geq 0$ as:
\begin{equation}\label{eq: discrete fem space}
    S_{h,k_g}^{k} := \{v_h \in C^0(\Gah) : v_h|_{T} = \hat{v}_h\circ \hat{F}_{T}^{-1} \text{ for some } \hat{v}_h \in \mathbb{P}^{k}(\hat{T}), \ \text{ for all } T \in \Th \},
\end{equation}
where $\mathbb{P}^{k}(\hat{T})$ is the space of piecewise polynomials of degree $k_u$ on the reference element $\hat{T}$ . Due to the mapping $\hat{F}_T^{-1}$, \eqref{eq: hatT T map}, we observe that the space does not necessarily consists of polynomials over each triangle $T \in \Th$, unless we consider affine finite elements \cite[Example 6.7]{EllRan21}, consisting of piecewise linear triangulation and first-order polynomial functions. Note as well that $S_{h,k_g}^{k} \subset C^0(\Gah)\cap H^1(\Gah)$. An alternate and equivalent way to define the parametric space of polynomial order $k$ and surface approximation of order $k_g$ is the following:
\begin{equation}\label{eq: discrete fem space 2}
    S_{h,k_g}^{k} := \{v_h \in C^0(\Gah) : v_h|_{T} = \tilde{v}_h \circ F_{T}^{-1} \text{ for some } \tilde{v}_h  \in  S_{h,1}^{k}, \ \text{ for all } T \in \Th \}.
\end{equation}
This definition is equivalent to the previous one since $\tilde{v}_h (\tilde{x}) = \hat{v}_h\circ \hat{F}_{\Ttilde}^{-1}(\tilde{x})$, for $\hat{v}_h \in \mathbb{P}^{k}(\hat{T})$ and thus $\tilde{v}_h \circ F_{T}^{-1} = \hat{v}_h\circ \hat{F}_{\Ttilde}^{-1} \circ F_{T}^{-1}$ which according to the map 
\eqref{eq: hatT T map} becomes $\tilde{v}_h \circ F_{T}^{-1} = \hat{v}_h\circ \hat{F}_T^{-1}$. This definition will be useful when we want to relate function on $\Galin$ to functions on $\Gah$.

All the above can be easily expanded for vector valued functions where now the finite element space takes the form:
\begin{equation*}
    (S_{h,k_g}^{k})^3 = S_{h,k_g}^{k}\times S_{h,k_g}^{k} \times S_{h,k_g}^{k} \subset \bfH^1(\Gah) = (H^1(\Gah))^3.
\end{equation*}

\noindent We denote the velocity approximation order by $k_u$, and for the pressures $p$ and $\lambda$, we represent their approximation orders as $k_{pr}$ and $k_{\lambda}$, respectively. Writing 
$$\bm{\mathcal{P}}_{k_u} := (S_{h,k_g}^{k_u})^3, \quad \mathcal{P}_{k_{pr}}:= S_{h,k_g}^{k_{pr}}, \quad
\mathcal{P}_{k_{\lambda}} :=S_{h,k_g}^{k_\lambda},$$ 
we  use the notation $\bm{\mathcal{P}}_{k_u}/\mathcal{P}_{k_{pr}}/\mathcal{P}_{k_{\lambda}}$ to denote  \emph{Taylor-Hood} finite elements, where $k_{pr}= k_u-1$.

\subsection{Surface lifting}\label{sec: surface lifting}
\subsubsection{Lifting to the exact surface}

Our finite element approximation is defined for functions on the discrete surface $\Gah$, and therefore the solution will also belong to this approximated surface, while the solution of \eqref{eq: generalized tangential surface stokes} lies on $\Ga$. To analyze and compare the two solutions we need mappings between functions on $\Ga$ and $\Gah$. We accomplish this by using an extension/lifting procedure; see \cite{DziukElliott_acta,Demlow2009}.

We already assumed that $h$ is small enough such that $\Gah \subset U_\delta$, hence the projection $\pi(\cdot)$ is well-defined and a bijection. Therefore, for every triangle in the triangulation $T \in \Th$ there exists an induced curvilinear triangle $T^\ell = \pi(T) \subset \Ga$ by the triangulation $\Th$ of $\Gah$. We may then define the conforming quasi-uniform triangulation $\Th^\ell$ of $\Ga$ :
\begin{equation*}
    \Ga = \cup_{T^\ell \in \Th^\ell}T^\ell.
\end{equation*}

We will define the lift operator on $\Gah$ using the closest point projection operator $\pi(\cdot)$. For any finite element function $v_h: \Gah \to \mathbb{R}^n$, we define the \emph{lift} onto $\Ga$ by
\begin{equation}\label{eq: lift}
    v_h^\ell(\pi(x)) := v_h(x), \quad \text{for } x\in\Gah.
\end{equation}
Similar to the extension \eqref{eq: smooth extension}, we also define the inverse lift operator for a function $v :\Ga \to \mathbb{R}^m$:
\begin{equation}\label{eq: inverse lift}
    v^{-\ell}(x) := v(\pi(x)), \quad \text{for } x\in\Gah.
\end{equation}
Notice that the two extensions \eqref{eq: smooth extension} and \eqref{eq: inverse lift} agree on the discrete surface $\Gah$ i.e. $v^{-\ell} = v^e|_{\Gah}$ since we assumed $\Gah \subset U_\delta$. 
Lastly, we also introduce the space of lifted finite element functions given by
\begin{equation*}
    (S_{h,k_g}^{k})^\ell = \{v_h^\ell \in H^1(\Ga) : v_h \in S_{h,k_g}^{k} \} \subset H^1(\Ga).
\end{equation*}

We now compare the functions defined on the discrete surface $\Gah$ and the exact surface $\Ga$. We relate their tangential derivatives with the help of the lift extension defined above. We start with the standard lifts regarding scalar functions and then extend these results component-wise to vector-valued functions.  Set

 \begin{equation}\Bhg = \bfPh(\bfI - d\bfH)\bfPg.\end{equation}

\begin{description}
    \item[Scalar Functions :: ] Considering the extension of a scalar function  $v_h : \Gah \to \mathbb{R}$, \eqref{eq: lift} and the chain rule we have that the discrete tangent gradient of $v_h \in H^1(\Gah)$, for $x \in \Gah$ is
\begin{equation}
    \begin{aligned}
        \nbgh v_h(x) = \bfPh \nb v_h(x) = \bfPh(\bfI - d\bfH)\bfPg\nb v_h^\ell(\pi(x)) = \Bhg \nbg v_h^\ell(\pi(x)),
    \end{aligned}
\end{equation}
where we see later on that  $\Bhg : T_{\pi(x)}(\Ga) \mapsto T_x(\Gah)$ is invertible, i.e. we have that the map $\Bhg^{-1}$ on the tangent space $T_{x}(\Gah)$ is given by $$\Bhg^{-1}|_{T_{x}(\Gah)} = \bfPg(\bfI - d\bfH)^{-1}(\bfI - \frac{\nh\otimes\bfn}{\nh \cdot \bfn})\bfPh,$$
thus we have for $u_h : \Gah \to \mathbb{R}$, and for $x\in \Gah$, see \cite{demlow2007adaptive},
\begin{equation}
    \begin{aligned}
        \nbg v_h^{\ell}(\pi(x)) = \bfPg(\bfI - d\bfH)^{-1}(\bfI - \frac{\nh\otimes\bfn}{\nh \cdot \bfn})\nbgh v_h(x) = \Bhg^{-1}\nbgh v_h(x).
    \end{aligned}
\end{equation}
Noticing that the matrix $\bfG = (\bfI - \frac{\nh\otimes\bfn}{\nh \cdot \bfn})|_{T_x(\Gah)} = (\bfI - \frac{\nh\otimes\bfn}{\nh \cdot \bfn}) \bfPh$ we can easily calculate that $\Bhg^{-1}\Bhg = \bfPg$ and $\Bhg\Bhg^{-1} = \bfPh$.
\item[Vector Valued Functions :: ] Considering the tangential gradients of vector valued function \eqref{eq: tangential derivative} and \eqref{eq: tangential derivative matrix} we may apply the lift extension component-wise, for each row, and with similar calculations to the scalar case we may get, after factoring out the common map $\Bhg$, that
\begin{equation}
    \begin{aligned}
        \nbgh \vh(x) = \nbgh \vh^\ell(\pi(x)) 
        =  \nb\vh^\ell(\pi(x))\bfPh = \nbg \vh^\ell(\pi(x))\Bhg^t.
    \end{aligned}
\end{equation}
Likewise for the inverse transformation we get
\begin{equation}
    \begin{aligned}
        \nbg \vh^{\ell}(\pi(x)) = \nbgh \vh(x)(\Bhg^{-1})^t.
    \end{aligned}
\end{equation}
\end{description}

Recalling the results from \cite{Heine2004,DziukElliott_acta,hansbo2020analysis,jankuhn2021trace} we have the following lemma, which also shows that $\bfB_h$ is invertible.
\begin{lemma}[Estimates for $\bfB_h$]\label{Lemma: Bh estimates}
For sufficiently small $h$ we have the following bounds
\begin{align}
\label{eq: Bh stability}
    \norm{\Bhg}_{L^{\infty}(\Ga)} &\leq c, \qquad \qquad \ \norm{\Bhg^{-1}}_{L^{\infty}(\Gah)} \leq 1,\\
    \label{eq: Bh estimates}
    \norm{\bfPg - \Bhg}_{L^{\infty}(\Ga)} &\leq ch^{k_g} , \qquad \quad \norm{\bfPg - \Bhg^t\Bhg}_{L^{\infty}(\Ga)} \leq ch^{k_g+1},    
\end{align}
where the constant $c$ is independent of $h$.
\end{lemma}
\begin{proof}
    The first two bounds on $\Bhg$ and its inverse follow directly from the bounds on the projection and normals, as well as the geometric errors \eqref{eq: geometric errors 1}. For the third estimate we have the following calculations 
    \begin{equation*}
        \begin{aligned}
            \norm{\bfPg - \Bhg}_{L^{\infty}(\Gah)} &\leq \norm{\bfPg - \bfPh(\bfI - d\bfH)\bfPg }_{L^{\infty}(\Gah)} \leq \norm{(\bfPg- \bfPh\bfPg) + \bfPh d\bfH\bfPg }_{L^{\infty}(\Gah)} \\ &\leq \norm{ \bfPg \nh \otimes \nh}_{L^{\infty}(\Gah)} + \norm{\bfPh d\bfH\bfPg }_{L^{\infty}(\Gah)} \leq ch^{k_g},
        \end{aligned}
    \end{equation*}
    where the last inequality holds again due to the geometric errors \eqref{eq: geometric errors 1} and the uniform bound of the curvature tensor $\bfH$. For the fourth estimate we have similarly 
     \begin{equation*}
        \begin{aligned}
            \norm{\bfPg - \Bhg^t\Bhg}_{L^{\infty}(\Gah)} &\leq \norm{\bfPg - \bfPg(\bfI - d\bfH)\bfPh (\bfI - d\bfH)\bfPg}_{L^{\infty}(\Gah)} \\
            &\leq \norm{\bfPg - \bfPg\bfPh\bfPg}_{L^{\infty}(\Gah)} + \mathcal{O}(h^{k_g+1}) \\
            &\leq \norm{\bfPg\nh\otimes\nh\bfPg}_{L^{\infty}(\Gah)} + \mathcal{O}(h^{k_g+1}) \\
            &\leq ch^{k_g+1}.
        \end{aligned}
    \end{equation*}
    which completes the proof. 
    \end{proof}

\subsubsection{Norm equivalence}

We use $\sim$ to denote the norm equivalence independent of $h$, then by \cref{lemma: Geometric errors} we obtain the following.

%and by the estimates in \cref{lemma: Geometric errors} we recall the following Lemma, involving equivalence of norms on different surfaces.
\begin{lemma}[Norm Equivalence]
    \label{lemma: norm equivalence}
    Let $v_h \in H^j(T),$ $j\geq 2$ for $T\in \Th$ a face of $\Gah$ and let $\tilde{v}_h(\tilde{x}) = v_h(F_T(\tilde{x}))$, with $\Tilde{T}$ a triangular face of           $\Galin$. Then for $h$ small enough we have that following equivalences
    \begin{equation}\label{eq: norm equivalence}
        \begin{aligned}
            \norm{v_h^{\ell}}_{L^2(T^{\ell})}&\sim \norm{v_h}_{L^2(T)}\sim \norm{\tilde{v}_h}_{L^2(\Tilde{T})}, \qquad
            \norm{\nbg v_h^{\ell}}_{L^2(T^{\ell})}\sim \norm{\nbgh v_h}_{L^2(T)}\sim \norm{\nbglin\tilde{v}_h}_{L^2(\Tilde{T})},\\
           % \norm{\vhl}_{H^{-1}(T^{\ell})} &\sim \norm{\vh}_{H^{-1}(T)} \\
            |\nbglin^j \tilde{v}_h|_{L^2(\Tilde{T})} &\leq c \sum_{k=1}^j |\nbgh v_h|_{L^2(T)}, \\
            |\nbgh^j v_h|_{L^2(T)} &\leq c \sum_{k=1}^j |\nbg v_h^{\ell}|_{L^2(T^\ell)}, \qquad
             |\nbg^j v_h^{\ell}|_{L^2(T^{\ell})} \leq c \sum_{k=1}^j |\nbgh v_h|_{L^2(T^\ell)}.
        \end{aligned}
    \end{equation}
\end{lemma}
\begin{proof}
The first three equivalences can be found in several papers, e.g. \cite{DziukElliott_SFEM,Demlow2009}. Regarding the last two inequalities, see \cite[Lemma 3.3]{LarssonLarsonContDiscont}.
    % The first two set of equivalences and the fourth inequality can be found in several papers, e.g. \cite{DziukElliott_SFEM,Demlow2009}.  Finally, regarding the last two inequalities, see \cite[Lemma 3.3]{LarssonLarsonContDiscont}.
    %The third inequality is an immediate consequence of the second equivalence relation and \eqref{eq: muhkg estimate}.
\end{proof}

\subsection{Interpolation}\label{sec: Interpolation}
For the error analysis and stability we are going to use the Scott-Zhang interpolation operator $\Ihzlin : L^2(\Galin)^3 \to (S_{h,1}^{k})^3$, introduced for triangulated surfaces in \cite{camacho20152}. Approximation properties were proven for the case of affine linear triangulations and scalar functions.  These can be easily extended to vector-valued functions so that the following interpolation bounds hold. For every $\Ttilde \in \Thlin$ and $\tilde{\bfv} \in L^2(\Galin)^3$:
\begin{equation}\label{eq: Scott-Zhang interpolant linear}
    \norm{\tilde{\bfv} - \Ihzlin\tilde{\bfv}}_{H^l(\Ttilde)} \leq ch^{k+1-l}\norm{\tilde{\bfv}}_{H^{k
+1}(\omega_{\Ttilde})} \leq ch^{k+1-l}\norm{\tilde{\bfv}^{\ell}}_{H^{k+1}(\omega_{\Ttilde}^\ell)}, \ \ 0\leq l \leq k+1, \ k\geq0
\end{equation}
where $\norm{\tilde{\bfv}}_{H^{k+1}(\omega_{\Ttilde})} = \sum_{\Ttilde \in \omega_T} \norm{\tilde{\bfv}}_{H^{k+1}(\Ttilde)}$ and $\omega_{\Ttilde}$ is the union of elements in $\Thlin$ which  share a node with $\Ttilde$. The last inequality holds due the norm equivalence \eqref{eq: norm equivalence}, the bijection of the projection $\pi(\cdot)$, Section \ref{sec: surface lifting}, and the fact that the curvilinear triangles $\omega_T^\ell$ are  lifts of  elements in $\omega_{\Ttilde}$. From the construction of the interpolant in \cite{camacho20152} we can easily see that that $\Ihzlin$ is a projection.

Now following \cite{hansbo2020analysis}, \cite{Demlow2009}, with the help of the maps in Section \ref{sec: Triangulated surfaces} we can extend the definition onto the approximating surface of higher order  $\Gah$ by the interpolant $\Ihz = \widetilde{I}_{h,k_g}^z : (L^2(\Gah))^3 \to (S_{h,k_g}^{k})^3$ for all $\bfv \in L^2(\Gah)^3$ with the help of the map $F_T : \Ttilde \to T$, see \eqref{eq: tildeT to T map}, as
\begin{equation*}
    \Ihz \bfv = \Ihzlin \Tilde{\bfv}(F_T^{-1}(x)), \quad\text{ for } x \in \Gah
\end{equation*}
with lift $\Ihzl \bfv^\ell = (\Ihz\bfv)^\ell$. 
As a result of the uniform $L^\infty$-bounds on derivatives of our map $F_T$, since by construction this map is the Lagrange interpolant of the closest point projection $\pi$, as well as the bound on $\muh$ for small enough $h$,  we have the following inequality
\begin{equation}
\begin{aligned}\label{eq: Scott-Zhang interpolant}
    \norm{\bfv - \Ihz \bfv}_{H^l(T)} \leq c \norm{\bfv - \Ihzlin \bfv}_{H^{l}(\Ttilde)}&\leq ch^{k+1-l}\norm{\bfv}_{H^{k+1}(\omega_{\Ttilde})} \\ 
    &\leq ch^{k+1-l}\norm{\bfv^{\ell}}_{H^{k+1}(\omega_{\Ttilde}^\ell)}, \quad 0\leq l \leq k+1, \ k \geq 0,
    \end{aligned}
\end{equation}
for every $\bfv \in L^2(\Gah)^3$, $T\in \Th$, where again $\omega_{\Ttilde}$ is the union of elements in $\Thlin$ which are neighbours to $\Ttilde$ i.e. they share a node with $\Ttilde$. 
According to \cite{hansbo2020analysis} we also have the stability estimate 
\begin{equation}\label{eq: stability of Scott-Zhang interpolant}
    \norm{\Ihz \bfv}_{H^k(T)} \leq c\norm{\bfv}_{H^k(\omega_{\Ttilde})} \leq c\norm{\bfv^{\ell}}_{H^k(\omega_{\Ttilde}^\ell)}, \quad \text{ for } k=0,1.
\end{equation}
for every $\Ttilde\in \Thlin$.

Finally, we present a super-approximation type of estimate and stability for the Scott-Zhang interpolant. This type of estimate has been proven in \cite{hansbo2020analysis} due to the projection property of $\Ihz$ onto the finite element space $(S_{h,k_g}^{k})^3$.
We present the lemma for completeness.
\begin{lemma}\label{Lemma: super-approximation properties}
    For discrete functions $\vh\in (S_{h,k_g}^{k})^3$ and $\chi \in [W_{\infty}^{k+1}(\Ga)]^3$  it holds
    \begin{equation}
        \begin{aligned}
        \label{eq: super-approximation estimate 1}
            \sum_{T\in \Th}\norm{\nbgh(I - \Ihz)(\chi^{-\ell} \cdot \vh)}_{L^2(T)} \leq c h\norm{\chi}_{W_{\infty}^{k+1}(\Ga)} \norm{\vh}_{H^1(\Gah)},
        \end{aligned}
    \end{equation}
     \begin{equation}
        \begin{aligned}
        \label{eq: super-approximation estimate 2}
            \sum_{T\in \Th}\norm{(I - \Ihz)(\chi^{-\ell} \cdot \vh)}_{L^2(T)} \leq ch \norm{\chi}_{W_{\infty}^{k+1}(\Ga)} \norm{\vh}_{L^2(\Gah)}.
        \end{aligned}
    \end{equation}
We also have the $L^2$ stability estimate
    \begin{equation}
        \begin{aligned}
        \label{eq: super-approximation stability}
            \norm{\Ihz(\chi^{-\ell} \cdot \vh)}_{L^2(\Gah)} \leq c\norm{\chi^{-\ell} \cdot \vh}_{L^2(\Gah)}.
        \end{aligned}
    \end{equation}
\end{lemma}
\noindent The second estimate \eqref{eq: super-approximation estimate 2} can be calculated similarly to the estimate \eqref{eq: super-approximation estimate 1} as in \cite{hansbo2020analysis}. We will see that this estimate plays an important role for the well-posedness of the discrete Lagrange formulation.
\subsection{Geometric perturbations}\label{sec: Geometric perturbations}

Analogously to the continuous case of Definition \ref{def: Covariant derivatives}  we define similarly for the discrete case the following operators for $\vh \in \bfH^1(\Gah)$,
\begin{equation}
    \begin{aligned} \label{eq: Eh forms}
        \nbgcovh \vh &:= \bfPh \nbgh \vh , \ x \in \Gah \\
        E_h(\vh) &:= \half (\nbgcovh \vh + \nb_{\Gah}^{cov,t} \vh), \\
     \end{aligned}
\end{equation}

\begin{lemma}[Change of the domain of integration]\label{lemma: errors of domain of integration}
For $h$ small enough, $\uh,\vh \ \in \bfH^1(\Gah)$, $\bff \in (L^2(\Ga))^3$, and $g \in L^2(\Ga)$ we have the following error bounds
\begin{equation}
    \begin{aligned}\label{eq: errors of domain of integration data}
        |(\bff,\vhl)_{L^2(\Ga)}- (\bff^{-\ell},\vh)_{L^2(\Gah)}| &\leq  ch^{k_g+1}\norm{\bff}_{L^2(\Ga)}\norm{\vh}_{L^2(\Gah)},\\[5pt]
        |(g,\qhl)_{L^2(\Ga)} - (g^{-\ell},\qh)_{L^2(\Gah)}| &\leq  ch^{k_g+1}\norm{g}_{L^2(\Ga)}\norm{\qh}_{L^2(\Gah)}. \\
    \end{aligned}
\end{equation}
\end{lemma}
\begin{proof}
See \cite{highorderESFEM, DziukElliott_acta} for the proof that can be easily extended to vector-valued functions.
\end{proof}
\begin{lemma}\label{lemma: Geometric perturbations a}
For $\bfw_h,\vh \in \bfH^1(\Gah)$ the following bounds hold:
    \begin{align} \label{eq: errors of domain of integration stiffness}
        &|(\nbgcov \bfw_h^{\ell}, \nbgcov \vhl)_{L^2(\Ga)} - (\nbgcovh \bfw_h,\nbgcovh \vh)_{L^2(\Gah)}| \leq ch^{k_g} \norm{\bfw_h}_{H^1(\Gah)}\norm{\vh}_{H^1(\Gah)},\\[7pt]
    \label{eq: Geometric perturbations a 1}
        &| (E_h(\bfw_h),E_h(\vh))_{L^2(\Gah)} - (E(\bfw_h^{\ell}),E(\vhl))_{L^2(\Ga)}| \leq ch^{k_g} \norm{\bfw_h}_{H^1(\Gah)}\norm{\vh}_{H^1(\Gah)} \\[7pt]
      &\norm{E(\bfw_h^{\ell})^{-\ell} - E_h(\bfw_h)}_{L^2(\Gah)} \leq ch^{k_g}\norm{\bfw_h}_{H^1(\Gah)}\label{eq: Geometric perturbations a 2},
    \end{align}
where the constant $c$ are independent of the mesh parameter $h$. 

\end{lemma}
\begin{proof}
Let us start by proving \eqref{eq: errors of domain of integration stiffness}, which can be estimated similarly to \cite{Heine2004,DziukElliott_acta}. Considering that $\nbgcov (\cdot) = \bfPg\nbg (\cdot)$ and $\nbgcovh (\cdot) = \bfPh\nbgh (\cdot)$ the calculations are as followed,
\begin{equation*}
    \begin{aligned}
        &|(\nbgcov \bfw_h^{\ell}, \nbgcov \vhl)_{L^2(\Ga)} - (\nbgcovh \bfw_h,\nbgcovh \vh)_{L^2(\Gah)}| \\
        &\ = \big| \int_{\Ga}\nbgcov \bfw_h^{\ell} : \nbgcov \vhl \, \ds - \int_{\Gah}\nbgcovh \bfw_h : \nbgcovh \vh \, \dsh \big|\\
        &\ = \big|\int_{\Ga}\nbgcov \bfw_h^{\ell} : \nbgcov \vhl \, \ds - \int_{\Ga} \frac{1}{\mu_h} (\nbgcovh \bfw_h)^{\ell} : (\nbgcovh \vh)^{\ell} \, \ds\big|\\
        &\ = \big|\int_{\Ga} \bfPg\nbg \bfw_h^{\ell} : \bfPg\nbg \vhl \, \ds  - \int_{\Ga}\bfPh\nbg \bfw_h^{\ell} \Bhg^t : \bfPh\nbg \vhl \Bhg^t \, \ds \\
        &\quad +   \int_{\Ga} (1 - \frac{1}{\mu_h})\bfPh\nbg \bfw_h^{\ell} \Bhg^t : \bfPh\nbg \vhl \Bhg^t \, \ds \big|.
        \end{aligned}
\end{equation*}
The second term is once again immediate from the deformation estimate \eqref{eq: muhkg estimate} and the bounds of $\bfPh$ and $\Bhg$ \eqref{eq: Bh stability} so that 
\begin{equation*}
    \begin{aligned}
          \big|\int_{\Ga} (1 - \frac{1}{\mu_h})\bfPh\nbg \bfw_h^{\ell} \Bhg^t : \bfPh\nbg \vhl \Bhg^t \, \ds \leq ch^{k_g+1} \norm{\bfw_h}_{H^1(\Gah)}\norm{\vh}_{H^1(\Gah)} \big|.
    \end{aligned}
\end{equation*}
Regarding the first term,  using the properties of the Frobenius inner product we have that
\begin{equation*}
    \begin{aligned}
        & \big|\int_{\Ga} \bfPg\nbg \bfw_h^{\ell} : \bfPg\nbg \vhl \, \ds  - \int_{\Ga}\bfPh\nbg \bfw_h^{\ell} \Bhg^t : \bfPh\nbg \vhl \Bhg^t \, \ds \big|\\
         &=  \big|\int_{\Ga} \bfPg\nbg \bfw_h^{\ell} : \nbg \vhl \, \ds  - \int_{\Ga}\bfPh\nbg \bfw_h^{\ell}  : \nbg \vhl \Bhg^t\Bhg \, \ds \big| \\
        &=  \big|\int_{\Ga} \bfPg\nbg \bfw_h^{\ell} : \nbg \vhl (\bfPg - \Bhg^t\Bhg) \, \ds + \int_{\Ga}(\bfPg-\bfPh)\nbg \bfw_h^{\ell}  : \nbg \vhl \Bhg^t\Bhg \, \ds \big|\\
        &\leq ch^{k_g+1}\norm{\bfw_h}_{H^1(\Gah)}\norm{\vh}_{H^1(\Gah)} + ch^{k_g}\norm{\bfw_h}_{H^1(\Gah)}\norm{\vh}_{H^1(\Gah)},
    \end{aligned}
\end{equation*}
by using \eqref{eq: geometric errors 2}. Combining the two terms we get our desired result.

Now, to prove the estimate \eqref{eq: Geometric perturbations a 1} we just need to notice that $E_h(\bfw_h) = \half (\nbgcovh \bfw_h + \nabla_{\Gah}^{cov,t}\bfw_h)$  and $E(\bfw_h^{\ell}) = \half (\nbgcov \bfw_h^{\ell} + \nb_{\Ga}^{cov,t}\bfw_h^{\ell})$ from \eqref{eq: Eh forms} and thus in combination with the above estimate \eqref{eq: errors of domain of integration stiffness} the inequality follows. Finally, \eqref{eq: Geometric perturbations a 2}, can be easily estimated similarly as above.
\end{proof}

\subsection{Discrete Korn's inequality}

\begin{lemma}[Discrete Korn's inequality]
\label{Lemma: discrete Korn's inequality T nh}
For h small enough it holds that  
\begin{equation}
\label{discrete Korn's inequality T nh}
    \norm{\vh}_{H^1(\Gah)} \leq c(\norm{E_h(\vh)}_{L^2(\Gah)}  +\norm{\vh}_{L^2(\Gah)} + h^{-1}\norm{\vh \cdot \nh}_{L^2(\Gah)}), \quad\end{equation}
$\text{ for all } \vh \in (S_{h,k_g}^{k})^3$. 
\end{lemma}
\begin{proof}
Let us start from the continuous relation, similar to \cite{jankuhn2021trace}. We know in the continuous tangent space $\bfH_T^1(\Ga)$,  Korn's inequality
\eqref{eq: Korn inequality} holds so the bound for the tangential term $\norm{\bfPg \vhl}_{H^1(\Ga)}$ reads as 
\begin{equation}
    \begin{aligned}
    \norm{\bfPg \vhl}_{H^1(\Ga)} &\leq c(\norm{E(\bfPg \vhl)}_{L^2(\Ga)} + \norm{\bfPg \vhl}_{L^2(\Ga)}). \\
    \end{aligned}
\end{equation}
On the other hand, $E(\vhl) = E(\bfPg\vhl) + (\vhl \cdot \bfng)\bfH$ for all $\vhl \in \bfH^1(\Ga)$ so we may obtain, due to norm equivalence, that
\begin{equation}
    \begin{aligned}
    \label{Korn H1}
    \norm{\vh}_{H^1(\Gah)} \leq c\norm{\vhl}_{H^1(\Ga)} \leq c\norm{\bfPg \vhl}_{H^1(\Ga)} + c\norm{\vhl \cdot \bfng}_{H^1(\Ga)}.
    \end{aligned}
\end{equation}
We now try and bound the two terms appearing on the right hand-side appropriately. In regards to the second term we have for $\bfn = \bfng^{-\ell}$ that
\begin{equation*}
    \begin{aligned}
    \norm{\vhl \cdot \bfng}_{H^1(\Ga)} &\leq  c\norm{\vhl \cdot \bfng}_{L^2(\Ga)} + c\norm{\nbg(\vhl \cdot \bfng)}_{L^2(\Ga)} \\
    &\leq c\norm{\vh}_{L^2(\Gah)} + c\norm{\nbgh(\vh \cdot \bfn)}_{L^2(\Gah)}.    
    \end{aligned}
\end{equation*}
We cannot directly apply an inverse-type estimate to the second term. For that we want to utilize the super-approximation \eqref{eq: super-approximation estimate 1} and the fact that the Scott-Zhang interpolant \eqref{eq: Scott-Zhang interpolant} maps onto our finite element space $(S_{h,k_g}^{k})^3$. To do so we have the following
\begin{equation*}
    \begin{aligned}
   \norm{\nbgh(\vh \cdot \bfn)}_{L^2(\Gah)} &\leq \norm{\nbgh(I-\Ihz)(\vh \cdot \bfn)}_{L^2(\Gah)}  + \norm{\nbgh\Ihz(\vh \cdot \bfn)}_{L^2(\Gah)} \\
   &\leq c\norm{\vh}_{L^2(\Gah)} + ch^{-1}\norm{\vh \cdot \bfn}_{L^2(\Gah)},
    \end{aligned}
\end{equation*}
where we used the inverse estimate \cite[Lemma 4.5]{Heine2004} for the second right-hand side term. So, combining this result we get 
\begin{equation}\label{eq: h1n term bound}
     \norm{\vhl \cdot \bfng}_{H^1(\Ga)} \leq c\norm{\vh}_{L^2(\Gah)} + ch^{-1}\norm{\vh \cdot \bfn}_{L^2(\Gah)},
\end{equation}
It remains to bound the tangent term $\norm{\bfPg \vhl}_{H^1(\Ga)}$,
\begin{equation*}
    \begin{aligned}
    \norm{\bfPg \vhl}_{H^1(\Ga)} &\leq c\norm{E(\bfPg\vhl)}_{L^2(\Ga)} +  c\norm{\bfPg \vhl}_{L^2(\Ga)} \\ 
    &\leq c\norm{E(\vhl)}_{L^2(\Ga)} + c \norm{(\vhl \cdot \bfng)\bfH}_{L^2(\Ga)}  +\norm{\bfP \vh}_{L^2(\Gah)} \\ 
    & \leq c\norm{E_h(\vh)}_{L^2(\Gah)} + c\norm{E(\vhl) - E_h(\vh)}_{L^2(\Gah)} + c\norm{\vh}_{L^2(\Gah)}\\
    &\leq c\norm{E_h(\vh)}_{L^2(\Gah)} + ch^{k_g}\norm{\vh}_{H^1(\Gah)}+ ch^{k_g}\norm{\vh}_{L^2(\Gah)}+c\norm{\vh}_{L^2(\Gah)},
    \end{aligned}
\end{equation*}
where $\bfP = \bfPg^{-\ell}$ and where we have used geometric perturbation result \eqref{eq: Geometric perturbations a 2} in Section \ref{sec: Geometric perturbations} and the geometric estimates in \eqref{eq: geometric errors 2}.  Now combining the two estimate and absorbing the appropriate term to the left-hand side we get the desired result.
\end{proof}

\begin{remark}
% \leavevmode
 Note that, the needed regularity of the normal $\bfn_{\Ga}$ stems from applying the super-approximation estimate \eqref{eq: super-approximation estimate 1}. It forces us to consider normals such that $\bfn_{\Ga}\in W^{k+1}_{\infty}$ and thus limits us to surfaces of regularity at least $C^{k+2}$.

\end{remark}
\subsection{Integration by parts}
On the smooth surface the integration by parts for the bilinear form $b_T(\bfv,p)$ is given by $\int_\Ga \bfv \cdot \nbg p = - \int_\Ga
p \, \divg \bfv_T $. On the other hand, integration by parts for a general vector field function $\vh$ on a discretized surface $\Gah$ gives rise to extra terms due to jumps of co-normals $\mh$ %(cf.   \cite{jankuhn2021Higherror} and \cite{fries2018higher}). 
\begin{lemma}
 On the discrete surface $\Gah$ the following integration by parts formula holds:
 \begin{equation}
     \begin{aligned}
     \label{eq: integration by parts}
         \int_{\Gah} \vh \cdot \nbgh \qh \, \dsh &= - \int_{\Gah} \qh \divgh \vh \, \dsh + \sum_{T\in \mathcal{T}_h} \int_T (\vh \cdot \nh)\qh \divgh \nh \, \dsh \\
         &+ \sum_{E \in \mathcal{E}_h} \int_E [\mh \cdot \vh] \qh \, d\ell
     \end{aligned}
 \end{equation}
 for $\vh \in (S_{h,k_g}^{k})^3$ and $\qh \in S_{h,k_g}^{l}$, for any $k,l \geq \in \mathbb{N}$.
\end{lemma}

\section{The discrete Lagrange formulation}\label{sec: well-posedness of discrete formulation}

We consider the \emph{Taylor-Hood} surface finite elements $\bm{\mathcal{P}}_{k_u}/\mathcal{P}_{k_{pr}}/\mathcal{P}_{k_{\lambda}}$  for a $k_g$-order approximation of the surface, (c.f.  Section \ref{FEMspaces}), and  set $$k_u\geq 2,~k_{pr}=k_u-1~~ \mbox{and} ~~1 \leq k_{\lambda}\leq k_u.$$ Let us define the following finite element spaces
% We maintain the notation  $k_g,\,k_u,/ k_{pr}$ and $k_\lambda$ in order to track how  error dependencies arise. Let us define the following finite element spaces
\begin{equation}
    V_h = S_{h,k_g}^{k_u}, ~~ \bfV_h = ( V_h)^3, ~~
        Q_h = S_{h,k_g}^{k_{pr}} \cap L^2_0(\Gah), ~~
        \Lambda_h = S_{h,k_g}^{k_{\lambda}}.
   \end{equation}
We also define the subspace of weakly discretely tangential divergence-free velocity fields
\begin{align}
    \bfV_h^{div} := \{\bfw_h \in \bfV_h : \bhtil(\bfw_h,\{\qh,\xi_h\}) =0, ~~\forall \, \{\qh,\xi_h\} \in Q_h\times\Lambda_h\}.
\end{align}
Finally, we will need the dual space of the finite element space $\Lambda_h$, which we denote as 
$\Lambda_h^* := H_h^{-1}(\Gah)$. More precisely, for $\ell_h \in H_h^{-1}(\Gah)$  we define its dual norm as 
\begin{equation}\label{eq: H^-1h definition}
    \norm{\ell_h}_{H_h^{-1}(\Gah)} = \sup_{\xi_h\in \Lambda_h\backslash\{0\}}\frac{<\ell_h,\xi_h>_{H^{-1},H^1}}{\norm{\xi_h}_{H^1(\Gah)}}.
\end{equation}

\subsection{The discrete bilinear forms}

The finite element method relies on the   discrete  bilinear forms $\ah(\cdot,\cdot),  \, \bhtil(\cdot,\{\cdot,\cdot\})$  defined by,
\begin{align}
   \label{eq: lagrange discrete bilinear forms}
    \ah(\bfw_h,\vh) &:= \int_{\Gah} E_h(\bfw_h):E_h(\vh) \dsh + \int_{\Gah} \bfw_h \cdot \vh \,\dsh, \\
     \label{eq: lagrange discrete bilinear forms 22}
      \bhtil(\bfw_h,\{\qh,\xi_h\}) &:= \int_{\Gah} \bfw_h \cdot \nbgh \qh  \, \dsh + \int_{\Gah} \bfw_h\cdot \nh \,  \xi_h \, \dsh,\\
       \label{eq: lagrange discrete bilinear formm b}
      \bh(\bfw_h,\qh) &:= \int_{\Gah} \bfw_h \cdot \nbgh \qh  \, \dsh.
    \end{align}
In our stability analysis we use the energy norm 
\begin{equation}
\norm{\bfw_h}_{\ah} = \norm{E_h(\bfw_h)}_{L^2(\Gah)} + \norm{\bfw_h}_{L^2(\Gah)},
\end{equation}
for a velocity vector $\bfw_h \in \bfH^1(\Gah)$.

\begin{lemma}[Bounds and coercivity]
\label{Lemma: discrete bounds and coercivity results}
    The following bounds and coercivity estimates hold:
    \begin{align}
        \label{coercivity and Korn's inequality Lagrange}
        \norm{\bfw_h}_{H^1(\Gah)}^2 &\leq a_h(\bfw_h,\bfw_h) + h^{-2}\norm{\bfw_h\cdot\nh}_{L^2(\Gah)}^2\leq c h^{-2} \ah(\bfw_h,\bfw_h)  \quad  \ \, \textnormal{ for all } \bfw_h \in \bfV_h, \\
        \label{ah boundedness}
        \ah(\bfw_h,\vh) &\leq \norm{\bfw_h}_{\ah}\norm{\vh}_{\ah}  \ \ \ \qquad \quad \ \ \textnormal{ for all } \bfw_h,\vh \in \bfV_h,\\
          \label{bhtilde boundedness}
       \bhtil(\bfw_h,\{\qh,\xi_h\}) &\leq c \norm{\bfw_h}_{\ah} \norm{\{\qh,\xi_h\}}_{L^2(\Gah)} \ \textnormal{ for all } \bfw_h \in \bfV_h, \, (\qh,\xi_h) \in Q_h\times \Lambda_h.
    \end{align}
\end{lemma}

\begin{proof}
The $H^1$ coercivity w.r.t. energy norm \eqref{coercivity and Korn's inequality Lagrange} follows immediately from the discrete Korn's Inequality \eqref{discrete Korn's inequality T nh}, while the bound \eqref{ah boundedness} is derived with simple use of the Cauchy-Schwarz inequality. For the last inequality, as previously, we need to use the integration by parts formula \eqref{eq: integration by parts}. This, coupled with the bound for the co-normal $|[\mh]| \leq ch^{k_g}$ (c.f. \cite{EllRan21}) and $\divgh(\vh) = tr(\nbgh \vh) = tr(\nbgcovh \vh)$,
\begin{equation*}
    \int_{\Gah} \qh \divgh \vh \, \dsh \leq \norm{\qh}_{L^2(\Gah)}\norm{\vh}_{\ah},
\end{equation*}
completes the proof.
\end{proof}

\begin{lemma}[Improved $H^1$ coercivity bound]\label{lemma: improved h1-ah bound}
     If $\underline{k_{\lambda} = k_u}$, that is $\Lambda_h = V_h$, and $\bfw_h \in \bfV_h^{div}$ then the following  holds
     \begin{align}\label{eq: improved h1-ah bound}
         \norm{\bfw_h}_{H^1(\Gah)}^2 &\leq a_h(\bfw_h,\bfw_h).
     \end{align}
\end{lemma}
\begin{proof}
From \eqref{coercivity and Korn's inequality Lagrange} we see that we  need to find a bound for $\norm{\bfw_h\cdot\nh}_{L^2(\Gah)}$ of the form
\begin{equation}
    \norm{\bfw_h\cdot\nh}_{L^2(\Gah)} \leq ch\norm{\wh}_{\ah},
\end{equation}
so that the $H^1$ improved coercivity bound \eqref{eq: improved h1-ah bound} holds. Notice, that this is not true in general and only holds in this specific $k_{\lambda} = k_u$ case when $\wh \in \bfV_h^{div}$, since we allow $\int_{\Gah} \xi_h(\bfw_h\cdot\nh) =0$ for any $\xi_h \in \Lambda_h = S^{k_{\lambda}}_{h,k_g} = S^{k_u}_{h,k_g}= V_h$. More specifically, 
\begin{equation*}
     \begin{aligned}
         \norm{\bfw_h\cdot\nh}_{L^2(\Gah)}^2 &= \int_{\Gah} (\bfw_h\cdot\nh)(\bfw_h\cdot\nh) = \int_{\Gah} (\bfw_h\cdot\bfn)(\bfw_h\cdot\nh) + \int_{\Gah} (\bfw_h\cdot(\nh-\bfn))(\bfw_h\cdot\nh)\\
         & \leq \int_{\Gah} \underbrace{\big(\bfw_h\cdot\bfn)-  \Ihz(\bfw_h\cdot\bfn)\big)}_{\int_{\Gah} \Ihz(\bfw_h\cdot\bfn)(\bfw_h\cdot\nh) =0} (\bfw_h\cdot\nh) + ch^{k_g}\norm{\bfw_h\cdot\nh}_{L^2(\Gah)}\norm{\bfw_h}_{L^2(\Gah)}
     \end{aligned}
 \end{equation*}
where we have used the geometric errors \eqref{eq: geometric errors 1} and $\Ihz(\cdot)\in V_h$ as our test function $\xi_h$. Now, by the super-approximation property \eqref{eq: super-approximation estimate 2} (it holds since $\wh \in \bfV_h^{div}\subset\bfV_h$ and $\Ihz(\cdot): L^2(\Gah) \to V_h$) 
\begin{equation}
    \begin{aligned}
       \norm{\bfw_h\cdot\nh}_{L^2(\Gah)}^2 \leq  ch\norm{\bfw_h\cdot\nh}_{L^2(\Gah)}\norm{\bfw_h}_{L^2(\Gah)},
    \end{aligned}
\end{equation}
which gives our desired bound.
\end{proof}

\begin{remark}
    Notice the  $H^1$ coercivity bound  $\norm{\bfw_h}^2_{H^1(\Gah)} \leq ch^{-2} \norm{\bfw_h}_{\ah}^2$ in \eqref{coercivity and Korn's inequality Lagrange}.
As we shall see in \cref{sec: the kl=ku-1 case} for $k_{\lambda} = k_u-1$, this forces us to lose one geometric order of approximation in our error bounds. Therefore, we need to use super-parametric surface  finite elements to observe optimal convergence; see numerical experiments \cref{section: simple sphere experiment}. 

In the $k_\lambda = k_u$ case \cref{sec: the kl=ku case}, due to the improved $H^1$ coercivity bound \eqref{eq: improved h1-ah bound} this geometric error loss vanishes, so we can simply use iso-parametric finite elements; see \cref{example: Dz-Ell}.
\end{remark}

 \begin{lemma}[Interpolation in the Lagrange energy norm]\label{lemma: energy norm interpolant lagrange}
Let $\bfv \in H^{k+1}(\Ga)^3$, then the following interpolation estimate holds, for constant $c$ independent of the mesh parameter $h$
\begin{equation}\label{eq: interpolation vh ah}
    \norm{\bfv^{-\ell} - \Ihz\bfv^{-\ell}}_{\ah} \leq ch^{k}\norm{\bfv}_{H^{k+1}(\Ga)},
\end{equation}
for $k \leq k_u$. Furthermore, we have the stability estimate
    \begin{equation}
    \label{eq: interpolation stability vh ah}
        \norm{\Ihz (\bfv)}_{\ah} \leq c \norm{\Ihz (\bfv)}_{H^1(\Gah)} \leq c\norm{\bfv}_{H^1(\Ga)}.
    \end{equation}
\end{lemma}
\begin{proof}
   \ The estimate \eqref{eq: interpolation vh ah} is immediate if we combine the fact that $\norm{\bfv^{-\ell} - \Ihz\bfv^{-\ell}}_{\ah} \leq c  \norm{\bfv^{-\ell} - \Ihz\bfv^{-\ell}}_{H^1(\Gah)}$
    % \begin{equation}
    %     \begin{aligned}
    %          \norm{\bfv^{-\ell} - \Ihz\bfv^{-\ell}}_{\ah} \leq c  \norm{\bfv^{-\ell} - \Ihz\bfv^{-\ell}}_{H^1(\Gah)},
    %     \end{aligned}
    % \end{equation}
with the $H^1$-estimate for the Scott-Zhang interpolant \eqref{eq: Scott-Zhang interpolant}. The stability estimate follows similarly using \eqref{eq: stability of Scott-Zhang interpolant}.
\end{proof}

  \subsection{The discrete Lagrange approximation {\bf(LhP)}}
The Lagrange finite element approximation is based on the continuous homogeneous Problem \eqref{weak lagrange hom}, where the tangentiality condition is enforced weakly via the use of a Lagrange multiplier. This formulation has the advantage of being really easy to implement.
The discrete Lagrange finite element method reads as:\\

% thus the tangential condition, i.e. that the normal component vanishes, is enforced weakly via the use of a Lagrange multiplier. This formulation  has the advantage of avoiding calculating geometric quantities, like the Mean and Gaussian curvature tensor, and
% is easy to implement.  The discrete Lagrange finite element method reads as:\\

\noindent {\bf(LhP)} Given $\bff_h\in L^2(\Gah)$, determine $(\uh,\{\ph,\lh\}) \in \bfV_h \times Q_h \times \Lambda_h$ such that,
\begin{align}
\begin{cases}
\tag{LhP}
    \label{weak lagrange discrete}
        \ah(\uh,\vh) + \bhtil(\vh,\{\ph,\lh\}) &= (\bff_h,\vh) \ \ \text{for all } \vh \in \bfV_h,\\
        \bhtil(\uh,\{\qh,\xi_h\})&=0 \ \ \ \ \ \ \ \text{ for all } \{\qh,\xi_h\}\in Q_h \times \Lambda_h. 
    \end{cases}
\end{align}

\begin{theorem}[Well-posedness of the discrete Lagrange problem] \label{Lemma: Well-posedness of discrete Lagrange problem}
There exists a unique solution $(\uh,\{\ph,\lh\})$ that satisfies \eqref{weak lagrange discrete} for all $(\vh ,\{\qh,\xi_h\}) \in  \bfV_h \times Q_h \times \Lambda_h$. Moreover, we have the following a priori estimates: 
\begin{equation}\label{eq: Well-posedness of discrete Lagrange problem}
    \norm{\uh}_{\ah} + \norm{\{\ph,\lh\}}_{L^2(\Gah)} \leq c \norm{\bff_h}_{L^2(\Gah)}.
\end{equation}
\end{theorem}
\begin{proof}
This is a direct consequence of the inf-sup condition Lemma \ref{Lemma: Discrete inf-sup condition Gah Lagrange} presented in the following \Cref{Section: inf-sup lagrange}, the bounds in \cref{Lemma: discrete bounds and coercivity results}, and simple calculations.
\end{proof}

\subsection{The $L^2 \times L^2 $ inf-sup condition }\label{Section: inf-sup lagrange}
\begin{lemma}[Discrete Lagrange inf-sup condition]
\label{Lemma: Discrete inf-sup condition Gah Lagrange}
  Assume a quasi-uniform triangulation of $\Gah$, then there exists a constant $\beta >0$ independent of the mesh parameter $h$ such that 
    \begin{equation}\label{eq: discrete inf-sup condition Gah Lagrange}
    c\norm{\{\qh,\lh\}}_{L^2(\Gah)}\leq \sup_{\vh \in \bfV_h} \frac{\bhtil(\vh,\{\qh,\lh\})}{\norm{\vh}_{\ah}} \ \ \ \forall \{\qh,\lh\} \in Q_h \times \Lambda_h.
    \end{equation}
\end{lemma}
 The proof of \cref{Lemma: Discrete inf-sup condition Gah Lagrange} is based on   ``Verf\"urth's trick'' and is divided into three steps. \emph{First}, we bound below the bilinear form $\bhtil$ by choosing an appropriate discrete velocity with the help of the continuous problem, \emph{secondly} we try to control the pressure gradient term that arises in the first step by once again choosing another appropriate velocity, this time using the geometry of the problem, and finally in the \emph{third} step we combine the previous two to get our \emph{inf-sup} condition.

\begin{remark}\label{remark: equivalence Vh1-Vh}
   By $\bfV_h^{(1)}, Q_h^{(1)}$ we denote the parametric finite element spaces $(S_{h,1}^{k_u})^3,  \ S_{h,1}^{k_{pr}}$ i.e. arbitrary order spaces on flat triangles. Without loss of generality, we suppose that triangles $\Ttilde \in \Galin$ are the reference triangles (we basically just omit a linear affine map), and use the notation $\Tilde{\bfv}_h,\Tilde q_h$ for functions in $ \bfV_h^{(1)}, \ Q_h^{(1)}$. Then according to Section \ref{FEMspaces} we can easily see that $\vh = \Tilde{\bfv}_h \circ F_T^{-1} \in \bfV_h$ as a function of $\Gah$, where now the map $F_T = \pi_{k_g}$, and $\pi_{k_g} : \Galin \to \Gah$ the $k_g$th-order Lagrange interpolation of the closest point projection $\pi$, that is, $\pi_{k_g} = \mathcal{\tilde{I}}_h^{k_g}\pi$, see \cite{Demlow2009, DziukElliott_acta}. The same relationship also holds for the discrete spaces $Q_h$ and $Q_h^{(1)}$.
\end{remark}

\subsubsection{\bf Step 1}\label{sec: Step 1}

 It is convenient to  introduce for $\qh \in Q_h$ the following norm and associated equivalences that hold due to the shape regularity of our mesh
:
\begin{equation}
    \begin{aligned}\label{eq: h norm lagrange}
        \norm{\qh}_{h} = \Big(\sum_{T\in \Th}h_T^2\norm{\nbgh\qh}_{L^2(T)}^2\Big)^{1/2} \sim h\norm{\nbgh\qh}_{L^2(\Gah)} \sim h\norm{\nbglin \Tilde \qh}_{L^2(\Galin)},
    \end{aligned}
\end{equation}
We begin by establishing that given  $\{\qh,\lh\} \in Q_h \times \Lambda_h$ there exists a $\vh \in \bfV_h$ such that 
\begin{equation}\label{eq: pre almost discrete inf-sup lagrange}
    \begin{aligned}
        \bhtil(\vh,\{\qh,\lh\}) \geq c\norm{\{\qh,\lh\}}_{L^2(\Gah)}(\norm{\{\qh,\lh\}}_{L^2(\Gah)} -c\norm{\qh}_{h}),
    \end{aligned}
\end{equation}
and 
\begin{equation}\label{eq: energy interpolation stability vh lagrange}
     \norm{\vh}_{\ah}  \leq c \norm{\{\qh,\lh\}}_{L^2(\Gah)}.
\end{equation}

\noindent From which  we may obtain the following inequality:
\begin{equation}
    \begin{aligned}
    \label{eq: almost discrete inf-sup lagrange}
        \sup_{\vh \in \bfV_h} \frac{\bhtil(\vh,\{\qh,\lh\})}{\norm{\vh}_{\ah}} &\geq c\norm{\{\qh,\lh\}}_{L^2(\Gah)} - c\norm{\qh}_h.
    \end{aligned}
\end{equation}
The last term  is treated  in  \underline{Step 2}.

To show \eqref{eq: pre almost discrete inf-sup lagrange} and \eqref{eq: energy interpolation stability vh lagrange}  we proceed as follows:

 \noindent Since  $\{\qh,\lh\} \in Q_h \times \Lambda_h$  we have  $\{\qhl,\lhl\} \in H^1(\Ga)$. Observing that   $m_q = -\frac{1}{|\Ga|}\int_{\Ga}\qhl \, \ds$  satisfies  $|m_q| \leq h^2\norm{\qhl}_{L^2(\Ga)}$ we set the new pressure variable to be $q_h' = \qhl + m_q$ so  that   $\int_{\Ga}q_h' \, \ds = 0$ and  $q_h' \in H^1(\Ga)\cap L^2_0(\Ga)$. Now we know, by the smoothness of $\Gamma$ and elliptic regularity, that there exists a  $\bfv_T \in \bfH_T^1$  such that  $\nabla_\Gamma \cdot \bfv_T =\qh'$ that satisfies,
\begin{equation}\label{eq: tangent inf-sup and H^1 L^2 inequality lagrange}
    \norm{\bfv_T}_{H^1(\Ga)} \leq c  \norm{\qhl + m_q}_{L^2(\Ga)}  \leq c  \norm{\qh}_{L^2(\Gah)},
\end{equation} 
and furthermore
$$  \int_\Ga \bfv_T \cdot \nbg \qhl \ds=  \norm{\qhl + m_q}^2_{L^2(\Ga)}\geq c\norm{\qhl}^2_{L^2(\Ga)}.$$
Finally setting $\bfv_* = \bfv_T + \lhl\bfng $ we see that  $\bfv_* \in  \bfH^1(\Ga) $ has been constructed  such that
\begin{equation}
\begin{aligned}\label{eq: tangent inf-sup and H^1 L^2 inequality lagrange 2}
    b^L(\bfv_*,\{\qhl,\lhl\}) 
    & = \int_\Ga \bfv_T \cdot \nbg \qhl \ds + \norm{\lhl}_{L^2(\Ga)}^2 \geq c\norm{\{\qhl,\lhl\}}_{L^2(\Ga)}^2, \text{ and }\\    
   \norm{\bfv_*}_{\bfV} &=  \norm{\bfv_T}_{H^1(\Ga)} + \norm{\lhl}_{L^2(\Ga)} \leq c  \norm{\{\qhl,\lhl\}}_{L^2(\Ga)}  \leq c  \norm{\{\qh,\lh\}}_{L^2(\Gah)}.
    \end{aligned}
\end{equation}

\noindent  Using the Scott-Zhang interpolant \eqref{eq: Scott-Zhang interpolant linear} we set $\vh = \Ihz(\bfv_*^{-\ell})\in \bfV_h$. The super approximation property \eqref{eq: super-approximation estimate 2} and stability  \eqref{eq: interpolation stability vh ah} then yield the following  stability estimate in the  energy norm :\begin{equation*}
\begin{aligned}
    \norm{\vh}_{\ah}= \norm{\Ihz(\bfv_*^{-\ell})}_{\ah} &\leq \norm{\Ihz(\bfv_T^{-\ell})}_{\ah} + \norm{\Ihz(\lh\bfn)}_{\ah}  \qquad \quad (\text{by } \eqref{eq: interpolation stability vh ah})\\
    &\leq \norm{\bfv_T^{-\ell}}_{H^1(\Gah)}+ \norm{\Ihz(\lh\bfn)}_{\ah}\\
    &\leq \norm{\qh}_{L^2(\Gah)} + \norm{\Ihz(\lh\bfn)}_{\ah} \qquad \qquad (\text{by } \eqref{eq: tangent inf-sup and H^1 L^2 inequality lagrange}),
\end{aligned}
\end{equation*}
where once again $\bfn = (\bfng)^{-\ell}$. It remains to bound the second term, in terms of the $L^2-$norm of the Lagrange multiplier $\lh$. The following calculations hold:
\begin{equation}
    \begin{aligned}\label{eq: Ilhn}
       \norm{\Ihz(\lh\bfn)}_{\ah} \leq \norm{\lh\bfn}_{\ah} + \norm{\Ihz(\lh\bfn)-\lh\bfn}_{\ah}.
    \end{aligned}
\end{equation}
For the first term on the right-hand side of \eqref{eq: Ilhn} we consider the definition of the energy norm, the calculation presented in \eqref{eq: split cov} and the geometric errors in \eqref{eq: geometric errors 2}  to obtain, for small enough $h$
\begin{equation}
    \begin{aligned}
        \norm{\lh\bfn}_{\ah} &\leq c \norm{\nbgcovh(\lh\bfn)}_{L^2(\Gah)} + \norm{\lh\bfn}_{L^2(\Gah)} \leq c \norm{\bfPh \nbgh(\lh \bfn)}_{L^2(\Gah)} +\norm{\lh\bfn}_{L^2(\Gah)}\\
        &\leq c \norm{\bfPh ( \lh \nbgh \bfn + \bfn \nbgh^T\lh)}_{L^2(\Gah)} +\norm{\lh\bfn}_{L^2(\Gah)} \\
        &\leq c\norm{\lh}_{L^2(\Gah)} + ch\norm{\lh}_{L^2(\Gah)}\leq c\norm{\lh}_{L^2(\Gah)}.
    \end{aligned}
\end{equation}
Regarding the second term of \eqref{eq: Ilhn}, we use the super-approximation property \eqref{eq: super-approximation estimate 1}, where instead of an inner product we have the vector $\lh \bfn$. The calculations for this result follow similarly as presented in \cite[Lemma 4.3]{hansbo2020analysis}, and thus by \eqref{eq: super-approximation estimate 1}, the stability estimate in the Lagrange energy norm \eqref{eq: interpolation stability vh ah} and the inverse estimate we may obtain
\begin{equation}
    \begin{aligned}\label{eq: tangent inf-sup and H^1 L^2 inequality lagrange this}
       \norm{\Ihz(\lh\bfn)-\lh\bfn}_{\ah} \leq \norm{\Ihz(\lh\bfn)-\lh\bfn}_{H^1(\Gah)} \leq c \norm{\bfn_{\Ga}}_{W_{\infty}^{k_{\lambda}+1}(\Ga)}\norm{\lh}_{L^2(\Gah)}.
    \end{aligned}
\end{equation} 
Combining these results, the energy estimate for the chosen velocity $\vh = \Ihz(\bfv_*^{-\ell})$ is \eqref{eq: energy interpolation stability vh lagrange}, that is:
\begin{equation}
     \norm{\vh}_{\ah} = \norm{\Ihz(\bfv_*^{-\ell})}_{\ah} \leq c \norm{\{\qh,\lh\}}_{L^2(\Gah)}.
\end{equation}
\noindent Now, considering the discrete bilinear form $\bhtil(\cdot,\{\cdot,\cdot\})$ add and subtract the previously chosen velocity $\bfv_* = \bfv_T + \lhl \bfng \in \bfV$, such that for $\vh = \Ihz(\bfv^{-\ell})$ we may obtain
\begin{equation}
    \begin{aligned}\label{eq: splitting bh lagrange}
        \bhtil(\vh,\{\qh,\lh\}) &= \int_{\Gah} \vh \nbgh \qh \dsh + \int_{\Gah}\lh (\vh \cdot \nh) \dsh \\
        &= \underbrace{\int_{\Gah} \bfv_*^{-\ell} \nbgh \qh \, \dsh}_{\bold{Term \, (I)}} + \underbrace{\int_{\Gah} (\vh - \bfv_*^{-\ell}) \nbgh \qh \, \dsh}_{\bold{Term \, (II)}}\\
        &+ \underbrace{\int_{\Gah}\lh (\bfv_*^{-\ell} \cdot \nh) \, \dsh}_{\bold{Term \, (III)}} + \underbrace{\int_{\Gah}\lh ((\vh - \bfv_*^{-\ell}) \cdot \nh) \, \dsh.}_{\bold{Term \, (IV)}}
    \end{aligned}
\end{equation}
\noindent We now need to bound the four terms below.

$\underline{\bold{Term \, (I)}}$: is calculated as followed
\begin{equation}\label{eq: Term (I) inf-sup}
\begin{aligned}
    \int_{\Gah} \bfv_*^{-\ell} \nbgh \qh \, \dsh  &= \int_{\Gah} \bfv_T^{-\ell} \cdot \nbgh\qh\, \dsh + \int_{\Gah} \lh \bfn \cdot \nbgh\qh\, \dsh \\
    &= \int_{\Ga}\bfv_T \cdot \nbg \qhl\, \dsh + \int_{\Ga}(\frac{1}{\mu_h} -1)\bfv \cdot \nbg \qhl \, \dsh\\
    &\ \ + \int_{\Ga}\frac{1}{\mu_h}\bfPg\bfv \cdot \big( \bfPh(I-d\bfH) -I\big)\bfPg\nbg\qhl \, \dsh  + \int_{\Gah} \lh \bfn \cdot \nbgh\qh\, \dsh\\
        &= \int_{\Ga}\bfv_T \cdot \nbg \qhl \, \dsh + \int_{\Ga}(\frac{1}{\mu_h} -1)\bfv \cdot \nbg \qhl \, \dsh \\
        &\ \ + \int_{\Ga}\frac{1}{\mu_h}\big( [\bfPg\bfPh\bfPg - \bfPg] - \bfP_h d\bfH\bfP \big)\bfv \cdot \nbg\qhl \, \dsh + \int_{\Gah} \lh \bfn \cdot \nbgh\qh \, \dsh\\
        &\geq \int_{\Ga}\bfv_T \cdot \nbg \qhl \, \ds - ch^{k_g-1}(\norm{\bfv_T}_{L^2(\Ga)} + \norm{\lhl}_{L^2(\Ga)})\norm{\qh}_h,
    \end{aligned}
\end{equation}
where in the last inequality we used the geometric errors \eqref{eq: geometric errors 2}.

$\underline{\bold{Term \ (II)}}$: we split the vector field $\bfv_*$ into the normal and tangent component, $\bfv_* = \bfv_T + \lhl\bfng$, use the interpolation estimate \eqref{eq: Scott-Zhang interpolant linear} and the super-approximation \eqref{eq: super-approximation estimate 2}, where again instead of an inner product of the form $\bfn \cdot \vh$ we have a vector $\lh \bfn$, to obtain the estimate,
\begin{equation}
    \begin{aligned}\label{eq: Term (II) inf-sup}
        &\int_{\Gah} (\vh - \bfv_*^{-\ell}) \nbgh \qh \, \dsh \\
        &\ = \int_{\Gah} (\Ihz(\bfv_T^{-\ell}) - \bfv_T^{-\ell}) \nbgh \qh \, \dsh 
        + \int_{\Gah} (\Ihz(\lh \bfn) - \lh \bfn) \nbgh \qh \, \dsh \\
        &\ \geq -\sum_{T\in \Th}\norm{\nbgh \qh}_{L^2(T)} h_T \norm{\bfv_T^{-\ell}}_{H^1(\omega_T)} - \sum_{T\in \Th}\norm{\nbgh \qh}_{L^2(T)} h_T \norm{\lh\bfn}_{L^2(\omega_T)}\\
        &\ \geq -c (\norm{\bfv_T}_{H^1(\Ga)} + \norm{\lhl}_{L^2(\Ga)})\norm{\qh}_{h}.
    \end{aligned}
\end{equation}

$\underline{\bold{Term \ (III)}}$: we add and subtract the term $\bfv_*^{-\ell} \cdot \bfn$, where $\bfn = (\bfng)^{-\ell}$, and coupled with the geometric errors \eqref{eq: geometric errors 1} and the H\"older's inequality we prove that

\begin{equation}
    \begin{aligned}\label{eq: Term (III) inf-sup}
        \int_{\Gah}\lh (\bfv_*^{-\ell} \cdot \nh) \, \dsh &= \int_{\Gah}\lh (\bfv_*^{-\ell} \cdot \bfn) \, \dsh + \int_{\Gah}\lh (\bfv_*^{-\ell} \cdot (\nh-\bfn)) \, \dsh \\
        &\geq \int_{\Gah}\lh^2 \, \dsh -ch \norm{\lh}_{L^2(\Gah)}\norm{\bfv_*^{-\ell}}_{L^2(\Gah)}\\
        &\geq \int_{\Ga}(\lhl)^2 \, \ds + \int_{\Ga}(\frac{1}{\mu_h}-1)(\lhl)^2 \, \ds   -ch \norm{\lh}_{L^2(\Gah)}\norm{\bfv_*}_{L^2(\Ga)}\\
        &\geq (1-ch^{k_g})\norm{\lhl}^2_{L^2(\Ga)} -ch \norm{\lh}_{L^2(\Gah)}\norm{\bfv_*}_{L^2(\Ga)}.
    \end{aligned}
\end{equation}

$\underline{\bold{Term \ (IV)}}$: again the Scott-Zhang interpolant \eqref{eq: Scott-Zhang interpolant linear}, the super-approximation \eqref{eq: super-approximation estimate 2} and the norm equivalence $\eqref{eq: norm equivalence}$ leads to 
\begin{equation}
    \begin{aligned}\label{eq: Term (IV) inf-sup}
        \int_{\Gah}\lh ((\vh - \bfv_*^{-\ell}) \cdot \nh) \, \dsh 
        &= \int_{\Gah} \lh((\Ihz(\bfv_T^{-\ell}) - \bfv_T^{-\ell})\cdot \nh ) \, \dsh  + \int_{\Gah} \lh ((\Ihz(\lh \bfn) - \lh \bfn)\cdot \nh) \, \dsh \\
        &\geq -ch\norm{\lh}_{L^2(\Gah)} \norm{\bfv_T}_{H^1(\Ga)} - \tilde{c}h\norm{\lhl}^2_{L^2(\Ga)}.
    \end{aligned}
\end{equation}

\noindent Combing now these four terms \eqref{eq: Term (I) inf-sup}, \eqref{eq: Term (II) inf-sup}, \eqref{eq: Term (III) inf-sup}, \eqref{eq: Term (IV) inf-sup} into \eqref{eq: splitting bh lagrange} and taking $h$ sufficiently small we have, 
\begin{equation*}
    \begin{aligned}
        \bhtil(\vh,\{\qh,\lh\}) & \geq \int_{\Ga} \bfv_T \nbg \qhl \, \ds + (1-\tilde{c}h) \norm{\lhl}^2_{L^2(\Ga)} -c (\norm{\bfv_T}_{H^1(\Ga)}+  \norm{\lhl}_{L^2(\Ga)})\norm{\qh}_{h} \\ 
        & \quad -ch\norm{\lh}_{L^2(\Gah)} \norm{\bfv_T}_{H^1(\Ga)}\\
         &\geq \tilde{c}(\int_{\Ga} \bfv_T \nbg \qhl \, \ds + \norm{\lhl}^2_{L^2(\Ga)}) -c (\norm{\bfv_T}_{H^1(\Ga)}+  \norm{\lhl}_{L^2(\Ga)})\norm{\qh}_{h} \\ 
         &\quad -ch\norm{\{\qhl,\lhl\}}_{L^2(\Ga)}^2,
    \end{aligned}
\end{equation*}
by \eqref{eq: tangent inf-sup and H^1 L^2 inequality lagrange}  and Young's inequality.
The equations \eqref{eq: tangent inf-sup and H^1 L^2 inequality lagrange 2} and the norm equivalence \eqref{eq: norm equivalence}, then readily imply \eqref{eq: pre almost discrete inf-sup lagrange}.

\subsubsection{\bf Step 2}\label{sec: Step 2}
We still need to control the tangential gradient of the pressure, i.e. the term $\norm{\qh}_h$ \eqref{eq: h norm lagrange}. For this it is sufficient to  prove the following inequality
\begin{equation}
    \begin{aligned}
    \label{eq: control of tangential gradient lagrane}
        \sup_{\vh \in \bfV_h} \frac{\bh(\vh,\qh)}{\norm{\vh}_{\ah}}\geq \beta\Big(\sum_{T\in \Thlin}h_T^2\norm{\nbgh\qh}_{L^2(T)}^2\Big)^{1/2} \geq \beta h \norm{\nbgh\qh}_{L^2(\Gah)}. 
    \end{aligned}
\end{equation}
for some mesh-independent constant $\beta$.

Again for any given $\qh\in Q_h$ we construct a $\vh\in\bfV_h$, similarly to \cite{jankuhn2021Higherror,reusken2024analysis}, such that \begin{equation}
    \begin{aligned}
        \bh(\vh,\qh) \geq  ch^2\norm{\nbgh\qh}_{L^2(\Gah)}^2 \geq \beta h\norm{\vh}_{\ah}\norm{\nbgh\qh}_{L^2(\Gah)} \geq \beta \norm{\vh}_{\ah} \norm{\qh}_h,
    \end{aligned}
\end{equation}
where $\beta$ just a mesh independent constant.

We begin by considering the piecewise planar surface $\Galin$ together with the  mapping $F_T : \Galin \to \Gah$ and \Cref{remark: equivalence Vh1-Vh}. Set $\shlin$ to be the unit vector in a fixed direction across the  edge $E \in \mathcal{E}_h^{(1)}$. % between the two elements $T_1, \, T_2$. Now for each pair of neighbouring elements we consider 
Let the midpoint of  each  edge $E \in \mathcal{E}_h^{(1)}$ be denoted by  $\bfz_M$ to which we attach the  piecewise quadratic surface finite element basis function $\varphi_E$ such that
\begin{align*}
    \begin{cases}
    \varphi_E(\bfz) =1 \qquad \text{for } \bfz = \bfz_M, \\
    \varphi_E(\bfz) =0 \qquad \text{for the rest of the nodal points }  \bfz \text{ of the neighboring triangles}.
    \end{cases}
\end{align*}
For $\qhtilde = \qh \circ F_T$ we set  the discrete velocity  $\vhtilde$ on $\Galin$ to be
\begin{equation}
    \begin{aligned}
    \label{eq: choosing vh}
        \vhtilde = \sum_{E \in \mathcal{E}_h} h_E^2\varphi_E |\shlin \cdot \nbglin \qhtilde| \shlin,
    \end{aligned}
\end{equation}
and note that it is continuous, since $\shlin$ continuous across the edges of the triangle, and also piecewise polynomial of degree at most $k$ so 
$\vhtilde \in \bfV_h^{(1)}$ (cf. \cite{reusken2024analysis}).  Setting $\vh = \vhtilde \circ F_T^{-1}$ 
and  $\qh  = \qhtilde \circ F_T^{-1} $ we have the following important inequality 
\begin{equation}\label{eq: vh estimate inf-sup}
\begin{aligned}
    \norm{\vh}_{\ah} \leq \norm{\vh}_{H^1(\Gah)} \leq \norm{\vhtilde}_{H^1(\Galin)} &\leq  c\Big(\sum_{T\in \Thlin}h_T^{-2}\norm{\vhtilde}_{L^2(T)}^2\Big)^{1/2} \\
    \leq c \Big(\sum_{T\in \Thlin} h_T^2 \norm{\nbglin \qhtilde}^2_{L^2(T)}\Big)^{1/2}
    &\leq ch\norm{\nbglin \qhtilde}_{L^2(\Galin)} \leq ch\norm{\nbgh \qh}_{L^2(\Gah)},
    \end{aligned}
\end{equation}
due to \eqref{eq: h norm lagrange}, and we can also readily see for $\vhtilde$ that
\begin{equation}\label{eq: vh sh bound ah}
    \begin{aligned}
        \norm{\vhtilde}_{\ah} \leq c \Big(\sum_{T\in \Thlin} h_T^2 \norm{\nbglin \qhtilde}^2_{L^2(T)}\Big)^{1/2} \leq ch\norm{\nbglin \qhtilde}_{L^2(\Galin)}.
    \end{aligned}
\end{equation}
and
\begin{equation}
    \begin{aligned}\label{eq: vh sh l2 bound}
        \norm{\vh}_{L^2(\Gah)}\sim \norm{\vhtilde}_{L^2(\Galin)}  \leq ch^2\norm{\nbglin \qhtilde}_{L^2(\Galin)}.
    \end{aligned}
\end{equation}
We now consider the bilinear form $\bhlin(\cdot,\cdot)$ as in \eqref{eq: lagrange discrete bilinear formm b} for the planar case, i.e. $\Gah = \Galin$.
By substituting $\vhtilde$ in $\bhlin(\cdot,\cdot)$ with $\qhtilde  = \qh \circ F_T $ and using the fact that $0 \leq \varphi_E \leq 1$  we obtain: 
\begin{equation}
    \begin{aligned}\label{eq: vh sh bound L2}
      \bhlin(\vhtilde,\qhtilde) &= \sum_{T\in \Thlin}\int_T \sum_{E\in \mathcal{E}(T)} h_E^2 \varphi_E |\shlin \cdot \nbglin \qhtilde|^2 \geq   \sum_{T\in \Thlin} \int_T  \sum_{E\in \mathcal{E}(T)}  h_E^2  |\shlin \cdot \nbglin \qhtilde|^2 \\
      &\geq\sum_{T\in \Thlin} h_T^2    \int_T  \sum_{E\in \mathcal{E}(T)} |\shlin \cdot \nbglin \qhtilde|^2,
    \end{aligned}
\end{equation}
where for the last inequality we used the fact that $h_E \geq c h_T$ for a quasi-uniform triangulation (Section \ref{sec: Triangulated surfaces}).
Now since we have assumed a closed surface then the assumption as seen in \cite{ern2004theory} are satisfied, thus we have that $|\shlin \cdot \nbglin \qhtilde|$ controls $|\nbglin \qhtilde|$ and so
\begin{equation}
    \begin{aligned}
    \label{eq: just before control of tangenital gradient}
      \bhlin(\vhtilde,\qhtilde) \geq c\sum_{T\in \Thlin} h_T^2 \norm{\nbglin \qhtilde}_{L^2(T)}^2.
    \end{aligned}
\end{equation}
Currently we have an inequality w.r.t. $\bhlin(\vhtilde,\qhtilde)$ for the previously chosen $\vhtilde$. What we need is \eqref{eq: control of tangential gradient lagrane}. To arrive to \eqref{eq: control of tangential gradient lagrane} we will use the above equation \eqref{eq: just before control of tangenital gradient} add and subtract $\bh(\vh,\qh)$ to get
\begin{equation}
    \begin{aligned}\label{eq: before final inf-sup}
        \bh(\vh,\qh) + \underbrace{\bhlin(\vhtilde,\qhtilde) - \bh(\vh,\qh)}_{\bold{Term \, (I_2)}}\geq c\sum_{T\in \Thlin} h_T^2 \norm{\nbglin \qhtilde}_{L^2(T)}^2 \geq ch^2\norm{\nbgh\qh}_{L^2(\Gah)}^2
    \end{aligned}
\end{equation}
due to \eqref{eq: h norm lagrange}.
We would like to bound  $\bold{Term \, (I_2)}$. 

$\underline{ \bold{Term \, (I_2)}} : $ This is possible by recalling \cite[Section 2.4]{Demlow2009}. Then for the  gradient on the planar triangulation $\Galin$ we have that,
\begin{equation*}
    \nbglin \qhtilde (\Tilde{x}) = \Bhkg \nbgh \qh(F_T(\Tilde{x})), \ \ \Tilde{x} \in \Galin,
\end{equation*}
	where $F_T(\cdot)$ the map in \eqref{eq: tildeT to T map}, i.e. $F_T(\cdot) =\mathcal{\tilde{I}}_h^{k_g}\pi(\cdot) = \pi_{k_g}(\cdot)$, $\qhtilde(\tilde{x}) = \qh \circ F_T(\tilde{x})$ and also $\Bhkg = \nb \pi_{k_g} + \nh \otimes \nh$. Then by lifting to $\Gah$ we obtain the following 
\begin{equation}
\begin{aligned}\label{eq: b1 to b large}
    &\int_{\Galin} \vhtilde \cdot \nbglin \qhtilde \dslin \leq \int_{\Gah} \vh \cdot \Bhkg \nbgh \qh (\frac{1}{\mu_{h,1}^k} -1 )\dsh + \int_{\Gah} \vh \cdot (\Bhkg- \Bhg )\bfPh \nbgh \qh \dsh \\ 
    &\quad \ + \int_{\Gah} \vh \cdot (\Bhg - \bfPh) \nbgh \qh \dsh 
     + \int_{\Gah} \vh \cdot \nbgh \qh \dsh \\
    &\leq \int_{\Gah} \vh \cdot \nbgh \qh \dsh +ch\norm{\vh}_{L^2(\Gah)}\norm{\nbgh\qh}_{L^2(\Gah)} + ch^{2}\norm{\vh}_{L^2(\Gah)}\norm{\nbgh\qh}_{L^2(\Gah)},
    \end{aligned}
\end{equation}
where the first term in the first step was bounded by the estimate $|1 - \frac{1}{\mu_{h,1}^k}| \leq ch^2$ from \eqref{eq: muhk1 estimate}, as well as the boundedness of $\Bhkg$, as $F_T(\cdot) = \pi_{k_g}(\cdot)$ is the Lagrange interpolant of the closest point projection $p$ and thus $\norm{\pi_{k_g}}_{W^m_{\infty}(\Ttilde)} \leq c \norm{\pi}_{W^m_{\infty}(\Ttilde)}$ for $\Ttilde \in \Thlin$ and small enough $h$. For the second estimate, we note from \cite[Section 2.4]{Demlow2009} and some easy calculations that $\norm{(\Bhkg- \Bhg)\bfPh}_{L_{\infty}} \leq c h$. While, the third term can be estimated similarly to \eqref{eq: Bh estimates} where $\norm{\Bhg- \bfPh}_{L_{\infty}(\Gah)} \leq c h^{k_g}$, thus concluding the result. 

All in all we obtain from \eqref{eq: b1 to b large} and \eqref{eq: vh sh l2 bound} that
\begin{equation}
    \begin{aligned}\label{eq: galin gah diff h norm}
       \bold{Term \, (I_2)} =  \int_{\Galin} \vhtilde \cdot \nbglin \qhtilde \dslin - \int_{\Gah} \vh \cdot \nbgh \qh \dsh \leq ch^3\norm{\nbgh\qh}_{L^2(\Gah)}^2.
    \end{aligned}
\end{equation}
Plugging \eqref{eq: galin gah diff h norm} into \eqref{eq: before final inf-sup} and using \eqref{eq: vh sh bound ah} we have for sufficiently small $h$ 
\begin{equation}
    \begin{aligned}
        \bh(\vh,\qh) \geq  ch^2\norm{\nbgh\qh}_{L^2(\Gah)}^2 \geq \beta h\norm{\vh}_{\ah}\norm{\nbgh\qh}_{L^2(\Gah)} \geq \beta \norm{\vh}_{\ah} \norm{\qh}_h,
    \end{aligned}
\end{equation}
where $\beta$ is just a new mesh independent constant. This proves \eqref{eq: control of tangential gradient lagrane}.

\subsubsection{ \bf Step 3}\label{sec: Step 3}
Finally combining \eqref{eq: control of tangential gradient lagrane} and \eqref{eq: almost discrete inf-sup lagrange} we get the following for any $\{\qh,\lh\} \in Q_h\times \Lambda_h$
\begin{equation}\label{eq: step 3 infsup}
   \frac{c}{\beta}\sup_{\vh \in \bfV_h} \frac{\bh(\vh,\qh)}{\norm{\vh}_{\ah}} + \sup_{\vh \in \bfV_h} \frac{\bhtil(\vh,\{\qh,\lh\})}{\norm{\vh}_{\ah}} \geq c\norm{\{\qh,\lh\}}_{L^2(\Gah)},
\end{equation}
 Focusing on the left hand-side we just notice that for any $\{\qh,\lh\} \in Q_h\times \Lambda_h$ 
\begin{equation*}
    \sup_{\vh \in \bfV_h} \frac{\bh(\vh,\qh)}{\norm{\vh}_{\ah}} 
 \leq \sup_{\vh \in \bfV_h} \frac{\bhtil(\vh,\{\qh,\lh\})}{\norm{\vh}_{\ah}}.
 \end{equation*}
Then, going back to \eqref{eq: step 3 infsup} and factoring out a new constant $c$ independent of $h$, we get our \emph{inf-sup} condition \eqref{eq: discrete inf-sup condition Gah Lagrange} for any $\{\qh,\lh\} \in Q_h\times \Lambda_h$:
\begin{equation}
    \begin{aligned}
         \sup_{\vh \in \bfV_h} \frac{|\bhtil(\vh,\{\qh,\lh\})|}{\norm{\vh}_{\ah}} \geq c\norm{\{\qh,\lh\}}_{L^2(\Gah)}, %\quad \text{for all } \{\qh, \lh\} \in  Q_h \times \Lambda_h.
    \end{aligned}
\end{equation}

\subsection{The $L^2 \times H_h^{-1}$ inf-sup condition }\label{Section: inf-sup lagrange L^2 H-1}
\begin{lemma}[$L^2\times H_h^{-1}$ Discrete Lagrange inf-sup condition]\label{lemma: L^2 H^{-1} discrete inf-sup condition Gah Lagrange}
Assume a quasi-uniform triangulation of $\Gah$, then there exists a constant $\beta >0$ independent of the mesh parameter $h$ such that 
\begin{equation}\label{eq: L^2 H^{-1} discrete inf-sup condition Gah Lagrange}
    c\norm{\{\qh,\lh\}}_{L^2(\Gah)\times H_h^{-1}(\Gah)}\leq \sup_{\vh \in \bfV_h} \frac{\bhtil(\vh,\{\qh,\lh\})}{\norm{\vh}_{H^1(\Gah)}} \ \ \ \forall \{\qh,\lh\} \in Q_h \times \Lambda_h.
\end{equation}
\end{lemma}
\begin{proof}
This proof follows very similarly to the proof of the \emph{discrete inf-sup condition} \cref{Lemma: Discrete inf-sup condition Gah Lagrange}  and mimics that of the \emph{continuous inf-sup} in \cref{lemma: inf-sup cont lagrange}. For that reason, we will try to remain brief. More specifically, we will focus on Step 1; \cref{sec: Step 1}, as it is there that the main changes occur. Similarly to the inequality \eqref{eq: almost discrete inf-sup lagrange} of \cref{sec: Step 1}, we now instead want to prove the following
\begin{equation}
    \begin{aligned}
    \label{eq: almost discrete inf-sup lagrange L^2 H-1}
        \sup_{\vh \in \bfV_h} \frac{\bhtil(\vh,\{\qh,\lh\})}{\norm{\vh}_{H^1(\Gah)}} &\geq c\norm{\{\qh,\lh\}}_{L^2(\Gah)\times H_h^{-1}(\Gah)} - c\norm{\qh}_h.
    \end{aligned}
\end{equation}
Then we can readily see that the rest of the steps in \cref{sec: Step 2}  and  \cref{sec: Step 3} follow basically unchanged (in \cref{sec: Step 2}, eq. \eqref{eq: vh estimate inf-sup}, we just consider the $H^1$-norm instead of the energy norm, and the rest follow suit), and therefore we obtain our new \emph{discrete inf-sup condition} \eqref{eq: L^2 H^{-1} discrete inf-sup condition Gah Lagrange}.
    
So, let us prove \eqref{eq: almost discrete inf-sup lagrange L^2 H-1}. As established in \cref{sec: Step 1} let $\{\qh,\lh\}\in Q_h\times \Lambda_h$ given. Now we consider the well-posed discrete Riesz map $\mathcal{R}^h_{H^1}(\lh): H_h^{-1}(\Gah)\to \Lambda_h$ such that
\begin{align}
    \label{eq: discrete Riesz map bound}
    \norm{\mathcal{R}^h_{H^1}(\lh)}_{H^1(\Gah)} &\leq  \norm{\lh}_{H_h^{-1}(\Gah)},\\
    \label{eq: discrete Riesz map}
    <\lh,\mathcal{R}^h_{H^1}(\lh)>_{H_h^{-1},H^1} &= \norm{\mathcal{R}^h_{H^1}(\lh)}_{H^1(\Gah)}^2 = \norm{\lh}_{H_h^{-1}(\Gah)}^2,
\end{align}
where, since $\lh \in \Lambda_h$,
\begin{equation}\label{eq: H^-1h definition 2}
    \norm{\lh}_{H_h^{-1}(\Gah)}^2 = \sup_{\xi_h\in \Lambda_h}\frac{<\lh,\xi_h>_{H^{-1},H^1}}{\norm{\xi_h}_{H^1(\Gah)}} = \sup_{\xi_h\in \Lambda_h}\frac{(\lh,\xi_h)_{L^2(\Gah)}}{\norm{\xi_h}_{H^1(\Gah)}}.
\end{equation}

This time we choose velocity $\bfv_* = \bfv_T + (\mathcal{R}^h_{H^1}\lh)^{\ell}\bfng$. Then it clear by definition \eqref{eq: discrete Riesz map} that instead of \eqref{eq: tangent inf-sup and H^1 L^2 inequality lagrange 2} we have
\begin{equation}\label{eq: new cont inf sup H-1}
        b^L(\bfv_*,\{\qhl,\lhl\}) \geq c\norm{\{\qh,\lh\}}_{L^2(\Gah)\times H_h^{-1}(\Gah)} \text{ and } \norm{\bfv_*}_{H^1(\Ga)} \leq \norm{\{\qh,\lh\}}_{L^2(\Gah)\times H_h^{-1}(\Gah)}.
\end{equation}
Considering the vector-valued Scott-Zhang interpolant $\Ihz(\cdot)\in \bfV_h$, due to the interpolation stability \eqref{eq: Scott-Zhang interpolant} and bound \eqref{eq: discrete Riesz map bound}, instead of \eqref{eq: energy interpolation stability vh lagrange} we have for $\vh =  \Ihz(\bfv_*^{-\ell}) \in \bfV_h$ the following bound
    \begin{equation}\label{eq: energy interpolation stability vh lagrange new}
    \norm{\vh}_{H^1(\Gah)} = \norm{\Ihz(\bfv_*^{-\ell})}_{H^1(\Gah)}\leq \norm{\qh}_{L^2(\Gah)} + \norm{\mathcal{R}_{H^1}\lh}_{H^1(\Gah)} \leq  c\norm{\{\qh,\lh\}}_{L^2(\Gah)\times H_h^{-1}(\Gah)}.
        \end{equation}
   Now we are left to prove \eqref{eq: almost discrete inf-sup lagrange L^2 H-1}. For that, we go back to \eqref{eq: splitting bh lagrange} and see that we just need to bound each $\underline{\bold{Term}}$ appropriately, again.
   
 $\underline{\bold{Term \ (I)}}$ and $\underline{\bold{Term \ (II)}}$   are bounded as in \cref{sec: Step 1} eq. \eqref{eq: Term (I) inf-sup}, \eqref{eq: Term (II) inf-sup} where have replaced $\lh$ with the Riesz map $\mathcal{R}^h_{H^1}(\lh)$. Therefore using the bounds \eqref{eq: discrete Riesz map bound}, \eqref{eq: new cont inf sup H-1} and the Scott-Zhang super-approximation estimate \eqref{eq: super-approximation estimate 2} \big(that is $\norm{\Ihz(\mathcal{R}^h_{H^1}(\lh))-\mathcal{R}^h_{H^1}(\lh)}_{L^2(\Gah)} \leq ch\norm{\mathcal{R}^h_{H^1}(\lh)}_{L^2(\Gah)}$\big) we can see that
   \begin{equation}
     \begin{aligned}\label{eq: Term (I) (II) inf-sup new}
      \underline{\bold{Term \ (I)}}, \underline{\bold{Term \ (II)}} &\geq \int_{\Ga}\bfv_T \cdot \nbg \qhl \, \ds - c(\norm{\bfv_T^{-\ell}}_{H^1(\Gah)} + \norm{\lh}_{H_h^{-1}(\Gah)})\norm{\qh}_{h}\\
      &\geq \int_{\Ga}\bfv_T \cdot \nbg \qhl \, \ds -c\norm{\{\qh,\lh\}}_{L^2(\Gah)\times H_h^{-1}(\Gah)}\norm{\qh}_{h},
      \end{aligned}
   \end{equation}
   Likewise, for $\underline{\bold{Term \ (III)}}$ in \eqref{eq: Term (III) inf-sup} we see that with the help of the Riesz map \eqref{eq: discrete Riesz map}, and \eqref{eq: new cont inf sup H-1} 
   \begin{equation}
   \begin{aligned}\label{eq: Term (III) inf-sup new}
      \underline{\bold{Term \ (III)}} &\geq \int_{\Gah}\lh \mathcal{R}^h_{H^1}(\lh) \dsh -ch^{k_g}\norm{\lh}_{L^2(\Gah)}\norm{\bfv_*^{-\ell}}_{L^2(\Gah)}\\
      &\geq \norm{\lh}_{H_h^{-1}(\Gah)}^2-ch^{k_g-1}\norm{\{\qh,\lh\}}_{L^2(\Gah)\times H_h^{-1}(\Gah)}\sup_{\vh \in \bfV_h} \frac{\bhtil(\vh,\{\qh,\lh\})}{\norm{\vh}_{H^1(\Gah)}},
      \end{aligned}
   \end{equation}
where we used the second bound in \eqref{eq: new cont inf sup H-1} and the previously proven \emph{ $L^2\times L^2$ discrete inf-sup} \cref{Lemma: Discrete inf-sup condition Gah Lagrange} to see that:
\begin{equation*}
    \begin{aligned}
        c\norm{\lh}_{L^2(\Gah)}\leq \sup_{\vh \in \bfV_h} \frac{\bhtil(\vh,\{\qh,\lh\})}{\norm{\vh}_{\ah}} \overset{\eqref{coercivity and Korn's inequality Lagrange}}{\leq} \sup_{\vh \in \bfV_h} \frac{\bhtil(\vh,\{\qh,\lh\})}{h\norm{\vh}_{H^1(\Gah)}},
    \end{aligned}
\end{equation*}
and thus $ch\norm{\lh}_{L^2(\Gah)} \leq \sup_{\vh \in \bfV_h} \frac{\bhtil(\vh,\{\qh,\lh\})}{\norm{\vh}_{H^1(\Gah)}}$.

%where we used \eqref{eq: new cont inf sup H-1} and \eqref{eq: energy interpolation stability vh lagrange new}. %the bound $\norm{\lhl}_{H^{-1}(\Ga)} \leq \norm{\lh}_{H^{-1}(\Gah)}$.
Finally, we have to bound $\underline{\bold{Term \ (IV)}}$ in \eqref{eq: Term (IV) inf-sup}. Using again the Scott-Zhang interpolation estimate \eqref{eq: Scott-Zhang interpolant} we see that
\begin{equation}
   \begin{aligned}\label{eq: Term (IV) inf-sup new}
      \underline{\bold{Term \ (IV)}} &\geq -ch\norm{\lh}_{L^2(\Gah)}\norm{\bfv_*^{-\ell}}_{H^1(\Gah)}-ch\norm{\lh}_{L^2(\Gah)}\norm{ \mathcal{R}^h_{H^1}(\lh)}_{H^1(\Gah)}\\
      &\geq -c\norm{\{\qh,\lh\}}_{L^2(\Gah)\times H_h^{-1}(\Gah)}\sup_{\vh \in \bfV_h} \frac{\bhtil(\vh,\{\qh,\lh\})}{\norm{\vh}_{H^1(\Gah)}},
      \end{aligned}
   \end{equation}
where we used \eqref{eq: discrete Riesz map bound}, the second bound in \eqref{eq: new cont inf sup H-1} and the fact that $ch\norm{\lh}_{L^2(\Gah)} \leq \sup_{\vh \in \bfV_h} \frac{\bhtil(\vh,\{\qh,\lh\})}{\norm{\vh}_{H^1(\Gah)}}$.

So, combining the above results \eqref{eq: Term (I) (II) inf-sup new}, \eqref{eq: Term (III) inf-sup new}, \eqref{eq: Term (IV) inf-sup new} with \eqref{eq: new cont inf sup H-1}, and the bound 
\eqref{eq: energy interpolation stability vh lagrange new}, and dividing by $\norm{\vh}_{H^1(\Gah)}$, moving $\sup_{\vh \in \bfV_h} \frac{\bhtil(\vh,\{\qh,\lh\})}{\norm{\vh}_{H^1(\Gah)}}$ to the RHS, and taking the supremum over $\vh \in \bfV_h$, we get our desired result \eqref{eq: almost discrete inf-sup lagrange L^2 H-1}, and therefore proving the new \emph{discrete inf-sup condition} \eqref{eq: L^2 H^{-1} discrete inf-sup condition Gah Lagrange}.
% we see that instead of \eqref{eq: almost discrete inf-sup lagrange} we have 
% \begin{equation}
%     \begin{aligned} \label{eq: almost discrete inf-sup lagrange new}
%         \sup_{\vh \in \bfV_h} \frac{\bhtil(\vh,\{\qh,\lh\})}{\norm{\vh}_{H^1(\Gah)}} &\geq c\norm{\{\qh,\lhl\}}_{L^2(\Gah)\times H^{-1}(\Ga)} - c\norm{\qh}_h\\
%         &\geq  c\norm{\{\qh,\lh\}}_{L^2(\Gah)\times H^{-1}(\Gah)} - c\norm{\qh}_h,
%     \end{aligned}
% \end{equation}
% since it can be readily seen that $\norm{\lh}_{H^{-1}(\Gah)}\leq c \norm{\lhl}_{H^{-1}(\Ga)}$.
% Then Step 2 in \cref{sec: Step 2}  and step 3  \cref{sec: Step 3} remain basically unchanged: Instead of the energy norm $\norm{\cdot}_{\ah}$ we consider the $H^1$-norm, 
\end{proof}

\section{Error bounds}\label{Section: error estimates}
For the error bounds, we will consider two scenarios regarding the degree of the finite element space approximation of the extra Lagrange multiplier $\lh$. First, the case $k_\lambda=k_u-1$ and second, the case where $k_\lambda=k_u$. The difference in the error bounds for these two cases lies in the order of the geometry approximation. But first, let us introduce some standard results.

\subsection{Perturbation bounds for bilinear forms}

\begin{lemma}\label{lemma: Geometric perturbations b lagrange}
Let $\bfw_h,\vh \in \bfH^1(\Gah)$ and $\{\qh,\xi_h \}\in H^1(\Gah)\times L^2(\Gah)$. Then we have
\begin{align}
    \label{eq: Geometric perturbations a}
         |a(\bfw_h^{\ell},\vhl) - \ah(\bfw_h,\vh)| &\leq ch^{k_g} \norm{\bfw_h}_{H^1(\Gah)} \norm{\vh}_{H^1(\Gah)},\\
         |\bhtil(\bfw_h,\{\qh,\xi_h\}) - b^L(\bfw_h^{\ell},\{\qhl,\xi_h^{\ell}\})| &\leq ch^{k_g} \norm{\bfw_h}_{L^2(\Gah)}(\norm{\qh}_{H^1(\Gah)} + \norm{\xi_h}_{L^2(\Gah)}) \nonumber \\
         \label{eq: Geometric perturbations btilde}
        & \leq ch^{k_g-1} \norm{\bfw_h}_{L^2(\Gah)}(\norm{\qh}_{L^2(\Gah)}+ h\norm{\xi_h}_{L^2(\Gah)}). %+ch^{k_g} \norm{\bfw_h}_{L^2(\Gah)}\norm{\xi_h}_{L^2(\Gah)}.
\end{align}
% \begin{equation}
%     \begin{aligned} \label{eq: Geometric perturbations a}
%          |a(\bfw_h^{\ell},\vhl) - \ah(\bfw_h,\vh)| \leq ch^{k_g} \norm{\bfw_h}_{H^1(\Gah)} \norm{\vh}_{H^1(\Gah)}.
%     \end{aligned}
% \end{equation}
% \begin{equation}
%     \begin{aligned}\label{eq: Geometric perturbations btilde}
%         &|\bhtil(\bfw_h,\{\qh,\xi_h\}) - b^L(\bfw_h^{\ell},\{\qhl,\xi_h^{\ell}\})| \leq ch^{k_g} \norm{\bfw_h}_{L^2(\Gah)}(\norm{\qh}_{H^1(\Gah)} + \norm{\xi_h}_{L^2(\Gah)}) \\
%         &\quad \leq ch^{k_g-1} \norm{\bfw_h}_{L^2(\Gah)}\norm{\qh}_{L^2(\Gah)} +ch^{k_g} \norm{\bfw_h}_{L^2(\Gah)}\norm{\xi_h}_{L^2(\Gah)}.
%     \end{aligned}
% \end{equation}
If  $\bfw \in \bfH_T^1$, then we get the following higher order bound
\begin{equation}
    \begin{aligned}\label{eq: Geometric perturbations btilde tangent}
        |\bhtil(\bfw^{-\ell},\{\qh,\xi_h\}) - b^L(\bfw,\{\qhl,\xi_h^{\ell}\})|&\leq ch^{k_g+1} \norm{\bfw}_{L^2(\Ga)}\norm{\qh}_{H^1(\Gah)} + ch^{k_g}\norm{\bfw}_{L^2(\Ga)}\norm{\xi_h}_{L^2(\Gah)}\\
        &\leq ch^{k_g} \norm{\bfw}_{L^2(\Ga)}\norm{\{\qh,\xi_h\}}_{L^2(\Gah)}.
    \end{aligned}
\end{equation}
Furthermore if  $\bfw \in \bfH_T^1\cap \bfH^2(\Ga)$ then we have
\begin{equation}\label{eq: Geometric perturbations btilde tangent extra regularity}
    |\bhtil(\bfw^{-\ell},\{\qh,\xi_h\}) - b^L(\bfw,\{\qhl,\xi_h^{\ell}\})|\leq ch^{k_g+1} \norm{\bfw}_{H^2(\Ga)}\norm{\{\qh,\lh\}}_{H^1(\Gah)}. 
\end{equation}
\end{lemma}\vspace{-8mm}

\begin{proof}
\ Considering the discrete \eqref{eq: lagrange discrete bilinear forms}  and continuous \eqref{eq: regular bilinear a} forms, we can readily prove \eqref{eq: Geometric perturbations a} if we view the estimate \eqref{eq: errors of domain of integration data} (for $\bfw_h$) and \Cref{lemma: Geometric perturbations a}. Therefore 
\begin{equation*}
    \begin{aligned}
       |a(\bfw_h^{\ell},\vhl) - \ah(\bfw_h,\vh)| \leq ch^{k_g} \norm{\bfw_h}_{H^1(\Gah)} \norm{\vh}_{H^1(\Gah)}.
    \end{aligned}
\end{equation*}  
\noindent   For the estimate \eqref{eq: Geometric perturbations btilde} using similar to the previous calculations we get that
\begin{equation}
    \begin{aligned}\label{eq:  Geometric perturbations btilde in proof}
    &\bhtil(\bfw_h,\{\qh,\xi_h\}) - b^L(\bfw_h^{\ell},\{\qhl,\xi_h^{\ell}\})   = \int_{\Gah} \bfw_h \nbgh \qh - \int_\Ga \bfw_h^{\ell}\nbg\qhl + \int_{\Gah} \xi_h \bfw_h \cdot \nh - \int_\Ga \xi_h^{\ell} \bfw_h \cdot \bfng  \\
    & =  \int_{\Ga} (\frac{1}{\mu_h} -1)\bfw_h^{\ell} \Bhg \nbg \qhl + \int_{\Ga} \bfw_h^{\ell}( \Bhg - \bfPg)\nbg \qhl  +  \int_{\Ga} (\frac{1}{\mu_h} -1)\xi_h^{\ell} \bfw_h^{\ell} \cdot \nhl + \int_{\Ga} \xi_h^{\ell} \bfw_h^{\ell} \cdot(\nhl - \bfng) \\
    &\leq ch^{k_g+1} \norm{\bfw_h}_{L^2(\Gah)}\norm{\qh}_{H^1(\Gah)} +  \int_{\Ga} \bfPg \nhl\otimes\nhl \bfw_h^{\ell} \cdot \nbg \qhl + c(h^{k_g+1} + h^{k_g}) \norm{\bfw_h}_{L^2(\Gah)}\norm{\xi_h}_{L^2(\Gah)}\\
    & \leq ch^{k_g} \norm{\bfw_h}_{L^2(\Gah)}(\norm{\qh}_{H^1(\Gah)} + \norm{\xi_h}_{L^2(\Gah)}),
    \end{aligned}
\end{equation}
where using the application of an inverse inequality we may also get the last inequality of \eqref{eq: Geometric perturbations btilde}. 
For inequality \eqref{eq: Geometric perturbations btilde tangent} we go back to the line before the last inequality of \eqref{eq:  Geometric perturbations btilde in proof}, where since $\bfw\cdot\bfng=0$ we have that
\begin{equation}
    \begin{aligned}\label{eq:  Geometric perturbations btilde in proof 2}
    &\bhtil(\bfw^{-\ell},\{\qh,\xi_h\}) - b^L(\bfw,\{\qhl,\xi_h^{\ell}\})\\
    &\leq ch^{k_g+1} \norm{\bfw^{-\ell}}_{L^2(\Gah)}\norm{\qh}_{H^1(\Gah)} +  \underbrace{\int_{\Ga} \bfPg \nhl\otimes\nhl \bfw \cdot \nbg \qhl}_{\leq h^{2k_g}\norm{\bfw^{-\ell}}_{L^2(\Gah)}\norm{\qh}_{H^1(\Gah)}} + c(h^{k_g+1} + h^{k_g}) \norm{\bfw^{-\ell}}_{L^2(\Gah)}\norm{\xi_h}_{L^2(\Gah)}\\
    &\overset{\text{inverse inequality}}{\leq}ch^{k_g} \norm{\bfw}_{L^2(\Ga)}\norm{\{\qh,\xi_h\}}_{L^2(\Gah)}.
    \end{aligned}
\end{equation}

For the last inequality \eqref{eq: Geometric perturbations btilde tangent extra regularity} we need to improve the estimate involving the extra Lagrange multiplier, more specifically the last terms on the second line of \eqref{eq:  Geometric perturbations btilde in proof} and of \eqref{eq:  Geometric perturbations btilde in proof 2} . For that we make use of a known estimate appearing in \cite[Lemma 4.2]{hansbo2020analysis}:
\begin{equation}\label{eq: high Pnh estimate}
    (\bfP\cdot\nh,\chi^{-\ell})_{L^2(\Gah)} \leq ch^{k_g+1}\norm{\chi}_{W^1_1(\Ga)}.
\end{equation}
Therefore with appropriate alteration, the norm equivalence \eqref{eq: norm equivalence} and the Sobolev inequalities we can readily see that
\begin{equation}
    \begin{aligned}
         \int_{\Ga} \xi_h^{\ell} \bfw\cdot(\nhl - \bfng) =  \int_{\Ga} \xi_h^{\ell}\bfw \bfPg\cdot\nhl  \leq ch^{k_g+1}\norm{\xi_h^{\ell}\bfw}_{W^1_1(\Ga)} \leq ch^{k_g+1}\norm{\bfw}_{H^2(\Ga)}\norm{\xi_h}_{H^1(\Gah)}.
    \end{aligned}
\end{equation}
Therefore, replacing the last term in the second line of \eqref{eq:  Geometric perturbations btilde in proof 2} with this new result gives our final desired inequality.
\end{proof}

\subsection{The error equation}

To start, let us recall again the continuous and discrete Lagrange multiplier problem.
Let $(\bfu,\{p,\lambda\}) \in \bfH^1(\Ga) \times L^2_0(\Ga) \times L^2(\Ga)$ be the solution of \eqref{weak lagrange hom}:
\begin{align*}
\begin{cases}
        a(\bfu,\bfv) \, + \!\!\!\!\!\!\!&b^L(\bfv,\{p,\lambda\}) = (\bff,\bfv) \qquad \ \ \text{for all } \bfv\in \bfH^1(\Ga),\\
        &b^L(\bfu,\{q,\xi\})=0 \qquad \qquad  \text{ for all } \{q,\xi\}\in L^2_0(\Ga)\times L^2(\Ga),
    \end{cases}
\end{align*}
satisfying  $\bfu \cdot \bfng =0 $. Also, let $(\uh,\{\ph,\qh\}) \in \bfV_h \times (Q_h \times \Lambda_h)$ be the solution to \eqref{weak lagrange discrete}:
\begin{align*}
\begin{cases}
        \ah(\uh,\vh) \, + \!\!\!\!\!\!\!\!&\bhtil(\vh,\{\ph,\lh\}) = (\bff_h,\vh) \quad \ \   \text{for all } \vh \in \bfV_h,\\
        &\bhtil(\uh,\{\qh,\xi_h\})=0 \qquad \qquad \text{ for all } \{\qh,\xi_h\}\in Q_h \times \Lambda_h, 
    \end{cases}
\end{align*}
where we consider the approximation of $\bff$ is such that $\bff_h^{\ell} = \bff$. We begin with the standard decomposition of the error into approximation and discretizations errors:
\begin{align}\label{eq: decomposition error 1 lagrange}
    \eu = \bfu^{-\ell} - \uh &= \underbrace{(\bfu^{-\ell}- \Ihz \bfu^{-\ell})}_{\text{Interpolation error }} + \underbrace{(\Ihz \bfu^{-\ell} - \uh)}_{\text{discrete remainder}},\\ %=:\rho_{\bfu} + \theta_{\bfu} ,\\
        \label{eq: decomposition error 2 lagrange} \ep = p^{-\ell} - \ph &= \underbrace{(p^{-\ell}- \Ihz p^{-\ell})}_{\text{Interpolation error}} + \underbrace{(\Ihz p^{-\ell} - \ph)}_{\text{discrete remainder}},\\%=: \rho_{p} + \theta_{p},\\
        \label{eq: decomposition error 3 lagrange} \el =  \lambda^{-\ell} - \lh &= \underbrace{(\lambda^{-\ell}- \Ihz \lambda^{-\ell})}_{\text{Interpolation error}} + \underbrace{(\Ihz \lambda^{-\ell} - \lh)}_{\text{discrete remainder}}, %=: \rho_{\lambda} + \theta_{\lambda},
\end{align}
 where we use the notation
 \begin{align}
&\uh^{I} = \Ihz \bfu^{-\ell},~~~~~~~~ \ph^I = \Ihz p^{-\ell},~~~~~~~~\lh^I =  \Ihz \lambda^{-\ell}.\\
&\rho_{\bfu}= \uh^{I}-\bfu^{-\ell},~~~ \rho_{p} = \ph^{I}-p^{-\ell},~~~ \rho_{\lambda} = \lh^{I} - \lambda^{-\ell}.\\
&\theta_{\bfu}= \uh-\uh^{I},~~~~~ \theta_{p} = \ph - \ph^{I},~~~~~ \theta_{\lambda} = \lh -\lh^{I}.
\end{align}
The interpolation errors are already known; see Section \ref{sec: Interpolation} and \cref{lemma: energy norm interpolant lagrange}. Our task is to bound the discretisation errors.
% in the Lagrange energy $\ah-$norm (Lemma \ref{lemma: energy norm interpolant lagrange}) for the velocity field.
\begin{lemma}\label{lemma: Lh consistency}
The discretisation errors satisfy,
\begin{align}
\begin{cases}\label{eq: Lh consistency}
        \ah(\theta_{\bfu},\vh)\,+\!\!\!\!\!&\bhtil(\vh,\{\theta_{p},\theta_{\lambda}\}) = \ah(\rho_{\bfu},\vh) + \bhtil(\vh,\{\rho_{p},\rho_{\lambda}\}) \, + \,  \sum_{i=1}^3 \text{Err}_{i}^{L}(\vh) \ \  \text{ for all } \vh \in \bfV_h,\\
        & \bhtil(\theta_{\bfu},\{\qh,\xi_h\})= \bh(\rho_{\bfu},\{\qh,\xi_h\}) \, + \,  \text{Err}_{4}^{L}(\{\qh,\xi_h\})  \quad \ \ \text{for all } \{\qh,\xi_h\}\in Q_h\times \Lambda_h,
    \end{cases}
\end{align}
with the consistency errors:
\begin{itemize}
\setlength\itemsep{0.7em}
\item \ \ $ \text{Err}_1^{L}(\vh) := a(\bfu,\vhl) - \ah(\bfu^{-\ell},\vh)$,
\item \ \ $ \text{Err}_2^{L}(\vh) := b^L(\vhl,\{p,\lambda\})-\bhtil(\vh,\{p^{-\ell},\lambda^{-\ell}\})$,
\item \ \ $ \text{Err}_3^{L}(\vh) := (\bff_h,\vh)_{L^2(\Gah)} - (\bff,\vhl)_{L^2(\Ga)}$,
\item \ \ $ \text{Err}_4^{L}(\{\qh,\xi_h\}) := b^L(\bfu,\{\qhl,\xi_h^{\ell}\})-\bhtil(\bfu^{-\ell},\{\qh,\xi_h\})$.
\end{itemize}
\end{lemma}
\begin{proof}
\ Let $(\uh,$ $\{\ph,\lh \}) \in \bfV_h\times (Q_h\times \Lambda_h)$ the unique solution of \eqref{weak lagrange discrete} then add and subtract the interpolation errors $(\uh^{I},\{\ph^{I},\lh^I\}) \in  \bfV_h\times (Q_h\times \Lambda_h)$ to get
% Let us start with the discrete Lagrange multiplier scheme \eqref{weak lagrange discrete}. 
% Assume $(\uh,$ $\{\ph,\lh \}) \in \bfV_h\times (Q_h\times \Lambda_h)$ the unique solution,
% then add and subtract the interpolation errors $(\uh^{I},\{\ph^{I},\lh^I\}) \in  \bfV_h\times (Q_h\times \Lambda_h)$ to get
\begin{equation}
    \begin{aligned}\label{eq: Lh rewrite vhtilde}
         \ah(\thbfu,\vh)\,+ \,\bhtil(\vh,\{\thp,\thl\}) &=  \ah(-\uh^{I},\vh) +\bhtil(\vh,-\{\ph^{I},\lh^{I}\})+(\bff_h, \vh)_{L^2(\Gah)}, \\
        \bhtil(\uh-\uh^{I},\{\qh,\xi_h\})&= \bhtil(-\uh^{I},\{\qh,\xi_h\}).
    \end{aligned}
\end{equation}
By adding and subtracting appropriate terms on \eqref{eq: Lh rewrite vhtilde} with the help of \eqref{weak lagrange hom} we can readily see that
\begin{equation}
    \begin{aligned}\label{eq: Lh consistency line 1}
         &\ah(\theta_{\bfu},\vh)\,+ \,\bhtil(\vh,\{\theta_{p},\theta_{\lambda}\}) = \ah(\bfu^{-\ell}-\uh^{I},\vh) + \bhtil(\vh,\{p^{-\ell}-\ph^I, \lambda^{-\ell}-\lh^I\}) \\
          &\quad  - \ah(\bfu^{-\ell},\vh) -\bhtil(\vh,\{p^{-\ell},\lambda^{-\ell}\}) + (\bff_h,\vh)_{L^2(\Gah)} \\
         &= \underbrace{\ah(\rho_{\bfu},\vh)}_{\text{Int}_1^{L}} + \underbrace{\bhtil(\vh,\{\rho_{p}, \rho_{\lambda}\})}_{\text{Int}_2^{L}} + \underbrace{a(\bfu,\vhl) - \ah(\bfu^{-\ell},\vh)}_{\text{Err}_1^{L}} \\
         &\quad  +\underbrace{b^L(\vhl,\{p,\lambda\})-\bhtil(\vh,\{p^{-\ell},\lambda^{-\ell}\})}_{\text{Err}_2^{L}} + \underbrace{(\bff_h,\vh)_{L^2(\Gah)} - (\bff,\vhl)_{L^2(\Ga)}}_{\text{Err}_3^{L}},
    \end{aligned}
\end{equation}
Similarly we get 
\begin{equation}
    \begin{aligned}\label{eq: Lh consistency line 2}
         \bhtil(\theta_{\bfu},\{\qh,\xi_h\}) &= \underbrace{\bhtil(\rho_{\bfu},\{\qh,\xi_h\})}_{\text{Int}_3^{L}} + \underbrace{b^L(\bfu,\{\qhl,\xi_h^{\ell}\})-\bhtil(\bfu^{-\ell},\{\qh,\xi_h\})}_{\text{Err}_4^{L}}.
    \end{aligned}
\end{equation}
Combining \eqref{eq: Lh consistency line 1} and \eqref{eq: Lh consistency line 2} we get our desired result.
\end{proof}

\subsection{Consistency and a-priori estimates for $k_\lambda = k_u-1$} \label{sec: the kl=ku-1 case}

\subsubsection{Energy error estimates}\label{sec: Energy Error Estimates}

\begin{lemma}[Consistency error bounds]\label{Lemma: Consistency error bounds Lagrange}
Let $(\vh,\{\qh,\lh\}) \in \bfV_h\times Q_h \times \Lambda_h$ and $(\bfu,\{p,\lambda\})\in \bfH^1(\Ga) \times L^2_0(\Ga)\times L^2(\Ga)$ be the solution to the continuous problem \eqref{weak lagrange hom}, then the consistency error terms $\text{Err}_i^{L}(\cdot)$, i=1,...,4 as in Lemma \ref{lemma: Lh consistency} are bounded as followed:
\begin{equation}\label{eq: Consistency error bounds Lagrange}
    \begin{aligned}
    \bfE_1^{L} = \sup_{\vh \in \bfV_h} \frac{|\text{Err}_{1}^{L}(\vh)|}{\norm{\vh}_{\ah}} &\leq ch^{k_g-1}\norm{\bfu}_{H^1(\Ga)},\\[5pt]
        \bfE_2^{L}=\sup_{\vh \in \bfV_h} \frac{|\text{Err}_{2}^{L}(\vh)|}{\norm{\vh}_{\ah}} &\leq ch^{k_g}(\norm{p}_{H^1(\Ga)}+\norm{\lambda}_{L^2(\Ga)}),\\[5pt]
         \bfE_3^{L}=\sup_{\vh \in \bfV_h} \frac{|\text{Err}_{3}^{L}(\vh)|}{\norm{\vh}_{\ah}} &\leq ch^{k_g+1}\norm{\bff}_{L^2(\Ga)},\\[5pt]
         \bfE_4^{L}=\sup_{\{\qh,\xi_h\} \in (Q_h\times\Lambda_h)} \frac{|\text{Err}_{4}^{L}(\{\qh,\xi_h\})|}{\norm{\{\qh,\xi_h\}}_{L^2(\Gah)}} &\leq ch^{k_g}\norm{\bfu}_{L^2(\Ga)},\\[5pt]
    \end{aligned}
\end{equation}
where the constants $c$ are independent of the mesh-parameter $h$.
\end{lemma}

\begin{proof}
The consistency error $\bfE_1^{L}$ can be estimated by using \eqref{eq: Geometric perturbations a} and the weaker bound $\norm{\vh}_{H^1(\Gah)}\leq h^{-1}\norm{\vh}_{\ah}$; see \eqref{coercivity and Korn's inequality Lagrange}. 
%which comes from the discrete Korn's inequality \eqref{discrete Korn's inequality T nh}.
The second estimate $\bfE_2^{L}$ may be calculated if we consider \eqref{eq: Geometric perturbations btilde}.
The third error follows easily from \eqref{eq: errors of domain of integration data}. Finally, 
since the solution to the homogeneous Lagrange problem \eqref{weak lagrange hom} is tangent i.e. $\bfu\cdot\bfng =0$, the fourth consistency error $\bfE_4^{L}$ follows from \eqref{eq: Geometric perturbations btilde tangent}.
\end{proof}
\begin{remark}\label{remark: The reason why we lose kg-1 consistency}
We notice that the consistency error $\bfE_1^{L}$ is the reason for the loss of one order in the geometric error, that is, the error caused by the approximation of the geometry. In \cref{sec: the kl=ku case}, where we consider $k_\lambda=k_u$, this is improved with the help of the improved $H^1$ coercivity estimates in  \cref{lemma: new H1 estimate thbfu kl=ku} (in comparison to 
\eqref{coercivity and Korn's inequality Lagrange}) when we consider $\vh = \thbfu^{div}\in \bfV_h^{div}$. 
\end{remark}

\begin{lemma}\label{Lemma: discrete remainder error Lagrange}
Let $(\uh,\{p_h,\lambda_h\})\in \bfV_h\times(Q_h\times \Lambda_h)$ be the unique solution of \eqref{weak lagrange discrete}. Then, for sufficiently small $h$, the following estimates hold
% For $\theta_{\bfu}= \uh-\uh^{I}$, $\theta_{p} = \ph - \ph^I$ and  $\theta_{\lambda} = \lh - \lh^I$ with $(\uh^{I},\{\ph^I,\lh^I\}) = (\Ihz \bfu^{-\ell}, \ \Ihz p^{-\ell},\ \Ihz \lambda^{-\ell}) \in \bfV_h\times (Q_h\times \Lambda_h)$ and sufficiently small $h$, the following stability estimates for the discrete remainders of the Lagrange multiplier discrete method hold,
\begin{equation}
    \begin{aligned}\label{eq: discrete remainder error Lagrange}
        \norm{\theta_{\bfu}}_{\ah} + \norm{\{\theta_{p}, \theta_{\lambda}\}}_{L^2(\Gah)} &\leq c (h^{k_u} + h^{k_g-1})\norm{\bfu}_{H^{k_u+1}(\Ga)} + ch^{k_g}\norm{\bfu}_{H^1(\Ga)}\\
        &\quad \ + c(h^{k_{pr}+1} + h^{k_g})(\norm{p}_{H^{k_{pr}+1}(\Ga)} + \norm{\lambda}_{L^2(\Ga)})  \\
        &\quad \ + ch^{k_{\lambda}+1}\norm{\lambda}_{H^{k_{\lambda}+1}(\Ga)} + ch^{k_g+1}\norm{\bff}_{L^2(\Ga)}.
    \end{aligned}
\end{equation}

\end{lemma}
\begin{proof}
This result follows from the well-posedness
\Cref{Lemma: Well-posedness of discrete Lagrange problem}, the interpolation error \eqref{eq: interpolation vh ah},\eqref{eq: Scott-Zhang interpolant} and the previously proven consistency errors \Cref{Lemma: Consistency error bounds Lagrange}.
\end{proof}

\begin{theorem}[Energy Error Estimate for Lagrange Formulation]\label{Theorem: Energy Error Estimate for Lagrange Formulation}
Let $(\bfu,\{p,\lambda\}) \in \bfH^1(\Ga) \times L^2_0(\Ga)\times L^2(\Ga)$ be the solution of \eqref{weak lagrange hom}. Also let $(\uh,$ $\{\ph,\lambda_h\}) \in \bfV_h \times Q_h \times \Lambda_h$ be the solution to the discrete scheme \eqref{weak lagrange discrete}. Then, for $k_u \geq 2$, $k_{pr}=k_u-1$, and $k_{\lambda} \geq 1$ the following error estimates hold
\begin{equation}
    \begin{aligned}
        \norm{\bfu^{-\ell} - \uh}_{\ah} + \norm{p^{-\ell} - \ph}_{L^2(\Gah)}  &+ \norm{\lambda^{-\ell} - \lh}_{L^2(\Gah)} \\
        & \leq  c (h^{k_u} + h^{k_g-1})\norm{\bfu}_{H^{k_u+1}(\Ga)} \\
        &\quad + c(h^{k_{pr+1}} + h^{k_g})(\norm{p}_{H^{k_{pr+1}}(\Ga)} + \norm{\lambda}_{L^2(\Ga)})  \\
        &\quad + ch^{k_{\lambda}+1}\norm{\lambda}_{H^{k_{\lambda}+1}(\Ga)} + ch^{k_g+1}\norm{\bff}_{L^2(\Ga)}.
    \end{aligned}
\end{equation}
\end{theorem}

\begin{proof}
Due to the decompositions \eqref{eq: decomposition error 1 lagrange}, \eqref{eq: decomposition error 2 lagrange}, and \eqref{eq: decomposition error 3 lagrange} if we combine the $L^2$ interpolation errors for the two pressures $p$ and $\lambda$; see \eqref{eq: Scott-Zhang interpolant}, and the velocity in the $\ah-$norm \eqref{eq: interpolation vh ah}, together with the previously proven error involving the discrete remainders in \Cref{Lemma: discrete remainder error Lagrange} we obtain the desired error bounds for the Lagrange multiplier formulation.
\end{proof}

\subsubsection{$L^2$-norm velocity error bounds}\label{sec: L2 Error Estimates}

We now require higher regularity for the Lagrange formulation of the problem \eqref{eq: generalized Lagrange surface stokes}. Considering the equation \eqref{eq: generalized Lagrange surface stokes} we can clearly see that its tangent solution is equal to the solution of \eqref{eq: generalized tangential surface stokes} while their pressure $p$ also agrees (assuming that the tangent parts of the source terms $\bff$ agree as well). As seen in \cite[Lemma 2.1]{olshanskii2021inf}, then, the tangent part of  the velocity $\bfu_T$ and the pressure $p$ have higher regularity. Then noticing that the extra Lagrange multiplier $\lambda = \bff_n-tr(E(\bfu_T)\bfH)$ 
we also get higher regularity for $\lambda$ (assuming that the normal part of the source term $\bff_n \in \bfH^1(\Ga)$).
So, we may consider the following adjoint problem: Find $(\bfw, \pi,\mu) \in \bfH^1(\Ga) \times L^2_0(\Ga)\times L^2(\Ga)$ such that  
\begin{align*}
\begin{cases}   
        a(\bfw,\bfv) \ + \!\!\!\!& b^L(\bfv,\{\pi,\mu\}) = (\bfh,\bfv) \ \ \ \ \ \text{for all } \bfv\in \bfH^1(\Ga),\\
        &b^L(\bfw,\{\sigma,\xi\})=0 \ \ \  \text{ for all } \{\sigma,\xi\} \in L_0^2(\Ga)\times L^2(\Ga),
    \end{cases}
\end{align*}
where $\bfh_T \in L^2(\Ga)$ and $\bfh_n \in H^1(\Ga)$ and $\bfw\cdot \bfng =0$ due to the second equation. Now, as mentioned before, by \cite[Lemma 2.1]{olshanskii2021inf} the following high regularity estimates hold
% the solution also satisfies the following higher regularity estimate
\begin{equation}\label{eq: regularity estimate lagrange}
    \norm{\bfw}_{H^2(\Ga)} + \norm{\{\pi,\mu\}}_{H^1(\Ga)} \leq \norm{\bfh_T}_{L^2(\Ga)} + \norm{\bfh_n}_{H^1(\Ga)}.
\end{equation}
\noindent Let us also denote the errors in \Cref{Theorem: Energy Error Estimate for Lagrange Formulation} as
$\eu = (\bfu - \uhl)$, $\ep = p-\phl$, and $\el = \lambda - \lhl$, and furthermore the tangential component $\eut = \bfP\eu = \bfP(\bfu - \uhl)$.
\begin{theorem}[ Tangential $L^2$ velocity error estimates]\label{Theorem: L^2 error velocity Lagrange}
    Let $(\bfu,p,\lambda)$ be solution of the continuous problem \eqref{weak lagrange hom} and assume further regularity $(\bfu,p,\lambda) \in (H^2(\Ga))^3\times H^1(\Ga)\times H^1(\Ga)$. Also let  $(\uh,\ph,\lh) \in \bfV_h \times Q_h\times \Lambda_h$ be the solution to the discrete scheme \eqref{weak lagrange discrete}. Then under the same assumptions as \Cref{Theorem: Energy Error Estimate for Lagrange Formulation} and $k_g \geq 2 $, the following estimate holds for the tangential velocity 
    \begin{equation}
        \begin{aligned}
            \norm{\eut}_{L^2(\Ga)}  &= \norm{\bfPg(\bfu-\uhl)}_{L^2(\Ga)} \\
            & \leq ch^{m+1}(\norm{\bfu}_{H^{k_u+1}(\Ga)}+\norm{p}_{H^{k_{pr}+1}(\Ga)} + \norm{\lambda}_{H^{k_{\lambda}+1}(\Ga)}+ \norm{\bff}_{L^2(\Ga)}),
        \end{aligned}
    \end{equation}
 where $m= min\{{k_u},{k_{pr}+1}, {k_{\lambda}+1}, {k_g-1}\}$.
\end{theorem}
\begin{proof}
    By letting $\bfh =  \eut$ and testing with $\eu$ we arrive at
    \begin{equation}\label{eq: Nitsche type equation lagrange}
        a(\bfw,\eu) + b^{L}(\eu,\{\pi,\mu\}) = \norm{\eut}_{L^2(\Ga)}^2.
    \end{equation}
     Notice also that 
    \begin{equation}\label{eq: bilinear b w p l}
       b^L(\bfw,\{p - \phl,\lambda-\lhl\})  = b^L(\bfw,\{\ep,\el\})= 0.
    \end{equation}
Now due to \eqref{eq: regularity estimate lagrange} we also have that 
    \begin{equation}\label{eq: regularity estimate lagrange inside}
        \norm{\bfw}_{H^2(\Ga)} + \norm{\{\pi,\mu\}}_{H^1(\Ga)} \leq \norm{\eut}_{L^2(\Ga)}.
    \end{equation}
Let us start with the bilinear form $a^L(\cdot,\cdot)$. The steps, though different due to the extra Lagrange multipliers e.g. the pressure $\pi$, follow very similarly the steps of the proof \cite[Theorem 5]{hansbo2020analysis}. We add and subtract appropriate term so that we may get
\begin{equation}
    \begin{aligned}\label{eq: a_T extra regularity error large lagrange old}
        a(\bfw,\eu) &= a(\bfw-\Ihzl(\bfw),\eu) + a(\Ihzl(\bfw),\eu)\\
         &=a(\bfw-\Ihzl(\bfw),\eu) - a(\Ihzl(\bfw),\uhl) - b^L(\Ihzl(\bfw),\{p,\lambda\}) + (\bff,\Ihzl(\bfw))\\
         &=a(\bfw-\Ihzl(\bfw),\eu) + \ah(\Ihz(\bfw^{-\ell}),\uh) - a(\Ihzl(\bfw),\uhl) \\
         &\ \  + \bhtil(\Ihz(\bfw^{-\ell}),\{\ph,\lh\}) - b^L(\Ihzl(\bfw),\{p,\lambda\})  -(\bff_h,\Ihz(\bfw^{-\ell}))+ (\bff,\Ihzl(\bfw))\\
         &\ \ + \underbrace{b^L(\bfw,\{\ep,\el\})}_{=0, \text{ by } \eqref{eq: bilinear b w p l}}.
         \end{aligned}
\end{equation}
Continuing with the calculations we have the following large expression
\begin{equation}
    \begin{aligned} \label{eq: a_T extra regularity error large lagrange}
        a(\bfw,\eu) &= a(\bfw-\Ihzl(\bfw),\eu) + \ah(\Ihz(\bfw^{-\ell}),\theta_{\bfu})- a(\Ihzl(\bfw),\theta_{\bfu}^{\ell})\\
        &\ \ + \ah(\Ihz(\bfw^{-\ell}),\Ihz(\bfu^{-\ell}))- a(\Ihzl(\bfw),\Ihzl(\bfu)) \\
        &\ \ + b^L(\bfw - \Ihzl(\bfw),\{\ep,\el\}) + \bhtil(\Ihz(\bfw^{-\ell}),\{\theta_{p},\theta_{\lambda}\})- b^L(\Ihzl(\bfw),\{\theta_{p}^{\ell},\theta_{\lambda}^{\ell}\})\\
        &\ \ + \bhtil(\Ihz(\bfw^{-\ell}),\{\Ihz(p^{-\ell}),\Ihz(\lambda^{-\ell})\}) - b^L(\Ihzl(\bfw),\{\Ihzl(p),\Ihzl(\lambda)\})\\
        &\ \ + (\bff,\Ihzl(\bfw))_{L^2(\Ga)} - (\bff_h,\Ihz(\bfw^{-\ell}))_{L^2(\Gah)}.
    \end{aligned}
\end{equation}
Let us bound its term of \eqref{eq: a_T extra regularity error large lagrange} appropriately one by one:
\begin{itemize}
    \item \ \ The first term of \eqref{eq: a_T extra regularity error large lagrange} can be bounded as followed
    \begin{equation}
        \begin{aligned}\label{eq: L^2 estimate proof 1 lagrange}
             a(\bfw-\Ihzl(\bfw),\eu) &\leq \norm{\bfw-\Ihzl(\bfw)}_{a}(\norm{\rho_\bfu^{\ell}}_{a} + \norm{\theta_{\bfu}^{\ell}}_{a}) \leq c  \norm{\bfw-\Ihzl(\bfw)}_{H^1(\Ga)}(\norm{\rho_\bfu^{\ell}}_{H^1(\Ga)} + \norm{\theta_{\bfu}}_{\ah})  \\ 
                 &\leq ch^{m+1}(\norm{\bfu}_{H^{k_u+1}(\Ga)} + \norm{p}_{H^{k_{pr}+1}(\Ga)} + \norm{\lambda}_{H^{k_{\lambda}+1}(\Ga)}), 
        \end{aligned}
    \end{equation}
        where $m= min\{{k_u+1},$ ${k_{pr}+1},\, {k_{\lambda}+1}, \,{k_g-1}\}$, and where we used \Cref{Lemma: discrete remainder error Lagrange}, the interpolation estimate \eqref{eq: Scott-Zhang interpolant}, the norm equivalence \eqref{eq: norm equivalence}  and the fact that $\norm{\theta_{\bfu}}_{a}\leq \norm{\theta_{\bfu}}_{\ah}$ assuming $k_g \geq 2$, due to \eqref{eq: Geometric perturbations a 1}, the weaker bound $\norm{\theta_{\bfu}}_{H^1(\Gah)}^2 \leq h^{-2}\ah(\theta_{\bfu},\theta_{\bfu})$ and assuming $k_g \geq 2$.
    \item \ \ Moving onto the next two term we make use of the geometric perturbation \eqref{eq: Geometric perturbations a 1} to obtain the following bound
    \begin{equation}
         \begin{aligned}\label{eq: L^2 estimate proof 2 lagrange}                    \ah(\Ihz(\bfw^{-\ell}),\theta_{\bfu})- a(\Ihzl(\bfw),\theta_{\bfu}^{\ell})
                    &\leq ch^{k_g}\norm{\Ihz(\bfw^{-\ell})}_{H^1(\Gah)}\norm{\theta_{\bfu}}_{H^1(\Gah)}
                    \\
                    &\leq ch^{k_g-1 + m}\norm{\bfw}_{H^2(\Ga)}(\norm{\bfu}_{H^{k_u+1}(\Ga)} + \norm{p}_{H^{k_{pr}+1}(\Ga)} + \norm{\lambda}_{H^{k_{\lambda}+1}(\Ga)}),
            \end{aligned}
    \end{equation}
    where we have used again use the bound $\norm{\uh}_{H^1(\Gah)}^2 \leq h^{-2}\ah(\uh,\uh)$ since the discrete Korn's inequality \eqref{discrete Korn's inequality T nh} holds, and \Cref{Lemma: discrete remainder error Lagrange}.
    \item \ For the next estimate $\bfG(\Ihzl(\bfw),\Ihzl(\bfu)) =  \ah(\Ihz(\bfw^{-\ell}),\Ihz(\bfu^{-\ell}))- a(\Ihzl(\bfw),\Ihzl(\bfu))$    we add and subtract appropriate terms, use the interpolation estimate \eqref{eq: interpolation vh ah}, \eqref{eq: interpolation stability vh ah} to obtain 
\begin{equation}
        \begin{aligned}\label{eq: L^2 estimate proof 3 lagrange}
            \bfG(\Ihzl(\bfw),\Ihzl(\bfu)) &=  \bfG(\Ihzl(\bfw)-\bfw,\Ihzl(\bfu)- \bfu) +  \bfG(\bfw,\Ihzl(\bfu)-\bfu) + \bfG(\Ihzl(\bfw)-\bfw,\bfu)  + \bfG(\bfw,\bfu)\\
            &\leq c(h^{k_g+k_u}+ h^{k_g+1})\norm{\bfw}_{H^2(\Ga)}\norm{\bfu}_{H^{k_u+1}(\Ga)} + \bfG(\bfw,\bfu)\\
            &\underbrace{\leq}_{k_u+k_g \leq k_g+1} ch^{k_g+1} \norm{\bfw}_{H^2(\Ga)}\norm{\bfu}_{H^2(\Ga)}
        \end{aligned}
    \end{equation}
    where for the last term we also used the better estimate in \cite[Lemma 5.4]{hansbo2020analysis}\  $\bfG(\bfw,\bfu) \leq ch^{k_g+1} \norm{\bfw}_{H^2(\Ga)}\norm{\bfu}_{H^2(\Ga)}$ for $\bfw,\bfu\in \bfH^1
    _T(\Ga)$.
    \item \ \ Now moving onto the other bilinear form using the bounds on $b^L(\cdot,\{\cdot,\cdot\})$ we readily see that
    \begin{equation}
        \begin{aligned}\label{eq: L^2 estimate proof 5 lagrange}
            b^L(\bfw - \Ihzl(\bfw),\{\ep,\el\})  &\leq \norm{\bfw - \Ihzl(\bfw)}_{H^1(\Ga)}\norm{\ep}_{L^2(\Ga)} + \norm{\bfw - \Ihzl(\bfw)}_{L^2(\Ga)}\norm{\el}_{L^2(\Ga)}\\
            &\leq ch^{m+1}\norm{\bfw}_{H^2(\Ga)}\big(\norm{\bfu}_{H^{k_u+1}(\Ga)}+\norm{p}_{H^{k_{pr}+1}(\Ga)} + \norm{\lambda}_{H^{k_{\lambda}+1}(\Ga)}\big).
        \end{aligned}
    \end{equation}
    \item \ \ For the next two terms we  apply the geometric perturbation \eqref{eq: Geometric perturbations btilde} and \eqref{eq: Geometric perturbations btilde tangent}, interpolation estimates and \Cref{Lemma: discrete remainder error Lagrange} to obtain
    \begin{equation}
        \begin{aligned}\label{eq: L^2 estimate proof 6 lagrange}
            &\bhtil(\Ihz(\bfw^{-\ell}),\{\theta_{p},\theta_{\lambda}\})- b^L(\Ihzl(\bfw),\{\theta_{p}^{\ell},\theta_{\lambda}^{\ell}\}) = \bhtil(\bfw^{-\ell},\{\theta_{p},\theta_{\lambda}\})- b^L(\bfw,\{\theta_{p}^{\ell},\theta_{\lambda}^{\ell}\})\\
            &\quad \ + \bhtil(\Ihz(\bfw^{-\ell})- \bfw^{-\ell},\{\theta_{p},\theta_{\lambda}\})- b^L(\Ihzl(\bfw)-\bfw,\{\theta_{p}^{\ell},\theta_{\lambda}^{\ell}\})\\
            & \leq ch^{k_g}\norm{\bfw}_{L^2(\Ga)}(\norm{\theta_{p}}_{L^2(\Gah)}+ \norm{\theta_{\lambda}}_{L^2(\Gah)}) + ch^{k_g-1}h^{2}\norm{\bfw}_{H^2(\Ga)}(\norm{\theta_{p}}_{L^2(\Gah)}+ \norm{\theta_{\lambda}}_{L^2(\Gah)}) \\
            & \leq c(h^{2k_g-1}+ h^{k_g+m})\norm{\bfw}_{H^2(\Ga)}\big(\norm{\bfu}_{H^{k_u+1}(\Ga)}+\norm{p}_{H^{k_{pr}+1}(\Ga)} + \norm{\lambda}_{H^{k_{\lambda}+1}(\Ga)}\big).
        \end{aligned}
    \end{equation}
    \item \ \ Now, again as above, by adding and subtracting appropriate bilinear terms, applying known interpolation estimates, the geometric perturbations \eqref{eq: Geometric perturbations btilde}, \eqref{eq: Geometric perturbations btilde tangent} and \eqref{eq: Geometric perturbations btilde tangent extra regularity} where appropriate, and the stability for the Scott-Zhang interpolant \eqref{eq: stability of Scott-Zhang interpolant}, we obtain the following bound for $\bfG_b(\Ihzl(\bfw),\{\Ihzl(p),\Ihzl(\lambda)\}) = \bhtil(\Ihz(\bfw^{-\ell}),\{\Ihz(p^{-\ell}),\Ihz(\lambda^{-\ell})\})
            -  b^L(\Ihzl(\bfw),\{\Ihz(p),\Ihz(\lambda)\})$
    \begin{equation}
        \begin{aligned}\label{eq: L^2 estimate proof 7 lagrange}
           &\bfG_b(\Ihzl(\bfw),\{\Ihzl(p),\Ihzl(\lambda)\}) = \underbrace{\bfG_b(\Ihzl(\bfw)-\bfw,\{\Ihzl(p)-p,\Ihzl(\lambda)-\lambda\})}_{\text{by } \eqref{eq: Geometric perturbations btilde} } + \underbrace{\bfG_b(\bfw,\{\Ihzl(p)-p,\Ihzl(\lambda)-\lambda\})}_{\text{by } \eqref{eq: Geometric perturbations btilde tangent} }\\
           & \qquad\qquad\qquad\qquad\qquad\quad \ +\underbrace{\bfG_b(\Ihzl(\bfw)-\bfw,\{p,\lambda\})}_{\text{by } \eqref{eq: Geometric perturbations btilde}} + \underbrace{\bfG_b(\bfw,\{p,\lambda\})}_{\text{by } \eqref{eq: Geometric perturbations btilde tangent extra regularity}}\\
           &\leq ch^{k_g+1}\norm{\bfw}_{H^2(\Ga)}(\norm{p}_{H^1(\Ga)}+\norm{\lambda}_{H^1(\Ga)})
        \end{aligned}
    \end{equation}
    \item \ \ The last estimate of the \eqref{eq: a_T extra regularity error large lagrange} is immediate from a change of the domain of integration \eqref{eq: errors of domain of integration data}
    \begin{equation}
        \begin{aligned}\label{eq: L^2 estimate proof 8 lagrange}
            (\bff,\Ihzl(\bfw))_{L^2(\Ga)} - (\bff_h,\Ihz(\bfw^{-\ell}))_{L^2(\Gah)}&\leq  ch^{k_g+1} \norm{\bff}_{L^2(\Ga)}\norm{\bfw}_{H^2(\Ga)}.
        \end{aligned}
    \end{equation}
\end{itemize}
Combining now \eqref{eq: L^2 estimate proof 1 lagrange}-\eqref{eq: L^2 estimate proof 8 lagrange} onto \eqref{eq: a_T extra regularity error large lagrange} and considering $k_g \geq2$ we get the following estimate
\begin{equation}
    \begin{aligned}\label{eq: al(w,eu)}
        a(\bfw,\eu) \leq c(h^{k_g-1+m}+h^{k_g+1}+h^{2k_g-1})\norm{\bfw}_{H^2(\Ga)}(\norm{\bfu}_{H^{k_u+1}(\Ga)}+\norm{p}_{H^{k_{pr}+1}(\Ga)} + \norm{\lambda}_{H^{k_{\lambda}+1}(\Ga)} + \norm{\bff}_{L^2(\Ga)}).
    \end{aligned}
\end{equation}
Let us now proceed with providing an estimate for the term $b^L(\eu,\{\pi,\mu\})$ of \eqref{eq: Nitsche type equation lagrange}. Using the second equation \eqref{weak lagrange hom} and \eqref{weak lagrange discrete} along with the appropriate calculation we may get
\begin{equation}
    \begin{aligned}\label{eq: b_T extra regularity error large lagrange}
        &b^L(\eu,\{\pi,\mu\}) = b^L(\eu,\{\pi-\Ihzl(\pi),\mu-\Ihzl(\mu) \}) + b^L(\eu,\{\Ihzl(\pi),\Ihzl(\mu)\}) \\
        &=  b^L(\eu,\{\pi-\Ihzl(\pi),\mu-\Ihzl(\mu) \}) + \bhtil(\uh,\{\Ihz(\pi^{-\ell}),\Ihz(\mu^{-\ell})\})  - b^L(\uhl,\{\Ihzl(\pi),\Ihzl(\mu) \})\\
        &= b^L(\eu,\{\pi-\Ihzl(\pi),\mu-\Ihzl(\mu) \}) + \bhtil(\theta_{\bfu},\{\Ihz(\pi^{-\ell}),\Ihz(\mu^{-\ell})\})- b^L(\theta_{\bfu}^{\ell},\{\Ihzl(\pi),\Ihzl(\mu) \})\\
        &\quad  + \bhtil(\Ihz(\bfu^{-\ell}),\{\Ihz(\pi^{-\ell}),\Ihz(\mu^{-\ell})\})  - b^L(\Ihzl(\bfu),\{\Ihzl(\pi),\Ihzl(\mu) \}).
    \end{aligned}
\end{equation}
Once again, let us bound the above terms one by one:
\begin{itemize}
    \item \ \ The first estimate of \eqref{eq: b_T extra regularity error large lagrange} can be calculated by the boundedness of the bilinear form, the Scott-Zhang interpolant \eqref{eq: Scott-Zhang interpolant} and the application of the Korn's inequality \eqref{eq: Korn inequality} for the tangent vector $\eut$
    \begin{equation}
        \begin{aligned}\label{eq: L^2 estimate b proof 1 lagrange}
            &\qquad b^L(\eu,\{\pi-\Ihzl(\pi),\mu-\Ihzl(\mu) \})\\
            &\qquad \ \leq c\norm{\eut}_{a}\norm{\pi-\Ihzl(\pi)}_{L^2(\Ga)} + c\norm{\eu}_{L^2(\Ga)}\norm{\mu-\Ihzl(\mu)}_{L^2(\Ga)}\\
            % &\qquad \ \leq ch\norm{\eut}_{a}\norm{\pi}_{H^1(\Ga)}  \hspace{70mm} \text{ by } \eqref{eq: Korn inequality}\\
            % &\qquad \ \ \ +    
            % ch^{m+1}(\norm{\bfu}_{H^{k_u+1}(\Ga)}+\norm{p}_{H^{k_{pr}+1}(\Ga)} + \norm{\lambda}_{H^{k_{\lambda}+1}(\Ga)})\norm{\mu}_{H^1(\Ga)}\\
            &\qquad \ \leq ch\norm{\eu}_{\ah}\norm{\pi}_{H^1(\Ga)} \\
            &\qquad \ \ \ +    
            ch^{m+1}(\norm{\bfu}_{H^{k_u+1}(\Ga)}+\norm{p}_{H^{k_{pr}+1}(\Ga)} + \norm{\lambda}_{H^{k_{\lambda}+1}(\Ga)})\norm{\mu}_{H^1(\Ga)}\\
            &\qquad \ \leq ch^{m+1}(\norm{\bfu}_{H^{k_u+1}(\Ga)}+\norm{p}_{H^{k_{pr}+1}(\Ga)} + \norm{\lambda}_{H^{k_{\lambda}+1}(\Ga)})(\norm{\pi}_{H^1(\Ga)} + \norm{\mu}_{H^1(\Ga)}),
        \end{aligned}
    \end{equation}
    where in the last equation we used again the fact that $\norm{\cdot}_{a} \leq c \norm{\cdot}_{\ah}$ for $k_g \geq 2$, as in \eqref{eq: L^2 estimate proof 1 lagrange}.
    \item \ \ Considering the improved geometric perturbation \eqref{eq: Geometric perturbations btilde}, and bounds on the Scott-Zhang interpolant \eqref{eq: discrete remainder error Lagrange}, \eqref{eq: stability of Scott-Zhang interpolant} we may get
    \begin{equation} 
        \begin{aligned}\label{eq: L^2 estimate b proof 2 lagrange}
            &\qquad\bhtil(\theta_{\bfu},\{\Ihz(\pi^{-\ell}),\Ihz(\mu^{-\ell})\}) - b^L(\theta_{\bfu}^{\ell},\{\Ihzl(\pi),\Ihzl(\mu) \})\\
            &\qquad \ \leq ch^{k_g}\norm{\theta_{\bfu}}_{L^2(\Ga)}(\norm{\Ihz(\pi^{-\ell})}_{H^1(\Gah)} + \norm{\Ihz(\mu^{-\ell})}_{L^2(\Gah)})\\
            &\qquad \ \leq ch^{k_g+m}(\norm{\bfu}_{H^{k_u+1}(\Ga)}+\norm{p}_{H^{k_{pr}+1}(\Ga)} + \norm{\lambda}_{H^{k_{\lambda}+1}(\Ga)})(\norm{\pi}_{H^1(\Ga)}+ \norm{\mu}_{L^2(\Ga)}).
        \end{aligned}
    \end{equation}
    \item \ \ Similar to \eqref{eq: L^2 estimate proof 7 lagrange}, we add and subtract appropriate bilinear term and using the perturbation bounds \eqref{eq: Geometric perturbations btilde}, \eqref{eq: Geometric perturbations btilde tangent} and \eqref{eq: Geometric perturbations btilde tangent extra regularity} and the Scott-Zhang interpolant  estimates \eqref{eq: Scott-Zhang interpolant}, \eqref{eq: stability of Scott-Zhang interpolant} we obtain
    \begin{equation}
        \begin{aligned}\label{eq: L^2 estimate b proof 3 lagrange}
            \bfG_b(\Ihzl(\bfu),\{\Ihzl(\pi),\Ihzl(\mu)\}) \leq ch^{k_g+1}\norm{\bfu}_{H^2(\Ga)}(\norm{\pi}_{H^1(\Ga)} +\norm{\mu}_{H^1(\Ga)}) 
            %&\bhtil(\Ihz(\bfu^{-\ell}),\{\Ihz(\pi^{-\ell}),\Ihz(\mu^{-\ell})\}) - b^L(\Ihzl(\bfu),\{\Ihzl(\pi),\Ihzl(\mu) \}) \\
             % &\leq ch^{k_g+1}\norm{\bfu}_{H^2(\Ga)}(\norm{\pi}_{H^1(\Ga)} +\norm{\mu}_{H^1(\Ga)}) +ch^{k_g+m}\norm{\bfu}_{H^1(\Ga)}(\norm{\pi}_{H^1(\Ga)}+\norm{\mu}_{L^2(\Ga)}).
        \end{aligned}
    \end{equation}
\end{itemize}
So, combining \eqref{eq: L^2 estimate b proof 1 lagrange}-\eqref{eq: L^2 estimate b proof 3 lagrange} and \eqref{eq: b_T extra regularity error large lagrange} we may obtain
\begin{equation}\label{eq: btil(eu,pi)}
\begin{aligned}
    &b^L(\eu,\{\pi,\mu\}) \\
    & \qquad \leq c(h^{m+1} + h^{k_g+1})(\norm{\bfu}_{H^{k_u+1}(\Ga)}+\norm{p}_{H^{k_{pr}+1}(\Ga)} + \norm{\lambda}_{H^{k_{\lambda}+1}(\Ga)})(\norm{\pi}_{H^1(\Ga)}+ \norm{\mu}_{H^1(\Ga)}).
    \end{aligned}
\end{equation}
Substituting \eqref{eq: al(w,eu)}, \eqref{eq: btil(eu,pi)} into \eqref{eq: Nitsche type equation lagrange} coupled with \eqref{eq: regularity estimate lagrange inside} we finally get the final error
\begin{equation*}
    \begin{aligned}
        \norm{\eut}_{L^2(\Ga)} \leq ch^{m+1}(\norm{\bfu}_{H^{k_u+1}(\Ga)}+\norm{p}_{H^{k_{pr}+1}(\Ga)} + \norm{\lambda}_{H^{k_{\lambda}+1}(\Ga)}+ \norm{\bff}_{L^2(\Ga)}). 
    \end{aligned}
\end{equation*}
\end{proof}

\subsection{Consistency and a-priori estimates for $k_\lambda = k_u$}\label{sec: the kl=ku case}

For the case $\underline{k_\lambda = k_u}$ we are able to derive optimal a-priori estimates \textbf{without} the need of super-parametric finite elements; see in \cref{Theorem: Energy Error Estimate for Lagrange Formulation} and consequently in \cref{Theorem: L^2 error velocity Lagrange} the dependency on the suboptimal geometric approximation order $\bigo(h^{k_g-1})$ and $\bigo(h^{k_g})$ respectively. More specifically, we prove optimal error bounds, with respect to the geometric part of the error, for the velocity $\bfu-\uh$ and the pressure $p-\ph$ in the energy and $L^2$-norm, respectively, in the case of \emph{iso-parametric surface finite elements} for $k_g \geq 1$. Then, similarly, we proceed to prove optimal error bounds in the tangential $L^2$ velocity norm. Finally, we prove error bounds for the normal $L^2$ velocity error.\\

We begin by considering new decompositions for the velocity error, splitting it into a weakly discrete tangential divergence-free approximation and an interpolation error, with the help of a weakly tangential discrete divergence-free interpolant. The remaining errors are decomposed as before in \eqref{eq: decomposition error 2 lagrange}, \eqref{eq: decomposition error 3 lagrange}. Recall also that now $\Lambda_h = S^{k_\lambda}_{h,k_g} = S^{k_u}_{h,k_g} = V_h$.
So we have:
\begin{align}\label{eq: decomposition error 1 lagrange  kl=ku}
    \eu = \bfu^{-\ell} - \uh &= \underbrace{(\bfu^{-\ell}- \vhtilde)}_{\text{Interpolation error}} + \underbrace{(\vhtilde - \uh)}_{\text{discrete remainder}}\!\!\!\!\!\!,
\end{align}
where $\vhtilde \in \bfV_h^{div}$. We denote $\rho_\bfu^{div} = \vhtilde - \bfu^{-\ell}$ and $\theta_{\bfu}^{div} = \uh - \vhtilde$, where now $\theta_{\bfu}^{div} \in \bfV_h^{div}$. So recalling \cref{lemma: Lh consistency}, we can readily see that our new discretisation errors satisfy the following error equation:
\begin{lemma}\label{lemma: Lh consistency kl=ku}
The discretisation errors satisfy,
\begin{align}
\begin{cases}\label{eq: Lh consistency kl=ku}
        \ah(\theta_{\bfu}^{div},\vh)\,+\!\!\!\!\!&\bhtil(\vh,\{\theta_{p},\theta_{\lambda}\}) = \ah(\rho_{\bfu}^{div},\vh) + \bhtil(\vh,\{\rho_{p},\rho_{\lambda}\}) \, + \,  \sum_{i=1}^3 \text{Err}_{i}^{L}(\vh) \ \  \text{ for all } \vh \in \bfV_h,\\
        & \bhtil(\theta_{\bfu}^{div},\{\qh,\xi_h\})= 0 \quad \ \ \text{for all } \{\qh,\xi_h\}\in Q_h\times \Lambda_h,
    \end{cases}
\end{align}
with the consistency errors:
\begin{itemize}
\setlength\itemsep{0.7em}
\item \ \ $ \text{Err}_1^{L}(\vh) := a(\bfu,\vhl) - \ah(\bfu^{-\ell},\vh)$,
\item \ \ $ \text{Err}_2^{L}(\vh) := b^L(\vhl,\{p,\lambda\})-\bhtil(\vh,\{p^{-\ell},\lambda^{-\ell}\})$,
\item \ \ $ \text{Err}_3^{L}(\vh) := (\bff_h,\vh)_{L^2(\Gah)} - (\bff,\vhl)_{L^2(\Ga)}$.
\end{itemize}
\end{lemma}
Now, as before in \cref{Lemma: Consistency error bounds Lagrange} we have the following new consistency bounds, which we instead write with respect to the $H^1$ and $L^2$ norms.
\begin{lemma}[Consistency error bounds]\label{Lemma: Consistency error bounds Lagrange kl=ku}
Let $(\vh,\{\qh,\lh\}) \in \bfV_h\times Q_h \times \Lambda_h$ and $(\bfu,\{p,\lambda\})\in \bfH^1(\Ga) \times L^2_0(\Ga)\times L^2(\Ga)$ be the solution to the continuous problem \eqref{weak lagrange hom}, then the consistency error terms $\text{Err}_i^{L}(\cdot)$, i=1,...,3 as in Lemma \ref{lemma: Lh consistency} are bounded as followed:
\begin{equation}\label{eq: Consistency error bounds Lagrange kl=ku}
    \begin{aligned}
    |\text{Err}_{1}^{L}(\vh)| &\leq ch^{k_g}\norm{\bfu}_{H^1(\Ga)}\norm{\vh}_{H^1(\Gah)},\\[5pt]
        |\text{Err}_{2}^{L}(\vh)| &\leq ch^{k_g}(\norm{p}_{H^1(\Ga)}+\norm{\lambda}_{L^2(\Ga)})\norm{\vh}_{L^2(\Gah)},\\[5pt]
         |\text{Err}_{3}^{L}(\vh)| &\leq ch^{k_g+1}\norm{\bff}_{L^2(\Ga)}\norm{\vh}_{L^2(\Gah)},\\[5pt]
    \end{aligned}
\end{equation}
where the constants $c$ are independent of the mesh-parameter $h$.
\end{lemma}

\subsubsection{Energy error estimates}
The following lemma follows from \cref{lemma: improved h1-ah bound} and the fact that the discrete remainder $\thbfu^{div}\in \bfV_h^{div}$. This is an improved $H^1$ coercivity bound.

\begin{lemma}\label{lemma: new H1 estimate thbfu kl=ku}
If $\underline{k_{\lambda} = k_u}$ then for $h$ sufficiently small the following improved bound holds
    \begin{equation}\label{eq: new H1 estimate thbfu 1 kl=ku}
        \begin{aligned}
            \norm{\thbfu^{div}}_{H^1(\Gah)} \leq c\norm{\thbfu^{div}}_{\ah}.
        \end{aligned}
    \end{equation}
\end{lemma}

\begin{theorem}\label{Theorem: Improved Energy Error Estimate for Lagrange Formulation kl=ku}
Let $k_u  = k_\lambda\geq 2$, $k_{pr}=k_u-1$ and $h$ sufficiently small. Then for $(\bfu,\{p,\lambda\}) \in \bfH^1(\Ga) \times L^2_0(\Ga)\times L^2(\Ga)$ the solution of the continuous problem \eqref{weak lagrange hom}, and $(\uh,$ $\{\ph,\lambda_h\}) \in \bfV_h \times Q_h \times \Lambda_h$ the solution to the discrete scheme \eqref{weak lagrange discrete} we have the following error estimates
\begin{equation}
    \begin{aligned}\label{eq: Improved Energy Error Estimate for Lagrange Formulation 1}
        \norm{\bfu^{-\ell} - \uh}_{\ah} + \norm{p^{-\ell} - \ph}_{L^2(\Gah)} &+ \norm{\lambda^{-\ell} - \lh}_{H_h^{-1}(\Gah)}
         \leq c\inf_{\vh\in\bfV_h^{div}}\norm{\bfu^{-\ell} - \vh}_{\ah} \\
        &+ ch^s\big(\norm{p}_{H^{k_{pr}+1}(\Ga)} + \norm{\lambda}_{H^{k_{\lambda}+1}(\Ga)}+ \norm{\bff}_{L^2(\Ga)}\big),
    \end{aligned}
\end{equation}
with $s= min\{{k_{pr}+1}, {k_{\lambda}+1}, {k_g}\}$.
\end{theorem}
\begin{proof}
\ Let us start with the discretization errors. To find a bound for $\thbfu^{div}$ we go back to the new error equation \eqref{eq: Lh consistency kl=ku}, and test the first equation with $\theta_{\bfu}^{div} \in \bfV_h^{div}$. Then, the consistency error bounds \eqref{eq: Consistency error bounds Lagrange kl=ku} and simple calculations give
\begin{equation}
    \begin{aligned}
        \norm{\theta_{\bfu}^{div}}_{\ah}^2 &\leq  \norm{\rho_{\bfu}^{div}}_{\ah}\norm{\theta_{\bfu}^{div}}_{\ah} + \norm{\theta_{\bfu}^{div}}_{\ah}\norm{\{\rho_p,\rho_\lambda\}}_{L^2(\Gah)} + ch^{k_g}(\norm{\bff}_{L^2(\Ga)}+ \norm{\bfu}_{H^1(\Ga)})\norm{\theta_{\bfu}^{div}}_{H^1(\Gah)}.
    \end{aligned}   
\end{equation}
With the help of the new coercivity bound in \cref{lemma: new H1 estimate thbfu kl=ku} and the interpolation bounds of $\rho_p,\rho_\lambda$, it is easy to see, after a simple kickback argument, that
\begin{equation}
    \begin{aligned}\label{eq: h1 imrpoved inside 1 kl=ku}
        \norm{\theta_{\bfu}^{div}}_{\ah} \leq c\norm{\rho_{\bfu}^{div}}_{\ah} + ch^{s}(\norm{p}_{H^{k_{pr}+1}(\Ga)} + \norm{\lambda}_{H^{k_{\lambda}+1}(\Ga)}+ \norm{\bff}_{L^2(\Ga)})
    \end{aligned}
\end{equation}
with $s = min\{h^{k_{pr+1}}, h^{k_{\lambda}+1}, h^{k_g}\}$. Moving onto the pressure estimates, with the help of the $L^2\times H_h^{-1}$ \emph{discrete inf-sup} \eqref{eq: L^2 H^{-1} discrete inf-sup condition Gah Lagrange}, the error equation \eqref{eq: Lh consistency kl=ku} along with the consistency errors \eqref{eq: Consistency error bounds Lagrange kl=ku}, we obtain
\begin{equation}
    \begin{aligned}\label{eq: improved pressure bound before kl=ku}
        \norm{\{\theta_p,\theta_\lambda\}}_{L^2(\Gah)\times H_h^{-1}(\Gah)} &\leq \norm{\rho_{\bfu}^{div}}_{\ah} + \norm{\{\rho_{p},\rho_{\lambda}\}}_{L^2(\Gah)} + \norm{\thbfu^{div}}_{\ah} + \sup_{\vh \in \bfV_h} \frac{Err_1^{L}(\vh)}{\norm{\vh}_{H^1(\Gah)}}\\
        &\leq \norm{\rho_{\bfu}^{div}}_{\ah} + ch^{s}(\norm{p}_{H^{k_{pr}+1}(\Ga)} + \norm{\lambda}_{H^{k_{\lambda}+1}(\Ga)}+ \norm{\bff}_{L^2(\Ga)}),
    \end{aligned}
\end{equation}
with $s = min\{{k_{pr+1}}, {k_{\lambda}+1}, {k_g}\}$. To complete the proof, we consider the decompositions \eqref{eq: decomposition error 1 lagrange  kl=ku} \eqref{eq: decomposition error 2 lagrange}, and \eqref{eq: decomposition error 3 lagrange}, use the triangle inequality, and combine the estimates \eqref{eq: h1 imrpoved inside 1 kl=ku} and \eqref{eq: improved pressure bound before kl=ku}.
\end{proof}

\begin{lemma}\label{lemma: inf need to show kl=ku}
Let the $L^2\times L^2$ discrete inf-sup condition \cref{Lemma: Discrete inf-sup condition Gah Lagrange} hold. Then have the following bound
    \begin{equation}
        \begin{aligned}\label{eq: inf need to show kl=ku}
            \inf_{\vh\in\bfV_h^{div}}\norm{\bfu^{-\ell} - \vh}_{\ah} \leq \inf_{\wh\in\bfV_h}\norm{\bfu^{-\ell} - \wh}_{\ah} + ch^{k_g}\norm{\bfu}_{H^1(\Ga)}.
        \end{aligned}
    \end{equation}
Furthermore, if we now consider the  $L^2\times H_h^{-1}$ discrete inf-sup condition \cref{lemma: L^2 H^{-1} discrete inf-sup condition Gah Lagrange} instead, then we have the following bound
\begin{equation}
        \begin{aligned}\label{eq: inf need to show kl=ku H1}
            \inf_{\vh\in\bfV_h^{div}}\norm{\bfu^{-\ell} - \vh}_{H^1(\Gah)} \leq \inf_{\wh\in\bfV_h}\norm{\bfu^{-\ell} - \wh}_{H^1(\Gah)} + ch^{k_g}\norm{\bfu}_{H^1(\Ga)}.
        \end{aligned}
    \end{equation}
\end{lemma}
\begin{proof}
  For the first inequality, consider the discrete inf-sup \cref{Lemma: Discrete inf-sup condition Gah Lagrange}. First, the adjoint of $\bhtil$ can be calculated thanks to the integration by parts formula \eqref{eq: integration by parts} as  
\begin{equation}
    b_h^{L,*}(\vh,\{\qh,\xi_h\}) = - \sum_{T\in \mathcal{T}_h}\int_{T} \qh \divgh \bfPh\vh \, \dsh + \sum_{E \in \mathcal{E}_h} \int_E [\mh \cdot \vh] \qh \, d\ell + \int_{\Gah}\vh\cdot\nh \xi_h \, \dsh.
\end{equation}
Now let $\wh\in\bfV_h$ and choose $\vh \in \bfV_h^{div}$, such that $\wh-\vh \in (\bfV_h^{div})^{\perp}$. Due to the discrete inf-sup \cref{Lemma: Discrete inf-sup condition Gah Lagrange} we can see that the following holds
\begin{equation}
    \begin{aligned}\label{eq: corol inside 1 kl=ku}
        \norm{\vh-\wh}_{\ah} \leq \sup_{\{\qh,\xi_h\}\in Q_h\times\Lambda_h}\frac{ b_h^{L,*}(\vh-\wh,\{\qh,\xi_h\})}{\norm{\{\qh,\xi_h\}}_{L^2(\Gah)}}.
    \end{aligned}
\end{equation}
Adding and subtracting the true solution $\bfu$, using the fact that $\vh\in \bfV_h^{div}$, noting the estimates $|[\mh]| \leq ch^{k_g}$ and $|\bfP[\mh]| \leq ch^{2k_g}$ (cf. \cite[Lemma 3.5]{OlshReusXu2013}) and the consistency errors \cref{eq: Geometric perturbations btilde}, \cite[Lemma 6.6]{jankuhn2021Higherror}, we readily see  for sufficiently small $h$ that
\begin{equation}
    \begin{aligned}\label{eq: corol inside 2 kl=ku}
         &b_h^{L,*}(\vh-\wh,\{\qh,\xi_h\}) =  b_h^{L,*}(\bfu^{-\ell}-\wh,\{\qh,\xi_h\}) - b_h^{L,*}(\bfu^{-\ell},\{\qh,\xi_h\})\\
         &\leq c\norm{\bfu^{-\ell}-\wh}_{\ah}\norm{\qh}_{L^2(\Gah)} - \int_{\Gah}\qh \divgh \bfu^{-\ell} + \underbrace{\int_{\Ga}\qhl \divg \bfu}_{:=0} + \underbrace{\int_{\Gah}\bfu^{-\ell}\cdot(\nh-\bfn) \xi_h}_{\bfu^{-\ell}\cdot\bfn =0, \ \  \norm{\nh-\bfn}_{L^{\infty}}\leq ch^{k_g}} \\
         &\ + \sum_{E \in \mathcal{E}_h} \int_E [\mh] \cdot \bfu^{-\ell} \qh \, d\ell \\
         &\leq c\norm{\bfu^{-\ell}-\wh}_{\ah}\norm{\qh}_{L^2(\Gah)} + ch^{k_g}\norm{\bfu}_{H^1(\Ga)}\norm{\{\qh,\xi_h\}}_{L^2(\Gah)} + ch^{2k_g-1/2}\norm{\bfu}_{H^1(\Ga)}\norm{\qh}_{L^2(\Gah)},
    \end{aligned}
\end{equation}
where we also used the norm equivalence where appropriate. Now, combining \eqref{eq: corol inside 1 kl=ku}, \eqref{eq: corol inside 2 kl=ku}, using the following triangle inequality
\begin{equation}\label{eq: tring ineq inf proof}
   \norm{\bfu^{-\ell} - \vh}_{\ah} \leq \norm{\bfu^{-\ell} - \wh}_{\ah} + \norm{\wh - \vh}_{\ah}
\end{equation}
and taking the infimum over all $\vh \in \bfV_h^{div}$ and $\wh \in \bfV_h$ gives our assertion \eqref{eq: inf need to show kl=ku}. The second inequality \eqref{eq: inf need to show kl=ku H1} follows similarly if we use the $L^2\times H_h^{-1}$ discrete inf-sup condition \cref{lemma: L^2 H^{-1} discrete inf-sup condition Gah Lagrange} instead. We just need to bound the last term in the second line of \eqref{eq: corol inside 2 kl=ku} appropriately. We use the fact that $\norm{\xi_h}_{L^2(\Gah)} \leq ch^{-1}\norm{\xi_h}_{H_h^{-1}(\Gah)}$, by def. \eqref{eq: H^-1h definition 2} and the inverse inequality $\norm{\xi_h}_{H^1(\Gah)} \leq ch^{-1}\norm{\xi_h}_{L^2(\Gah)}$, to obtain
\begin{equation}
    \begin{aligned}
        &\int_{\Gah}\bfu^{-\ell}\cdot\nh \xi_h = \int_{\Gah}(\bfu^{-\ell}-\vh)\cdot(\nh-\bfn) \xi_h + \int_{\Gah}(\vh-\wh)\cdot\bfn\xi_h + \int_{\Gah}\wh\cdot\bfn\xi_h\\
        &\leq ch^{k_g-1}\norm{\bfu^{-\ell}-\vh}_{L^2(\Gah)}\norm{\xi_h}_{H_h^{-1}(\Gah)} + \int_{\Gah}\Ihz((\vh-\wh)\cdot\bfn)\xi_h + \int_{\Gah}((\vh-\wh)\cdot\bfn- \Ihz((\vh-\wh)\cdot\bfn))\xi_h\\
        &\  + \int_{\Gah}\Ihz(\wh\cdot\bfn)\xi_h+  \int_{\Gah}(\wh\cdot\bfn- \Ihz(\wh\cdot\bfn))\xi_h \\
        &\leq c\norm{\bfu^{-\ell}-\vh}_{L^2(\Gah)}\norm{\xi_h}_{H_h^{-1}(\Gah)} +c\norm{\vh-\wh}_{L^2(\Gah)}\norm{\xi_h}_{H_h^{-1}(\Gah)}+ c\norm{\wh-\bfu^{-\ell}}_{H^1(\Gah)}\norm{\xi_h}_{H_h^{-1}(\Gah)},
    \end{aligned}
\end{equation}
where we used the super approximation properties \eqref{Lemma: super-approximation properties}, the interpolation estimate \eqref{eq: Scott-Zhang interpolant} and the fact that $\bfu^{-\ell}\cdot \bfn=0$. Therefore, recalling \eqref{eq: corol inside 2 kl=ku} and the $L^2\times H_h^{-1}$ inf-sup \cref{lemma: L^2 H^{-1} discrete inf-sup condition Gah Lagrange} we see that 
\begin{equation*}
    \begin{aligned}
        \norm{\vh-\wh}_{H^1(\Gah)} \leq c\norm{\bfu^{-\ell}-\vh}_{\ah} + c\norm{\bfu^{-\ell} - \wh}_{H^1(\Gah)} +  ch^{k_g}\norm{\bfu}_{H^1(\Ga)}.
    \end{aligned}
\end{equation*}
Using the triangle inequality we see that
\begin{equation*}
    \begin{aligned}
        \inf_{\vh\in\bfV_h^{div}}\norm{\bfu^{-\ell} - \vh}_{H^1(\Gah)} \leq \norm{\bfu^{-\ell} - \vh}_{H^1(\Gah)} &\leq \norm{\bfu^{-\ell} - \wh}_{H^1(\Gah)} + \norm{\wh - \vh}_{H^1(\Gah)}\\
        &\leq c\norm{\bfu^{-\ell}-\vh}_{\ah} + c\norm{\bfu^{-\ell} - \wh}_{H^1(\Gah)} +  ch^{k_g}\norm{\bfu}_{H^1(\Ga)},
    \end{aligned}
\end{equation*}
where taking the infimum over $\vh\in\bfV_h^{div}$ and $\wh \in \bfV_h$ and considering the previously proven estimate \eqref{eq: inf need to show kl=ku} we get our desired result.
\end{proof}

\begin{theorem}[Improved Error Estimates for Lagrange Formulation]\label{Corrolary: Improved Energy Error Estimate for Lagrange Formulation kl=ku}
    Let the assumptions in \cref{Theorem: Improved Energy Error Estimate for Lagrange Formulation kl=ku} hold. Then, the solution  $(\uh,$ $\{\ph,\lambda_h\}) \in \bfV_h \times Q_h \times \Lambda_h$ to the discrete scheme \eqref{weak lagrange discrete} satisfies the following error bounds
    \begin{equation}
        \begin{aligned} \label{eq: Improved Energy Error Estimate for Lagrange Formulation kl=ku new}
       \norm{\bfu^{-\ell} - \uh}_{\ah} + \norm{p^{-\ell} - \ph}_{L^2(\Gah)} &+ \norm{\lambda^{-\ell} - \lh}_{H_h^{-1}(\Gah)}\\ 
        &\leq ch^{\widehat{m}}\big(\norm{\bfu}_{H^{k_u+1}(\Ga)}+\norm{p}_{H^{k_{pr}+1}(\Ga)} + \norm{\lambda}_{H^{k_{\lambda}+1}(\Ga)}+ \norm{\bff}_{L^2(\Ga)}\big),
    \end{aligned}
\end{equation}
with $\widehat{m}= min\{{k_u},{k_{pr}+1}, {k_{\lambda}+1}, {k_g}\}$.
\end{theorem}
\begin{proof}
These improved error estimates can be shown by combining the previously proven results in \cref{Theorem: Improved Energy Error Estimate for Lagrange Formulation kl=ku} and \eqref{eq: inf need to show kl=ku} together with standard interpolation estimates; we may consider the Scott-Zhang interpolant for the right-hand side of \eqref{eq: inf need to show kl=ku} and the corresponding estimates \eqref{eq: Scott-Zhang interpolant},  \eqref{eq: interpolation vh ah}.
\end{proof}

\begin{corollary}[$H^1$ Error Bound]\label{corollary: H^1 bound kl-ku}
    Let the assumptions in \cref{Theorem: Improved Energy Error Estimate for Lagrange Formulation kl=ku} hold. Then, the solution  $(\uh,$ $\{\ph,\lambda_h\}) \in \bfV_h \times Q_h \times \Lambda_h$ to the discrete scheme \eqref{weak lagrange discrete} satisfies the following error bounds
    \begin{equation}
        \begin{aligned} \label{eq: Improved H1 Error Estimate for Lagrange Formulation kl=ku new}
       \norm{\bfu^{-\ell} - \uh}_{H^1(\Gah)} \leq ch^{\widehat{m}}\big(\norm{\bfu}_{H^{k_u+1}(\Ga)}+\norm{p}_{H^{k_{pr}+1}(\Ga)} + \norm{\lambda}_{H^{k_{\lambda}+1}(\Ga)}+ \norm{\bff}_{L^2(\Ga)}\big)
    \end{aligned}
\end{equation}
with $\widehat{m}= min\{{k_u},{k_{pr}+1}, {k_{\lambda}+1}, {k_g}\}$ as in \cref{Corrolary: Improved Energy Error Estimate for Lagrange Formulation kl=ku}.
\end{corollary}
\begin{proof}
    Considering $\vhtilde\in \bfV_h^{div}$, $\thbfu = \vhtilde - \uh \in \bfV_h^{div}$ and the $H^1$ coercivity bound \eqref{eq: new H1 estimate thbfu 1 kl=ku} we see that
\begin{equation*}
    \begin{aligned}
        \norm{\bfu^{-\ell} - \uh}_{H^1(\Gah)} &\leq \norm{\bfu^{-\ell}-\vhtilde}_{H^1(\Gah)} + \norm{\thbfu}_{H^1(\Gah)} \leq \norm{\bfu^{-\ell}-\vhtilde}_{H^1(\Gah)} + \norm{\thbfu}_{\ah} \\
        &\leq \norm{\bfu^{-\ell}-\vhtilde}_{H^1(\Gah)} + \norm{\bfu^{-\ell}-\vhtilde}_{\ah} + \norm{\bfu^{-\ell} - \uh}_{\ah}.
    \end{aligned}
\end{equation*}
The proof follows by taking infimum over all $\vhtilde\in \bfV_h^{div}$, and applying \cref{lemma: inf need to show kl=ku} and \eqref{eq: Improved Energy Error Estimate for Lagrange Formulation kl=ku new}.
\end{proof}

\subsubsection{$L^2$-norm velocity error bounds}

\begin{theorem}[Improved $L^2$ Error estimate for the tangential velocity]\label{Theorem: improved L^2 error velocity Lagrange}
    Let $k_u  = k_\lambda\geq 2$, $k_{pr}=k_u-1$ and $h$ sufficiently small. Then for
   $(\bfu,p,\lambda)$ the solution of the continuous problem \eqref{weak lagrange hom} and assume further regularity $(\bfu,p,\lambda) \in (H^2(\Ga))^3\times H^1(\Ga)\times H^1(\Ga)$, and  $(\uh,\ph,\lh) \in \bfV_h \times Q_h\times \Lambda_h$ the solution of the discrete scheme \eqref{weak lagrange discrete}, the following estimate holds for the tangential part of the velocity
    \begin{equation}
        \begin{aligned}\label{eq: improved L^2 error velocity Lagrange 1}
            \norm{\eut}_{L^2(\Ga)}  &= \norm{\bfPg(\bfu-\uhl)}_{L^2(\Ga)} \\
            & \leq ch^{\widehat{m}+1}(\norm{\bfu}_{H^{k_u+1}(\Ga)}+\norm{p}_{H^{k_{pr}+1}(\Ga)} + \norm{\lambda}_{H^{k_{\lambda}+1}(\Ga)}+ \norm{\bff}_{L^2(\Ga)}),
        \end{aligned}
    \end{equation}
    where $\widehat{m}= min\{{k_u},{k_{pr}+1}, {k_{\lambda}+1}, {k_g}\}$ as in \cref{Corrolary: Improved Energy Error Estimate for Lagrange Formulation kl=ku}.
\end{theorem}

\begin{proof}
The proof follows exactly as in \cref{Theorem: L^2 error velocity Lagrange}, where now we use the previous improved estimates instead, that is, \cref{Corrolary: Improved Energy Error Estimate for Lagrange Formulation kl=ku}. However, there are some terms that still need extra attention, mainly for the case $k_g=1$, since in \cref{Theorem: L^2 error velocity Lagrange} we see that a geometric approximation error $\bigo(h^{2k_g-1})$ arises. 
% This implies that for $k_g=1$ we do not get optimal bounds for the tangential $L^2$ norm of the velocity.

We will go through the error bounds appearing in  \cref{Theorem: L^2 error velocity Lagrange} and try to attain higher-order convergence. Let us start by finding an improved bound for $a(\bfw,\eu)$ of \eqref{eq: Nitsche type equation lagrange}:
\begin{description}
    \item[Terms \eqref{eq: L^2 estimate proof 1 lagrange}, \eqref{eq: L^2 estimate proof 5 lagrange}, \eqref{eq: L^2 estimate proof 7 lagrange} and \eqref{eq: L^2 estimate proof 8 lagrange}: ] These are bounded as in \cref{Theorem: L^2 error velocity Lagrange} but now we use the improved estimates in \cref{Corrolary: Improved Energy Error Estimate for Lagrange Formulation kl=ku} where appropriate, so:
    \begin{equation}
       \eqref{eq: L^2 estimate proof 1 lagrange},\, \eqref{eq: L^2 estimate proof 5 lagrange},\, \eqref{eq: L^2 estimate proof 7 lagrange} \text{ and } \eqref{eq: L^2 estimate proof 8 lagrange} \leq  ch^{\widehat{m}+1}(\norm{\bfu}_{H^{k_u+1}(\Ga)}+\norm{p}_{H^{k_{pr}+1}(\Ga)} + \norm{\lambda}_{H^{k_{\lambda}+1}(\Ga)}+ \norm{\bff}_{L^2(\Ga)}).
    \end{equation}
    Notice also that \eqref{eq: L^2 estimate proof 1 lagrange} now also holds for $k_g=1$, due to 
    \cref{corollary: H^1 bound kl-ku}.
   
    \item[Terms \eqref{eq: L^2 estimate proof 2 lagrange} and \eqref{eq: L^2 estimate proof 3 lagrange}: ]  These two terms came up by adding and subtracting $\Ihz(\bfu)$ at the terms $\ah(\Ihz(\bfw^{-\ell}),\uh) - a(\Ihzl(\bfw),\uhl)$ in \eqref{eq: a_T extra regularity error large lagrange old}. Now instead of that, we add and subtract $\vhtilde\in\bfV_h^{div}$ a weakly discretely tangential divergence free "interpolant", and consider the estimate in  \eqref{eq: inf need to show kl=ku H1}. Therefore, we re-write \eqref{eq: L^2 estimate proof 2 lagrange} and \eqref{eq: L^2 estimate proof 3 lagrange} as
    \begin{equation}
        \begin{aligned}
            &\ah(\Ihz(\bfw^{-\ell}),\theta_{\bfu}^{div})- a(\Ihzl(\bfw),\theta_{\bfu}^{div,\ell}) + \bfG(\Ihzl(\bfw),\vhtilde^{\ell}) \\
            &\leq ch^{k_g}\norm{\Ihz(\bfw^{-\ell})}_{H^1(\Gah)}\norm{\theta_{\bfu}^{div}}_{H^1(\Gah)} + ch^{k_g}\norm{\vhtilde-\bfu^{-\ell}}_{H^{1}(\Gah)} + ch^{k_g+1} \norm{\bfw}_{H^2(\Ga)}\norm{\bfu}_{H^2(\Ga)}\\
            &\leq c(h^{k_g + \widehat{m}}+ h^{2k_g}+ h^{k_g+1})\norm{\bfw}_{H^2(\Ga)}(\norm{\bfu}_{H^{k_u+1}(\Ga)} + \norm{p}_{H^{k_{pr}+1}(\Ga)} + \norm{\lambda}_{H^{k_{\lambda}+1}(\Ga)}), 
        \end{aligned}
    \end{equation}
where we have used \cref{Corrolary: Improved Energy Error Estimate for Lagrange Formulation kl=ku} (more specifically in this case the bound for $\theta_{\bfu}^{div}$; see proof) and \eqref{eq: inf need to show kl=ku H1} after taking the infimum over all $\vhtilde \in \bfV_h^{div}$.

\item[Term \eqref{eq: L^2 estimate proof 6 lagrange}: ] For this estimate we need to be more careful. Consider $\widetilde{\bfw}_h \in \bfV_h^{div}$ and remember by \cref{lemma: inf need to show kl=ku} that $\inf_{\wh\in\bfV_h^{div}}\norm{\bfw^{-\ell} - \wh}_{H^1(\Gah)} \leq \inf_{\wh\in\bfV_h}\norm{\bfw^{-\ell} - \wh}_{H^1(\Gah)} + ch^{k_g}\norm{\bfw}_{H^1(\Ga)} \leq ch\norm{\bfw}_{H^2(\Ga)}$. So, since also $\bfw\cdot\bfng=0$, we have that
\begin{equation*}
    \begin{aligned}
        &\eqref{eq: L^2 estimate proof 6 lagrange} = b_h(\Ihz(\bfw^{-\ell})- \bfw^{-\ell},\{\theta_{p},\theta_{\lambda}\})- b(\Ihzl(\bfw)-\bfw,\{\theta_{p}^{\ell},\theta_{\lambda}^{\ell}\}) + b_h(\bfw^{-\ell},\theta_{p})-  b(\bfw,\theta_{p}^{\ell})\\
        &\quad \  + \int_{\Gah}\Ihz(\bfw^{-\ell})\cdot\nh\theta_{\lambda} \, \dsh  - \int_{\Ga}(\Ihzl(\bfw)-\bfw)\cdot\bfn\theta_{\lambda}^{\ell} \, \dsh\\
        &\leq ch^{k_g}\norm{\bfw}_{L^2(\Ga)}\norm{\theta_{p}}_{L^2(\Gah)} + ch^{2}\norm{\bfw}_{L^2(\Ga)}\norm{\theta_{\lambda}}_{L^2(\Gah)}\\
        &\quad \ + \int_{\Gah}\big(\underbrace{(\Ihz(\bfw^{-\ell})- \widetilde{\bfw}_h)}_{\bfr_h\in\bfV_h}\big)\cdot(\nh-\bfn)\theta_{\lambda} \, \dsh + \int_{\Gah}\underbrace{(\Ihz(\bfw^{-\ell})- \widetilde{\bfw}_h)}_{\bfr_h\in\bfV_h}\cdot\bfn\theta_{\lambda} \, \dsh \\
        &\leq ch^{k_g}\norm{\bfw}_{L^2(\Ga)}\norm{\theta_{p}}_{L^2(\Gah)} + ch^{2}\norm{\bfw}_{L^2(\Ga)}\norm{\theta_{\lambda}}_{L^2(\Gah)} + ch^{k_g}\norm{\bfr_h}_{L^2(\Gah)}\norm{\theta_{\lambda}}_{L^2(\Gah)}\\
        &\quad \ + \int_{\Gah} \underbrace{\Ihz\big(\bfr_h\cdot\bfn\big)}_{\in V_h = \Lambda_h}\theta_{\lambda} \, \dsh +  \int_{\Gah} \underbrace{\big(\bfr_h\cdot\bfn - \Ihz(\bfr_h\cdot\bfn)\big)}_{\leq ch\norm{\bfr_h}_{L^2(\Gah)} \text{ by \eqref{eq: super-approximation estimate 2}}}\theta_{\lambda} \, \dsh\\
        &\hspace{-5mm}\overset{ \text{Def.}\eqref{eq: H^-1h definition}}{\leq} \!\!\!\!\!ch^{k_g}\norm{\bfw}_{L^2(\Ga)}\norm{\theta_{p}}_{L^2(\Gah)} +ch^{2}\norm{\bfw}_{L^2(\Ga)}\norm{\theta_{\lambda}}_{L^2(\Gah)} + ch^{k_g}\norm{\bfr_h}_{L^2(\Gah)}\norm{\theta_{\lambda}}_{L^2(\Gah)} \\
        &\quad \ +\norm{ \Ihz\big(\bfr_h\cdot\bfn\big)}_{H^1(\Gah)}\norm{\theta_{\lambda}}_{H_h^{-1}(\Gah)}.
    \end{aligned}
\end{equation*}
Now considering the stability estimate \eqref{eq: stability of Scott-Zhang interpolant}, using the fact that $\norm{\bfr_h}_{L^2(\Gah)} \leq \norm{\Ihz(\bfw^{-\ell})-\bfw^{-\ell}}_{L^2(\Gah)}+\norm{\bfw^{-\ell}-\widetilde{\bfw}_h}_{L^2(\Gah)}$, taking the infimum over all $\widetilde{\bfw}_h\in \bfV_h^{div}$ and applying the previous estimates \eqref{Lemma: discrete remainder error Lagrange} (they still hold in this case) and \cref{Theorem: Improved Energy Error Estimate for Lagrange Formulation kl=ku} where appropriate, we get
\begin{equation}
    \eqref{eq: L^2 estimate proof 6 lagrange}  \leq c(h^{k_g + \widehat{m}} + h^{k_g+1})\norm{\bfw}_{H^2(\Ga)}(\norm{\bfu}_{H^{k_u+1}(\Ga)} + \norm{p}_{H^{k_{pr}+1}(\Ga)} + \norm{\lambda}_{H^{k_{\lambda}+1}(\Ga)}  + \norm{\bff}_{L^2(\Ga)})
\end{equation}
\item[Term \eqref{eq: al(w,eu)}: ] Therefore combining the above new results we get the improved bound
\begin{equation}\label{eq: al(w,eu) new}
    \eqref{eq: al(w,eu)} \leq c(h^{\widehat{m}+1}+h^{k_g+1})\norm{\bfw}_{H^2(\Ga)}(\norm{\bfu}_{H^{k_u+1}(\Ga)}+\norm{p}_{H^{k_{pr}+1}(\Ga)} + \norm{\lambda}_{H^{k_{\lambda}+1}(\Ga)} + \norm{\bff}_{L^2(\Ga)}).
\end{equation}
\end{description}
\noindent Let us now proceed with providing an estimate for the term $b^L(\eu,\{\pi,\mu\})$ of \eqref{eq: Nitsche type equation lagrange}. Again, we follow the estimates in \cref{Theorem: L^2 error velocity Lagrange} and we improve upon them:
\begin{description}
\item[Term \eqref{eq: L^2 estimate b proof 1 lagrange}: ] Again it is clear that if we use the improved estimates in \cref{Corrolary: Improved Energy Error Estimate for Lagrange Formulation kl=ku} instead 
\begin{equation}\label{eq: kl inside L2 b 1}
    \eqref{eq: L^2 estimate b proof 1 lagrange}\leq  ch^{\widehat{m}+1}(\norm{\bfu}_{H^{k_u+1}(\Ga)}+\norm{p}_{H^{k_{pr}+1}(\Ga)} + \norm{\lambda}_{H^{k_{\lambda}+1}(\Ga)}+ \norm{\bff}_{L^2(\Ga)}).
\end{equation}
\item[Terms \eqref{eq: L^2 estimate b proof 2 lagrange}, \eqref{eq: L^2 estimate b proof 3 lagrange}: ] Similar to Terms \eqref{eq: L^2 estimate proof 2 lagrange} and \eqref{eq: L^2 estimate proof 3 lagrange} we add and subtract $\vhtilde$,  instead of $\Ihz(\bfu)$, in \eqref{eq: b_T extra regularity error large lagrange}. Then using, yet again, the improved estimates in \cref{Corrolary: Improved Energy Error Estimate for Lagrange Formulation kl=ku} (more specifically the bound for $\theta_{\bfu}^{div}$) we see that 
\begin{equation}
    \begin{aligned}\label{eq: kl inside L2 b 2}
        &\bhtil(\theta_{\bfu}^{div},\{\Ihz(\pi^{-\ell}),\Ihz(\mu^{-\ell})\}) - b^L(\theta_{\bfu}^{div,\ell},\{\Ihzl(\pi),\Ihzl(\mu) \}) + \bfG_b(\vhtilde,\{\Ihzl(\pi),\Ihzl(\mu)\}) \\
        &\leq ch^{k_g}\norm{\theta_{\bfu}^{div}}_{L^2(\Ga)}(\norm{\Ihz(\pi^{-\ell})}_{H^1(\Gah)} + \norm{\Ihz(\mu^{-\ell})}_{L^2(\Gah)}) + ch^{k_g}\norm{\vhtilde-\bfu^{-\ell}}_{H^{1}(\Gah)}\\
        &\ + ch^{k_g+1}\norm{\bfu}_{H^2(\Ga)}(\norm{\pi}_{H^1(\Ga)} +\norm{\mu}_{H^1(\Ga)})\\
        &\leq c(h^{k_g+\widehat{m}}+h^{k_g+1})(\norm{\bfu}_{H^{k_u+1}(\Ga)}+\norm{p}_{H^{k_{pr}+1}(\Ga)} + \norm{\lambda}_{H^{k_{\lambda}+1}(\Ga)})(\norm{\pi}_{H^1(\Ga)} +\norm{\mu}_{H^1(\Ga)}).
    \end{aligned}
\end{equation}
\item[Term \eqref{eq: btil(eu,pi)}: ] Now combining the new results \eqref{eq: kl inside L2 b 1}, \eqref{eq: kl inside L2 b 2} we obtain the improved bound
\begin{equation}\label{eq: btil(eu,pi) new}
\begin{aligned}
    &b^L(\eu,\{\pi,\mu\}) \leq ch^{\widehat{m}+1} (\norm{\bfu}_{H^{k_u+1}(\Ga)}+\norm{p}_{H^{k_{pr}+1}(\Ga)} + \norm{\lambda}_{H^{k_{\lambda}+1}(\Ga)})(\norm{\pi}_{H^1(\Ga)}+ \norm{\mu}_{H^1(\Ga)}).
    \end{aligned}
\end{equation}
\end{description}

Finally, we gather the improved bounds \eqref{eq: al(w,eu) new}, and \eqref{eq: btil(eu,pi) new} in \eqref{eq: Nitsche type equation lagrange} of \cref{Theorem: L^2 error velocity Lagrange} and coupled with \eqref{eq: regularity estimate lagrange inside} we finally get our assertion, following calculations as in \cref{Theorem: L^2 error velocity Lagrange}.
\end{proof}

\begin{remark}\label{remark: about k_g=1 kl=ku}
    As mentioned in the above proof of \cref{Theorem: improved L^2 error velocity Lagrange}, these bounds now also hold for $k_g=1$, as opposed to \cref{Theorem: L^2 error velocity Lagrange} in the case $k_\lambda = k_u-1$.
\end{remark}
\noindent Finally, we present error bounds for the $L^2$-norm of the normal part of the velocity. We see that our result is suboptimal by half an order compared to \cref{Theorem: improved L^2 error velocity Lagrange}. This may be improved if certain changes occur in our discrete scheme \eqref{weak lagrange discrete}; see \cref{remark: About the L^2 Error of the normal velocity}.
\begin{theorem}[$L^2$ Error estimate for the normal velocity]\label{Theorem: Error estimate for the normal velocity improved}
Let $ k_u = k_\lambda \geq 2$, $k_{pr}=k_u-1$. Also, let $(\bfu,\{p,\lambda\}) \in \bfH^1(\Ga) \times L^2_0(\Ga)\times L^2(\Ga)$ be the solution of the continuous problem \eqref{weak lagrange hom}, and let $(\uh,$ $\{\ph,\lambda_h\}) \in \bfV_h \times Q_h \times \Lambda_h$ be the solution to the discrete scheme \eqref{weak lagrange discrete}. Then, for sufficiently small $h$ with  $\widehat{m}= min\{{k_u},\, {k_{pr}+1}, \, {k_{\lambda}+1},\,  {k_g}\}$,  we have the following error estimate
\begin{equation}
    \begin{aligned}\label{eq: Error estimate for the normal velocity 1}
         \norm{(\bfu^{-\ell} - \uh)\cdot\bfn}_{L^2(\Gah)}\leq  ch^{\widehat{m}+1/2}(\norm{\bfu}_{H^{k_u+1}(\Ga)}&+\norm{p}_{H^{k_{pr}+1}(\Ga)}+ \norm{\lambda}_{H^{k_{\lambda}+1}(\Ga)}+ \norm{\bff}_{L^2(\Ga)}).
    \end{aligned}
\end{equation}
\end{theorem}
\begin{proof}
To address the error of the normal component of the velocity $\norm{\eu\cdot\bfn}_{L^2(\Gah)} = \norm{(\bfu^{-\ell} - \uh)\cdot\bfn}_{L^2(\Gah)}$, we add and subtract the interpolation error $\Ihz(\bfu^{-\ell})$, then using the interpolation bound \eqref{eq: Scott-Zhang interpolant}, the estimates in \cref{Corrolary: Improved Energy Error Estimate for Lagrange Formulation kl=ku} and remembering that $\norm{\bfn-\nh}_{L^{\infty}} \leq ch^{k_g}$ by \eqref{eq: geometric errors 1} we obtain
\begin{equation}
    \begin{aligned}\label{eq: initial eun kl=ku}
        \norm{\eu\cdot\bfn}_{L^2(\Gah)} &\leq \norm{\eu\cdot \nh}_{L^2(\Gah)} + ch^{k_g}\norm{\eu}_{L^2(\Gah)}\\
        &\leq Ch^{k_g+\widehat{m}} + \underbrace{\norm{(\bfu^{-\ell} - \Ihz(\bfu^{-\ell}))\cdot \nh}_{L^2(\Gah)}}_{\leq ch^{k_u+1}} +  \underbrace{\norm{(\Ihz(\bfu^{-\ell})-\uh)\cdot \nh}}_{ =: \norm{\theta_{\bfu}\cdot \nh}_{L^2(\Gah)}}
    \end{aligned}
\end{equation}
where $\widehat{m} =  min\{{k_u},\, {k_{pr}+1},\, {k_{\lambda}}+1,\, {k_g}\}$ and $C = \norm{\bfu}_{H^{k_u+1}(\Ga)}+\norm{p}_{H^{k_{pr}+1}(\Ga)} + \norm{\lambda}_{H^{k_{\lambda}+1}(\Ga)}+ \norm{\bff}_{L^2(\Ga)}$. Let us bound the last term $\norm{(\Ihz(\bfu^{-\ell})-\uh)\cdot \nh}_{L^2(\Gah)} = \norm{\theta_\bfu\cdot \nh}_{L^2(\Gah)}$:

% \begin{equation}
    \begin{align}\label{eq: main thbfu nh}
     &\norm{\thbfu\cdot\nh}_{L^2(\Gah)}^2  = \int_{\Gah}(\theta_\bfu\cdot \nh)(\theta_\bfu\cdot \bfn)\dsh +  \int_{\Gah}(\theta_\bfu\cdot \nh)(\theta_\bfu\cdot (\nh-\bfn))\dsh \nonumber \\
     &\leq   \int_{\Gah} (\thbfu\cdot\nh) \Ihz(\thbfu\cdot\bfn)\, \dsh + \int_{\Gah}(\thbfu\cdot\nh)\big(\Ihz(\thbfu\cdot\bfn) - \thbfu\cdot\bfn\big)\, \dsh + ch^{k_g}\norm{\thbfu\cdot\nh}_{L^2(\Gah)}\norm{\thbfu}_{L^2(\Gah)} \nonumber \\
     % &\ + ch^{k_g}\norm{\thbfu\cdot\nh}_{L^2(\Gah)}\norm{\thbfu}_{L^2(\Gah)}\\
     &\leq \int_{\Gah} (\Ihz(\bfu^{-\ell})\cdot\nh) \Ihz(\thbfu\cdot\bfn)\, \dsh + ch\norm{\thbfu\cdot\nh}_{L^2(\Gah)}\norm{\thbfu}_{L^2(\Gah)}.
    \end{align}
% \end{equation}
where in the last inequality we have used the super-approximation property \eqref{eq: super-approximation estimate 2} and the fact that 
$\int_{\Gah} (\uh\cdot \nh)\Ihz(\theta_\bfu\cdot \bfn)\,\dsh =0$ with $\Ihz(\cdot) \in S^{k_u}_{h,k_g} $,  since $\uh \in \bfV_h^{div}$ and therefore $\int_{\Gah}\uh\cdot\nh \xi_h \, \dsh=0$ for any $\xi_h \in \Lambda_h = S^{k_\lambda}_{h,k_g} = S^{k_u}_{h,k_g}$.
Now, with the help of the geometric errors \eqref{eq: geometric errors 1}, $\bfu \cdot \bfn =0$, the $L^2$-stability \eqref{eq: super-approximation stability} and the non-standard geometric estimate \eqref{eq: high Pnh estimate}
% \begin{equation}
    \begin{align}\label{eq: main thbfu nh 2}
       &\int_{\Gah} (\Ihz(\bfu^{-\ell})\cdot\nh) \Ihz(\thbfu\cdot\bfn) \,\dsh  = \int_{\Gah} \big((\Ihz(\bfu^{-\ell})-\bfu^{-\ell})\cdot\nh\big) \Ihz(\thbfu\cdot\bfn)\,\dsh + \int_{\Gah} (\bfu^{-\ell}\cdot\nh) \Ihz(\thbfu\cdot\bfn)\,\dsh \nonumber \\
       &\leq ch^{k_u+1}\norm{\bfu}_{H^{k_u+1}(\Ga)} \norm{\thbfu\cdot\bfn}_{L^2(\Gah)}+ch^{k_g+1}\norm{\bfu}_{H^2(\Ga)}\norm{\thbfu}_{H^1(\Gah)} \\
       &\leq ch^{k_u+1}\norm{\bfu}_{H^{k_u+1}(\Ga)} \norm{\thbfu\cdot\nh}_{L^2(\Gah)}+ch^{k_g+1}\norm{\bfu}_{H^2(\Ga)}\big( \norm{\bfu^{-\ell} - \Ihz(\bfu^{-\ell})}_{H^1(\Gah)} + \norm{\bfu^{-\ell} - \uh}_{H^1(\Gah)}\big). \nonumber
    \end{align}
% \end{equation}
Combining \eqref{eq: main thbfu nh} and \eqref{eq: main thbfu nh 2}, using \cref{Theorem: Improved Energy Error Estimate for Lagrange Formulation kl=ku} and the $H^1$ velocity error estimate \eqref{eq: Improved H1 Error Estimate for Lagrange Formulation kl=ku new} we see that, after applying young's inequality and a kickback argument
\begin{equation}
\begin{aligned}
        \norm{\thbfu\cdot\nh}_{L^2(\Gah)}^2 &\leq C(h^{k_u+k_g+1} + h^{\widehat{m}+k_g+1}) \leq Ch^{2\widehat{m}+1}.
    \end{aligned}
\end{equation}
Substituting this back to \eqref{eq: initial eun kl=ku} we  get our desired result.
\end{proof}
\begin{remark}[About the $L^2$ Error of the normal velocity]\label{remark: About the L^2 Error of the normal velocity}
    In the above proof of \cref{Theorem: Error estimate for the normal velocity improved}, it is not difficult to see that if instead an improved normal $\nhtil$ was used such that $\norm{\nhtil - \bfn}_{L^{\infty}}\leq ch^{k_g+1}$, in the discretization of our bilinear form $\bhtil$ \eqref{eq: lagrange discrete bilinear forms 22}, then we could have obtained optimal $\bigo(h^{\widehat{m}+1})$ bounds, as in \cref{Theorem: improved L^2 error velocity Lagrange}. This happens because then in the first line of \eqref{eq: main thbfu nh 2} we would instead have $\int_{\Gah} (\bfu\cdot\nhtil) \Ihz(\thbfu\cdot\bfn)\,\dsh \leq ch^{k_g+1}\norm{\bfu}_{L^2(\Ga)}\norm{\thbfu\cdot\bfn}_{L^2(\Gah)}$.
    Notice that replacing $\nh$ by $\nhtil$ in our bilinear $\bhtil$ would not have an impact on the rest of the result.    
    For construction of such a normal see \cite{hansbo2020analysis}. 
\end{remark}

\section{Numerical results}\label{Section: Numerical results}
In this section we present numerical results confirming the proved orders of convergence for the Lagrange multiplier method {\bf(L.M.)}. Comparisons with the penalty method, {\bf(P.M.)}, analyzed in \cite{reusken2024analysis} are made. The numerical results were implemented using the \emph{Firedrake} package \cite{FiredrakeUserManual}, where the linear systems were solved with a direct solver. The finite element approximations are carried out on sequences of mesh-refinements  in order to calculate experimental orders of convergence. 

\subsection{Simple sphere}\label{section: simple sphere experiment}

We consider a  unit sphere $\mathcal{S}$ with the exact  smooth solutions of \eqref{eq: generalized Lagrange surface stokes}
\begin{equation}
    \bfu = \bfPg(-x_3^2,x_2,x_1)^T, \quad p = x_1x_2^2 + x_3
\end{equation}
  satisfying $p \in L^2_0(\Ga)$. Formulae for the right-hand side $\bff$ and the Lagrange multiplier $\lambda$ can be calculated with the help of the exact solution above. This example was also computed  in \cite{Olshanskii2018} using the {\it trace} finite element method.
  \begin{figure}[h]
\centering
  \subfloat[Velocity $\bfu$]{\includegraphics[width=0.5\textwidth]{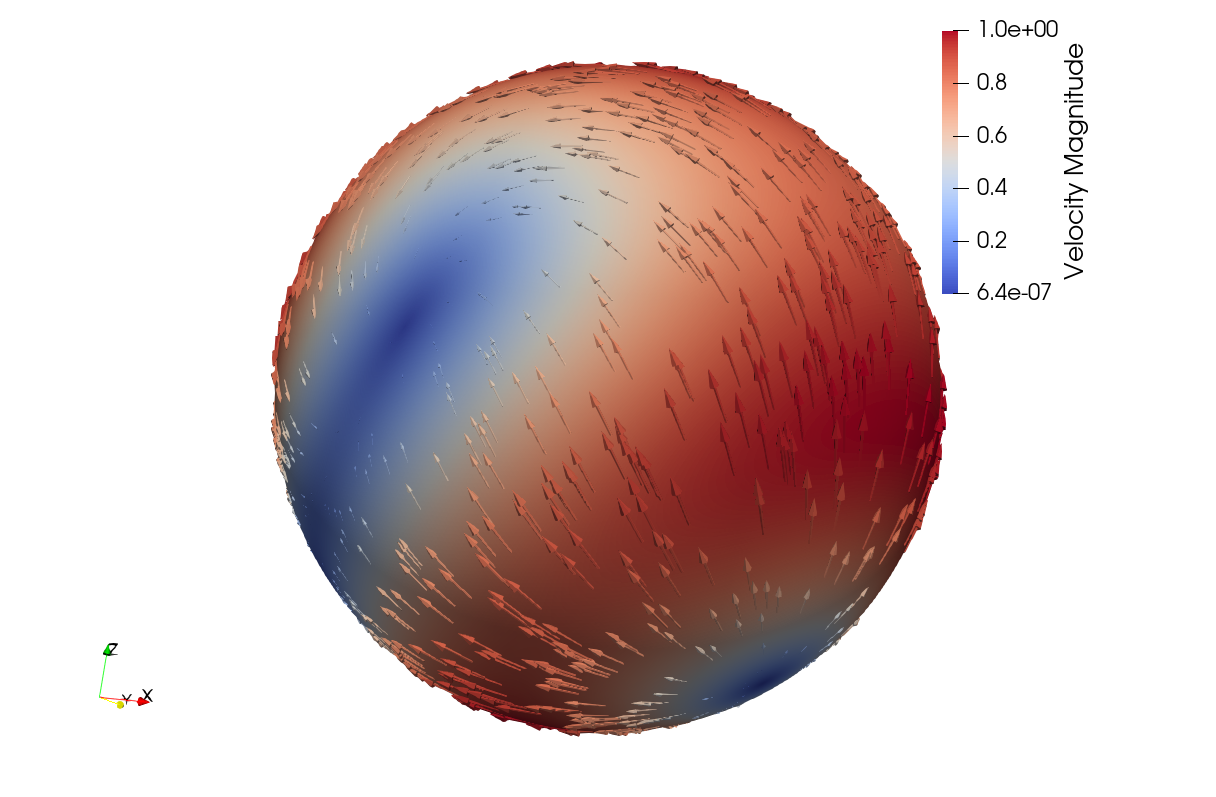}}
  \subfloat[Pressure $p$]{\includegraphics[width=0.5\textwidth]{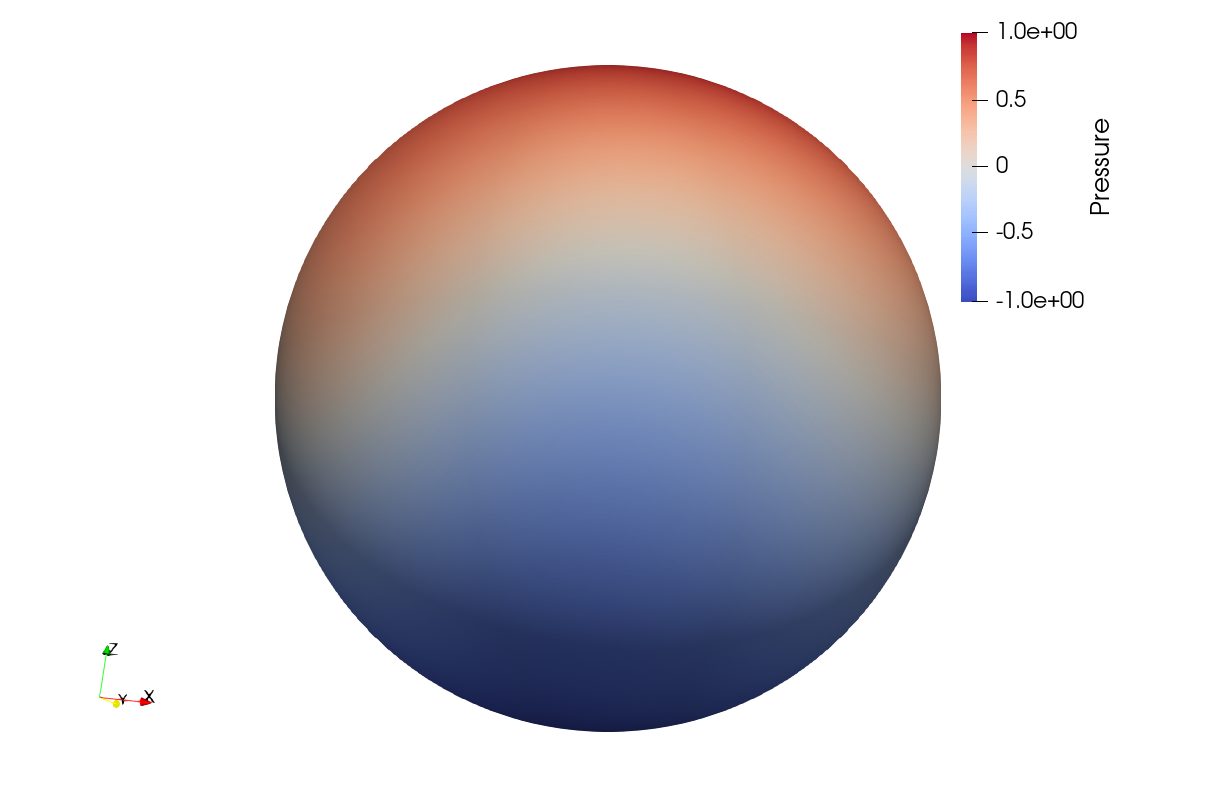}}
  \caption{Sphere | Velocity (a) and pressure (b) solutions}
  \label{fig: Sphere velocity pressure}
\end{figure}

The computations use $\bm{\mathcal{P}}_{3}/\mathcal{P}_{{2}}/\mathcal{P}_{{2}}$ \emph{Taylor-Hood} finite elements and two different orders of approximation of the surface $\Gah$. One is the  iso-parametric surface approximation such that $k_g=k_u=3$ and the other is  super-parametric with $k_g=k_u+1=4$. The rest of the orders are $k_{\lambda} = k_{pr} = k_u-1= 2$. In \cref{fig: Sphere u-Pu-p} the $L^2$-errors of the full velocity, tangential velocity, and pressure are presented. We notice that the order of convergence of both the full and tangential velocity is reduced by one order for the iso-parametric case compared to the super-parametric one, in agreement with \cref{Theorem: L^2 error velocity Lagrange}. However, the pressure (which actually exhibits better than expected $\bigo(h^{3.5})$ convergence) does not appear to be affected. This seems to be specific for this example, as a lower-order convergence has also been observed for the iso-parametric case also in \cite{fries2018higher}. 

\begin{figure}[h]
    \centering
    \includegraphics[width = \textwidth]{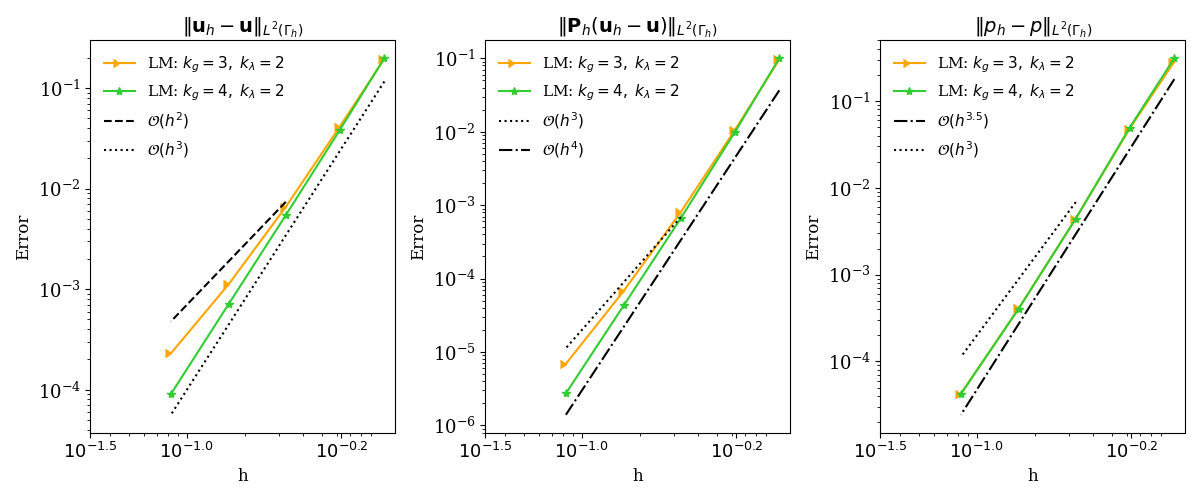}
    \caption{Sphere | Velocity-pressure $L^2$-Errors | \textcolor{Orange}{L.M.}: $\{k_g=3,k_u=3,k_{pr}=2,k_{\lambda}=2\}$, \textcolor{LimeGreen}{L.M.}: $\{k_g=4,k_u=3,k_{pr}=2,k_{\lambda}=2\}$.}
    \label{fig: Sphere u-Pu-p}
\end{figure}

We notice similar characteristic for the error in derivatives in  \cref{fig: Sphere Du-nu}. Letting $k_g=k_u=3$ we obtain a lower order of convergence $\bigo(h^2)$ (sub-optimal), along with worse errors in general, for the covariant derivatives, compared to when $k_g= k_u+1 = 4$, where we observe $\bigo(h^3)$ (optimal) convergence. Thus, it seems the use of super-parametric finite elements is necessary for optimal order of convergence as indicated by \cref{Theorem: Energy Error Estimate for Lagrange Formulation}.  Due to the $H^1$ coercivity bound \eqref{coercivity and Korn's inequality Lagrange}, that is $\norm{\uh}_{H^1(\Gah)} \leq ch^{-1}\norm{\uh}_{\ah},$
the error in the directional derivative $\norm{\nabla_{\Gah}^{dir}(\uh - \bfu^{-\ell})}_{L^2(\Gah)}=\norm{\nbgh(\uh - \bfu^{-\ell})}_{L^2(\Gah)}$  is one order less, i.e. $\bigo(h^2)$. This seems to indicate that this factor $h^{-1}$ is sharp in the discrete \emph{Korn's inequality} \eqref{discrete Korn's inequality T nh} in the case of Taylor-Hood elements.

Finally in the third figure in \cref{fig: Sphere Du-nu} we again notice that $\norm{\uh \cdot \nh}$ has a smaller absolute error, in the case of super-parametric elements $k_g = k_u+1 = 4$ and thus the tangential condition is enforced better.
\begin{figure}[h]
    \centering
    \includegraphics[width = \textwidth]{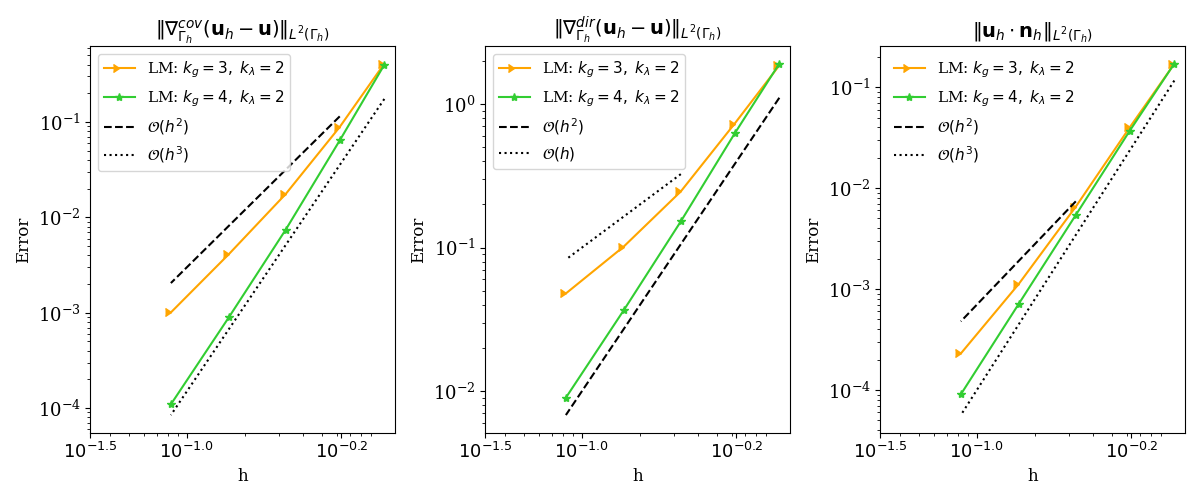}
    \caption{Sphere | Velocity  $\nbgcovh$-Errors (Left), $\nbgh$-Errors (Middle), Normal-Errors (Right) 
    | \textcolor{Orange}{L.M.}: $\{k_g=3,k_u=3,k_{pr}=2,k_{\lambda}=2\}$, \textcolor{LimeGreen}{L.M.}: $\{k_g=4,k_u=3,k_{pr}=2,k_{\lambda}=2\}$.}
    \label{fig: Sphere Du-nu}
\end{figure}

\subsection{Biconcave surface }\label{sec: bioconcave example}

\begin{figure}[h]
  \centering
  \subfloat[Velocity $\bfu$]{\includegraphics[width=0.45\textwidth]{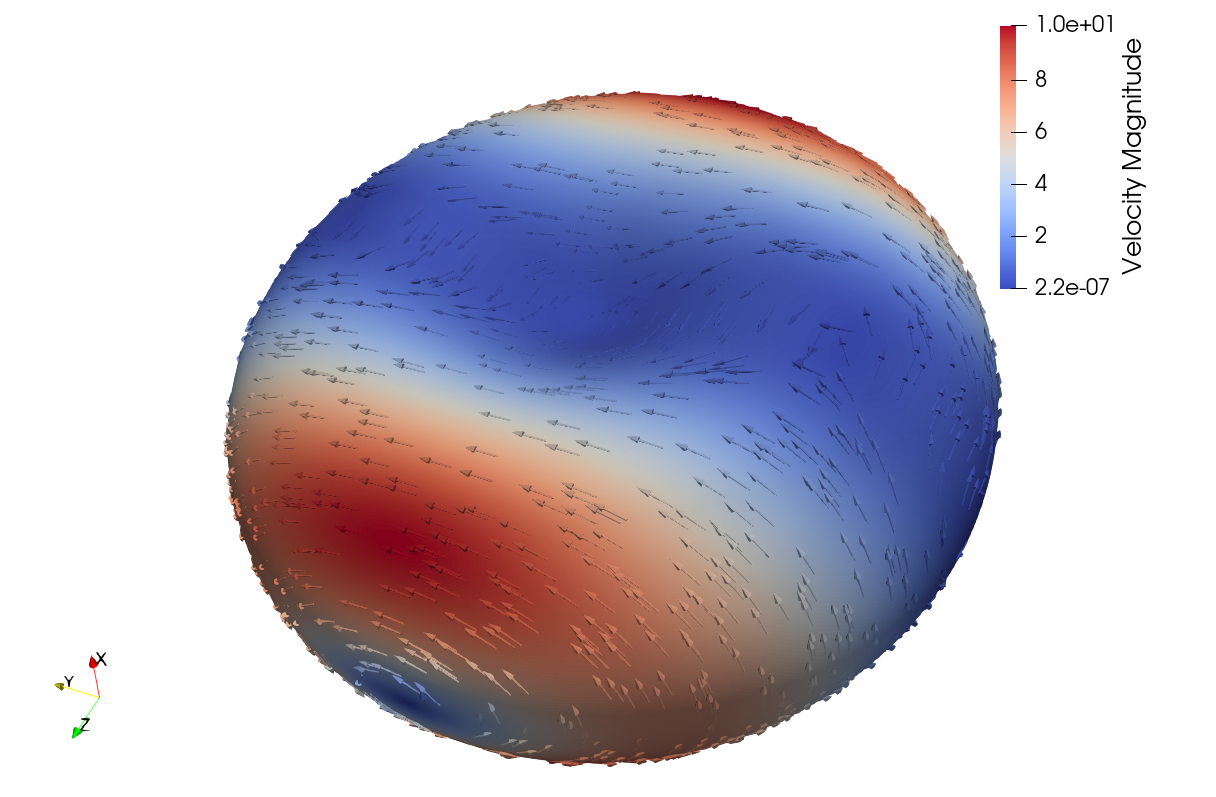}}
  \subfloat[Pressure $p$]{\includegraphics[width=0.45\textwidth]{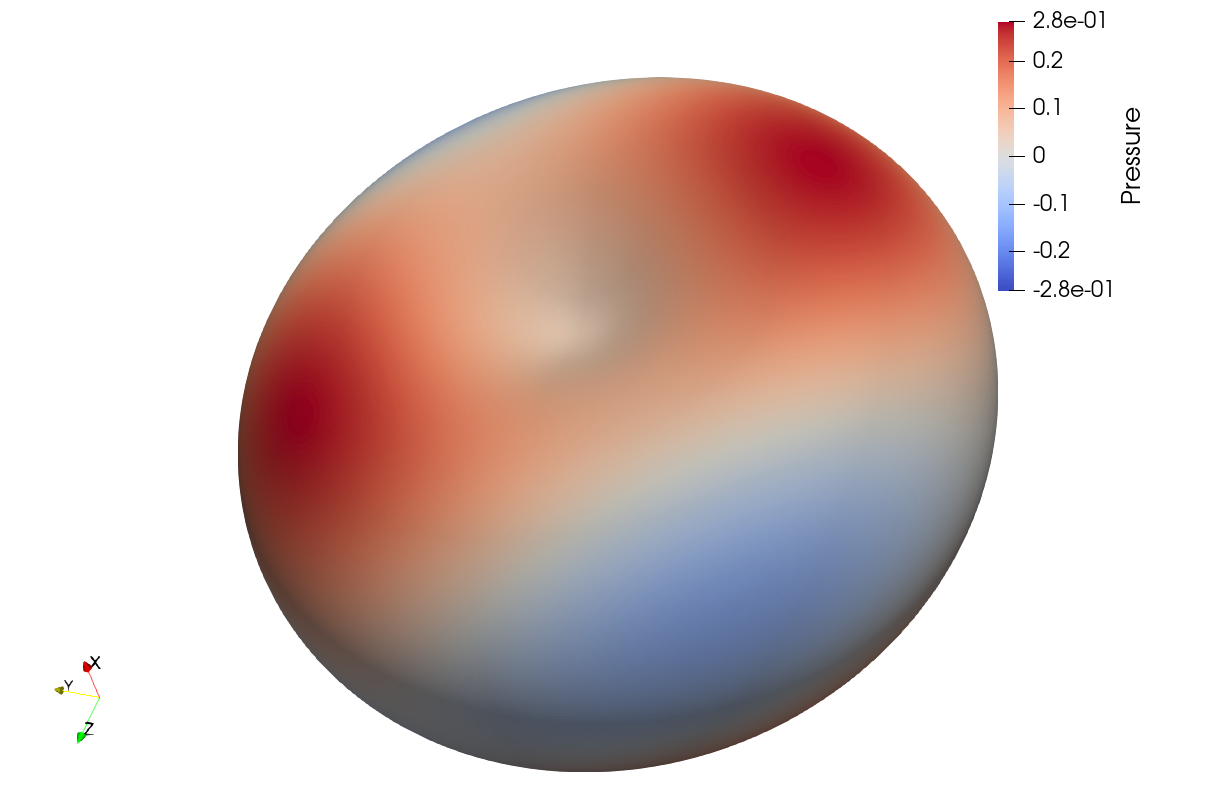}}
  \caption{L.M.  Biconcave Surface | Velocity and pressure solutions}
  \label{fig: Bio velocity pressure}
\end{figure}
Here we compare the Lagrange multiplier method ({\bf L.M.}) and the Penalty method ({\bf P.M.}) on  \cite[Example 5.1]{brandner2022finite} concerning the  biconcave shape surface $\Ga$   defined by the zero-level  set of the following function
\begin{equation}
     \phi(x) = (d^2 + x_1^2 + x_2^2 + x_3^2 )^3 - 8d^2(x_2^2+x_3^2) - c^4 =0.
\end{equation}
In our case $d=0.91$ and $c=0.95$. Since $\phi(x)$ is a level set and not a distance function, we calculated the closest-point
projection of the surface iteratively using  Newton's method. As in \cite{brandner2022finite} the solution of the Stokes problem \eqref{eq: generalized tangential surface stokes} and \eqref{eq: generalized Lagrange surface stokes} is set  to be
\begin{equation}
    \bfu = \textbf{curl}_{\Ga}\psi,\ \text{ with } \psi = x_1^2x_2 - 5x_3^3.
\end{equation}
From \cite{reusken2018stream} the surface curl is tangential and thus the solution is also clearly tangential but also divergence free, while $p \in L^2_0(\Ga)$. The data on the right hand side of \eqref{eq: generalized tangential surface stokes} i.e. $\bff$, $g$ and the Lagrange multiplier $\lambda$ can be evaluated with the help of the exact solution.

\begin{figure}[h]
    \centering
    \includegraphics[width = 0.9\textwidth]{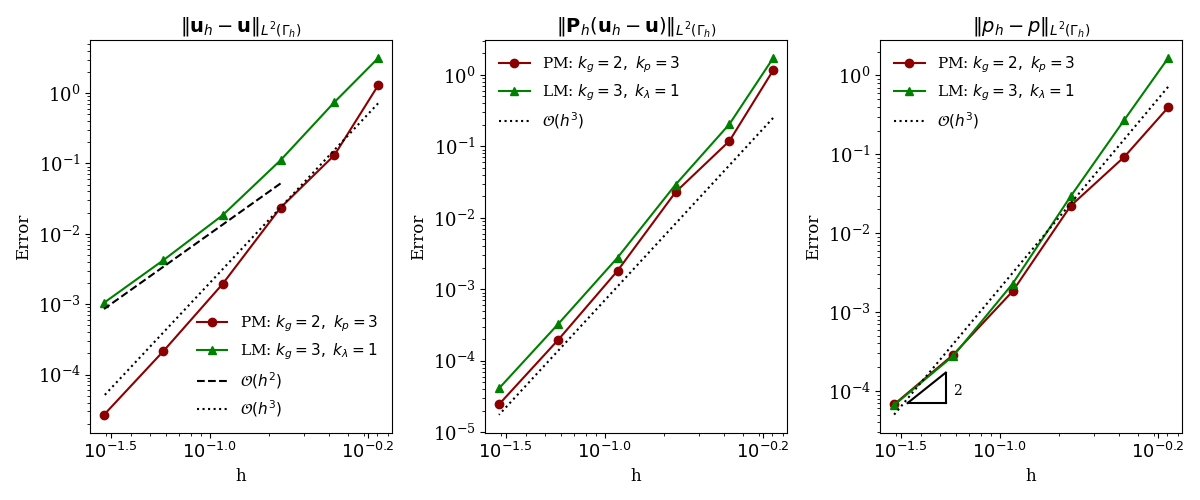}
    \caption{Biconcave Surface | Velocity-pressure $L^2$-Errors | \textcolor{BrickRed}{P.M.}: $\{k_g=2,k_u=2,k_{pr}=1,k_p=3\}$, \textcolor{OliveGreen}{L.M.}: $\{k_g=3,k_u=2,k_{pr}=1,k_{\lambda}=1\}$.}
    \label{fig: Bio u-Pu-p}
\end{figure}

We start with comparing $\bm{\mathcal{P}}_{2}/\mathcal{P}_{{1}}$ \emph{Taylor-Hood} finite elements for the penalty and $\bm{\mathcal{P}}_{2}/\mathcal{P}_{{1}}/\mathcal{P}_{{1}}$ \emph{Taylor-Hood} finite elements for the Lagrange method. We consider iso-parametric elements $k_g$ and an improved normal $\nhtil$ to be of order $k_p = k_g + 1$ in the Penalty method ({\bf P.M.}), whilst using  super-parametric finite elements for the Lagrange multiplier method ({\bf L.M.}).
We observe $\bigo(h^3)$ convergence in the full velocity for ({\bf P.M.}) and worse $\bigo(h^2)$ for ({\bf L.M.}), and optimal convergence in the tangential velocity $\bigo(h^3)$ for both methods. Both observations agree with \cref{Theorem: L^2 error velocity Lagrange}. For the $L^2$ pressure norm, we observed a higher than anticipated  convergence (specifically $\sim \bigo(h^3)$) at the beginning for both methods, but ultimately they tend towards $\bigo(h^2)$ as expected from \cref{Theorem: Energy Error Estimate for Lagrange Formulation}.

% Regarding the $L^2$-error norms of the full, and tangential velocity in \cref{fig: Bio u-Pu-p} we see  $\bigo(h^3)$ convergence in the full velocity for \emph{P.M.} and sub-optimal $\bigo(h^2)$ for \emph{L.M.}, and optimal convergence in the tangential velocity $\bigo(h^3)$ for both methods. Both observations agree with \cref{Theorem: L^2 error velocity Lagrange}. For the $L^2$ pressure norm, we observed a higher than anticipated  convergence (specifically $\sim \bigo(h^3)$) at the beginning for both methods, but ultimately they tend towards $\bigo(h^2)$ as expected from \cref{Theorem: Energy Error Estimate for Lagrange Formulation}.

\begin{figure}[h]
    \centering
    \includegraphics[width = \textwidth]{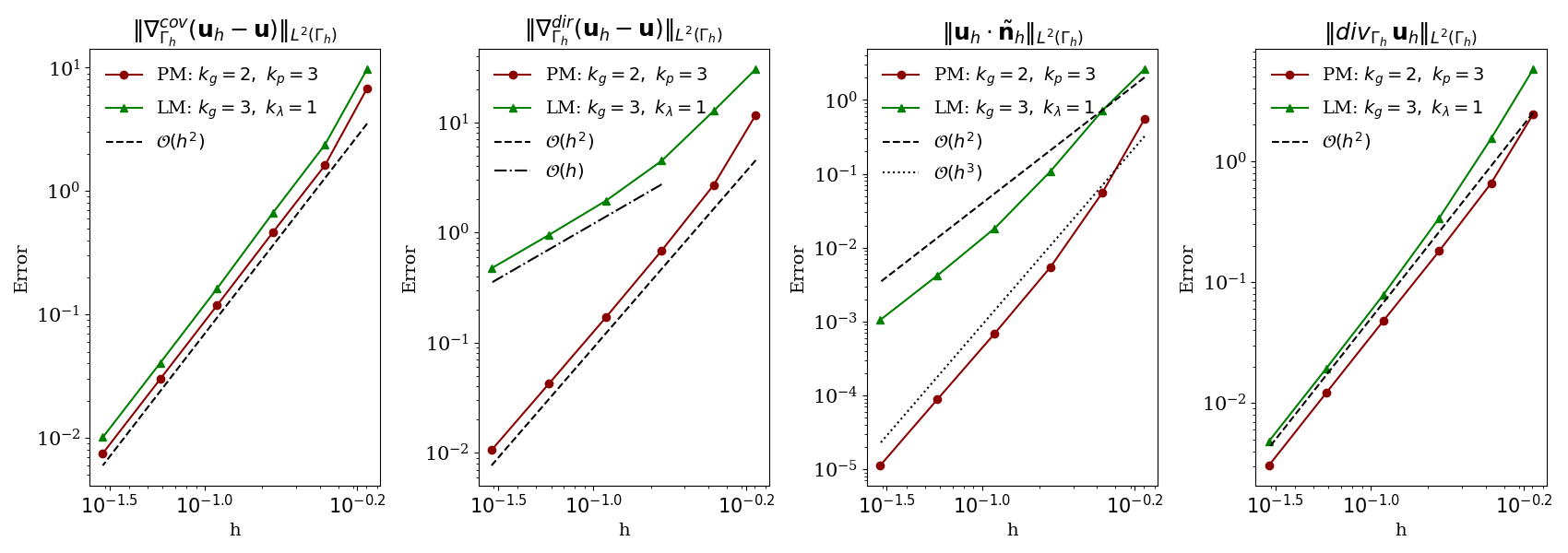}
    \caption{Biconcave Surface | Velocity : $\nbgcovh$-Errors, \ $\nbgh$-Errors, \  Normal-Errors (\textcolor{BrickRed}{P.M.} : $\norm{\uh\cdot\nhtil}_{L^2(\Gah)}$, \textcolor{OliveGreen}{L.M.} : $\norm{\uh\cdot\nh}_{L^2(\Gah)}$), $\divgh$-Errors
    | \textcolor{BrickRed}{P.M.}: $\{k_g=2,k_u=2,k_{pr}=1,k_p=3\}$,  \textcolor{OliveGreen}{L.M.}: $\{k_g=3,k_u=2,k_{pr}=1,k_{\lambda}=1\}$ .}
    \label{fig: Bio Du}
\end{figure}

For errors in derivatives, from the first two figures of \cref{fig: Bio Du} we notice optimal convergence $\bigo(h^2)$ for both methods in  the $\nbgcovh$- norm (or energy norm), while for the directional derivative $\nbgh^{dir}(\cdot) = \nbgh(\cdot)$ we notice an one order decrease $\bigo(h)$ in the Lagrange method ({\bf L.M.}), which is attributed to the
the discrete \emph{Korn's inequality} \eqref{discrete Korn's inequality T nh}  or otherwise the $H^1$ coercivity bound  $\norm{\uh}_{H^1(\Gah)} \leq ch^{-1}\norm{\uh}_{\ah}$ \eqref{coercivity and Korn's inequality Lagrange}.

Now we focus on the last two figures of \cref{fig: Bio Du}, where we notice that the tangential approximation, i.e. normal velocity error, determined by the quantity $\norm{\uh \cdot \nhtil}_{L^2(\Gah)}$ for the Penalty method ({\bf P.M.}) and $\norm{\uh \cdot \nh}_{L^2(\Gah)}$ for the Lagrange method ({\bf L.M.}), exhibits the expected characteristics. That is, higher order $\bigo(h^3)$ for ({\bf P.M.}) compared to the  $\bigo(h^2)$ convergence for ({\bf L.M.}). Finally, we see that both methods adequately approximate the divergence-free characteristic of our solution, with order $\bigo(h^2)$.

\subsection{Varying curvature surface $k_\lambda=k_u$}\label{example: Dz-Ell}
In this example, we consider the case where $k_\lambda=k_u$ for the extra Lagrange multiplier approximation, and test it against our theoretical results in Theorems \ref{Corrolary: Improved Energy Error Estimate for Lagrange Formulation kl=ku} and \ref{Theorem: improved L^2 error velocity Lagrange}. We consider the closed and compact surface $\Ga$ in \cite[Example 4.8]{DziukElliott_acta} which is a the level set function 
\begin{equation*}
    \phi(x) = \frac{1}{4} x_1^2 + x_2^2 + \frac{4 x_3^2}{(1 + \frac{1}{2} \sin(\pi x_1))^2} - 1 \quad x \in \mathbb{R}^3.
\end{equation*}
It is constructed by mapping from a discretised unit sphere $\mathcal{S}$ with the help of the mapping  $F(p) = (2p_0,\, p_1,\, \frac{1}{2}p_2(1 + \frac{1}{2}\sin(2\pi p_0))) $ for $p \in \mathcal{S}$, such that $\Ga = F(\mathcal{S})$. We describe an exact solution of the Stokes problems  \eqref{eq: generalized Lagrange surface stokes} as
\begin{equation}
\begin{aligned}
    \bfu = \textbf{curl}_{\Ga} \psi, \ \text{with} \ \psi =  \frac{1}{2\pi}\cos(2\pi x_1)\cos(2\pi x_2)\cos(2\pi x_3), \quad p = \sin(\pi x_1)\sin(2\pi x_2)\sin(2\pi x_3).
\end{aligned}
\end{equation}

\begin{figure}[h]
    \centering
    \includegraphics[width = 0.5\textwidth]{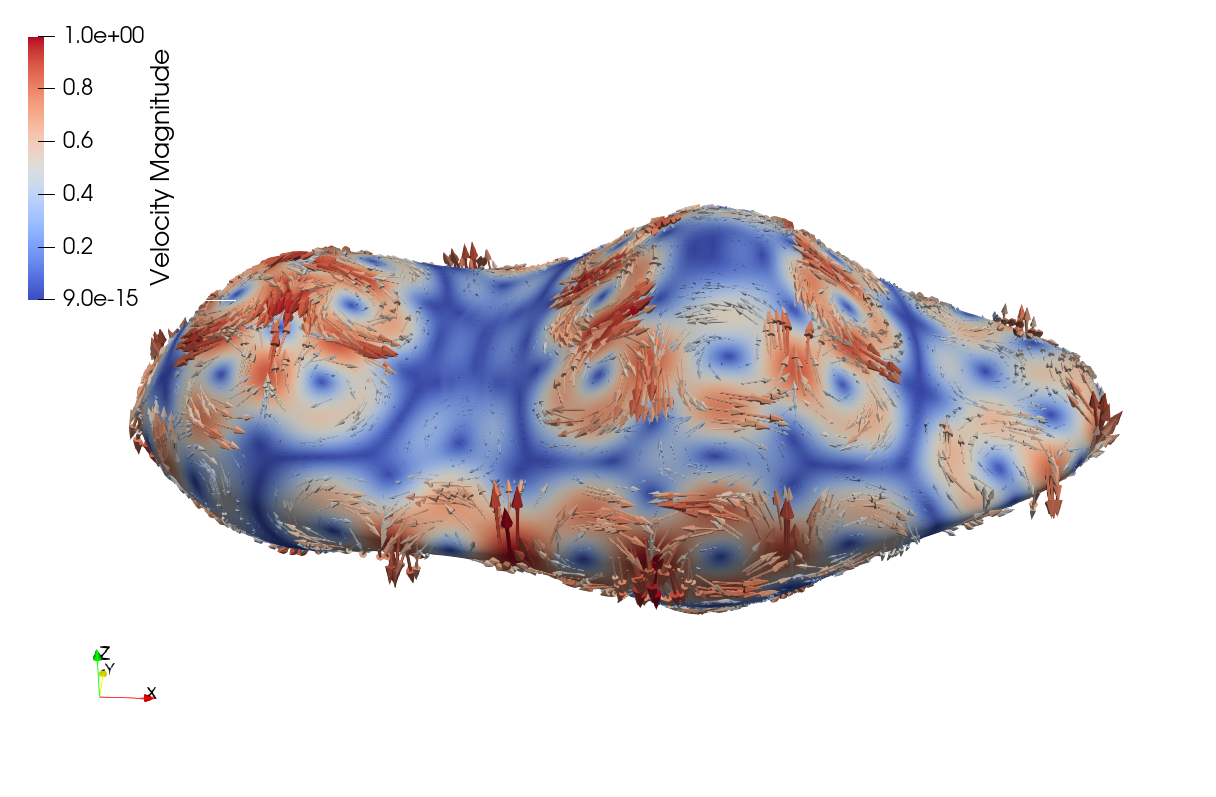}
    \caption{L.M.  Varying Curvature Surface | Velocity $\bfu$}
    \label{fig: Dziuk surface velocity}
\end{figure}

From \cite{reusken2018stream} we know that the surface curl is tangential. This means that the solution is also tangential and divergence-free, while $p \in L^2_0(\Ga)$. The data on the right-hand side, that is, $\bff$, $g$  can be calculated numerically with the help of the exact solution as interpolation of the smooth data. 
%and the Lagrange multiplier $\lambda$ Obviously in our case $g=0$. 

\begin{figure}[h]
    \centering
    \includegraphics[width = \textwidth]{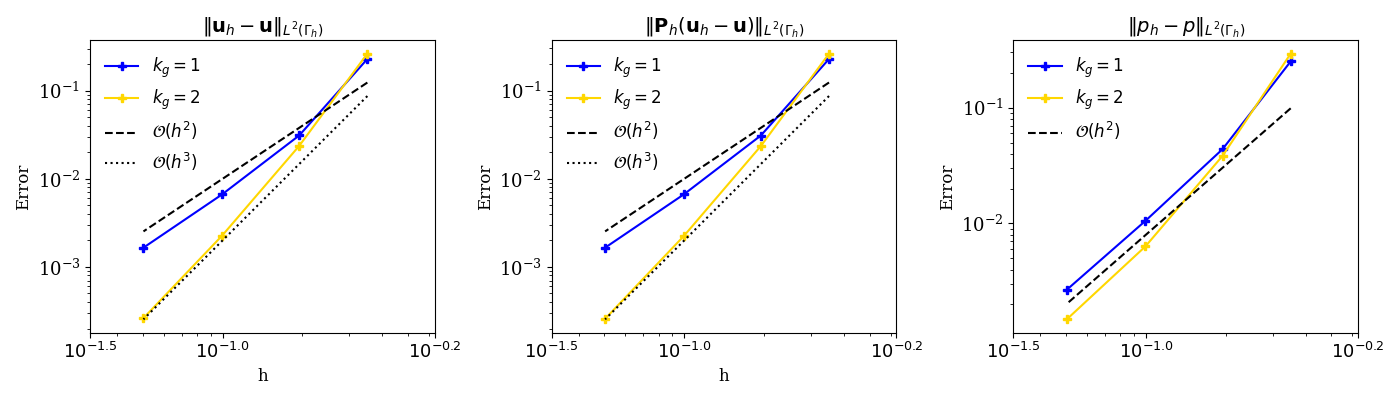}
    \caption{Varying Curvature Surface | Velocity-pressure $L^2$-Errors }
    \label{fig: Dziuk u-pu-p}
\end{figure}
In our example, we use $\bm{\mathcal{P}}_{2}/\mathcal{P}_{{1}}/\mathcal{P}_{{2}}$ extended \emph{Taylor-Hood} finite elements. We address
the convergence order of \emph{two} different surface discretizations, one where we use the linear setting $k_g=1$, which corresponds to a piecewise planar geometric approximation, and another where we use the quadratic surface approximation $k_g=2$.
%i.e. $k_\lambda = k_u=2$ and $k_{pr}=k_u-1=1$.

Let us start with the $k_g = 2$ case. First and foremost, in both Figures \ref{fig: Dziuk u-pu-p} and \ref{fig: Dziuk Du-nu} we notice that most of our theoretical results in \cref{sec: the kl=ku case} agree with the experimental convergence. That is, we have $\bigo(h^{3})$ convergence for the tangential $L^2$ velocity error norm; see \cref{Theorem: improved L^2 error velocity Lagrange}, and $\bigo(h^2)$ convergence in the energy norm for the velocity and $L^2$ for the pressure \cref{Corrolary: Improved Energy Error Estimate for Lagrange Formulation kl=ku}. 
We do notice though, a half an higher than expected convergence for our normal $L^2$ velocity error norm, and therefore the full $L^2$ norm, $\bigo(h^{k_g+1}) = \bigo(h^{3})$ compared to $\bigo(h^{k_g})$ predicted by \cref{Theorem: Error estimate for the normal velocity improved}, despite not using an improved normal $\nhtil$ in our discretization scheme; \cref{remark: About the L^2 Error of the normal velocity}.
 
For the planar case $k_g=1$, we see in Figure \ref{fig: Dziuk u-pu-p}, that our theoretical result for the tangential $L^2$ error \cref{Theorem: improved L^2 error velocity Lagrange} agrees with our experimental results $\bigo(h^2)$. On the other hand, in both Figures \ref{fig: Dziuk u-pu-p}, \ref{fig: Dziuk Du-nu} we notice higher than expected convergence rate ($\bigo(h^2)$) in several of our experimental quantities, as opposed to the expected $\bigo(h^{k_g})=\bigo(h)$ convergence rate from \cref{Corrolary: Improved Energy Error Estimate for Lagrange Formulation kl=ku,Theorem: Error estimate for the normal velocity improved}. Notice in particular that the $L^2$ error norm for the normal part of the velocity in Figure \ref{fig: Dziuk Du-nu} is one and a half orders higher than expected, $\bigo(h^3)$, which also explains the $\bigo(h^2)$ convergence for the full $L^2$ velocity norm; see \cref{Theorem: Error estimate for the normal velocity improved}.

Finally, we make a small note regarding the choice $k_\lambda$. We see that letting $k_\lambda= k_u$ we obtain optimal convergence even when using \emph{iso-parametric finite elements} $k_g=k_u=2$, compared to the other case; see \cref{section: simple sphere experiment}. We even obtain good results when considering planar triangulation, i.e. $k_g=1$. Despite the good results, 
for this choice the condition of the corresponding systems of equations scales much worse w.r.t. the mesh parameter $h$; for more details see \cite{fries2018higher}.

\begin{figure}[h]
    \centering
    \includegraphics[width = \textwidth]{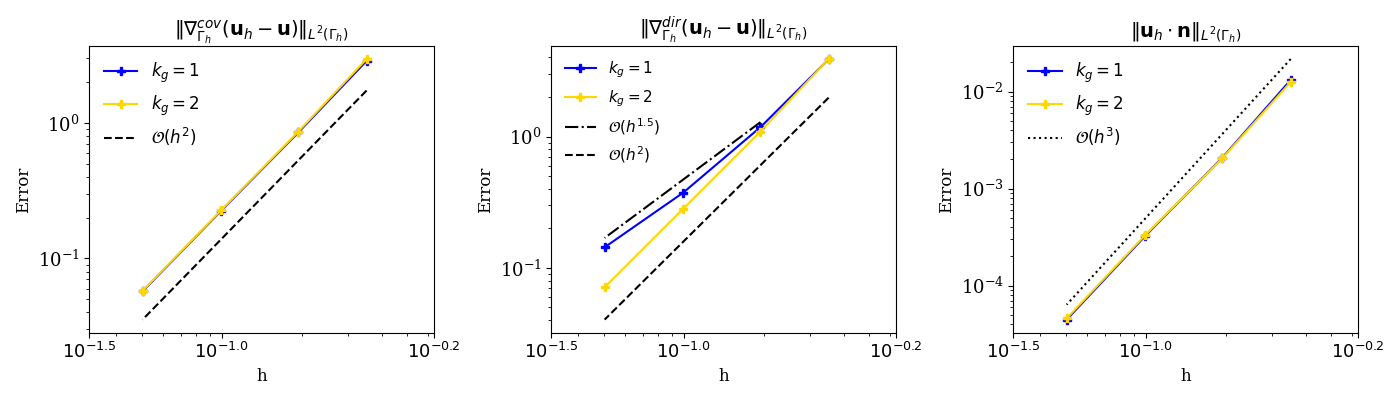}
    \caption{Varying Curvature Surface | Velocity:  $\nbgcovh$-Errors (Left), $\nbgh$-Errors (Middle), Normal-Error (Right).}
    \label{fig: Dziuk Du-nu}
\end{figure}

\noindent \textbf{Acknowledgments.} \ \ The authors extend their gratitude to Tom Ranner for helpful discussions regarding the numerical computations and to Arnold Reusken for the fruitful conversations. We also thank the anonymous referees for their comments.

\bibliographystyle{siam}
\bibliography{reference}
\end{document}